\documentclass{amsart}
\usepackage{amssymb,amscd}
\usepackage{graphicx,subfigure,psfrag}
\usepackage[all]{xy}

\newtheorem*{thma}{Theorem A}
\newtheorem*{thmb}{Theorem B}
\newtheorem{thm}{Theorem}[section]
\newtheorem{pro}[thm]{Proposition}
\newtheorem{lem}[thm]{Lemma}
\newtheorem{cor}[thm]{Corollary}

\newtheorem{exa}[thm]{Example}
\newtheorem{con}[thm]{Conjecture}

\newcommand{\im}{\operatorname{im}}
\newcommand{\rk}{\operatorname{rank}}
\newcommand{\cond}{\operatorname{Cond}}

\begin{document}

\title{Graph braid groups and right-angled Artin groups}
\author{Jee Hyoun Kim}
\author{Ki Hyoung Ko}
\author{Hyo Won Park}
\address{Department of Mathematics, Korea Advanced Institute
of Science and Technology, Daejeon, 305-701, Korea}
\email{\{kimjeehyoun,knot,H.W.Park\}@kaist.ac.kr}
\thanks{This work was supported by the Korea Science and Engineering
Foundation(KOSEF) grant funded by the Korea government(MOST) (No.
R01-2006-000-10152-0)}
\begin{abstract}
We give a necessary and sufficient condition for a graph to have a
right-angled Artin group as its braid group for braid index $\ge 5$.
In order to have the necessity part, graphs are organized into small
classes so that one of homological or cohomological characteristics
of right-angled Artin groups can be applied. Finally we show that a
given graph is planar iff the first homology of its 2-braid group is
torsion-free and leave the corresponding statement for $n$-braid
groups as a conjecture along with few other conjectures about graphs
whose braid groups of index $\le 4$ are right-angled Artin groups.
\end{abstract}

\maketitle

\section{Introduction}
As motivated by robotics, graph braid groups were first introduced
by Ghrist \cite{Ghr} in 1998. A motion of $n$ robots is a
one-parameter family of $n$-tuples in a graph $\Gamma$ that are
pairwise distinct, that is, a path in the configuration space of
$\Gamma$ which is the complement of diagonals in the $n$-fold
product $\Gamma^n$. The graph $n$-braid group over $\Gamma$ is,
roughly speaking, the group of these motions under concatenation,
and is more precisely the fundamental group of the (unordered)
configuration space.

On the other hand, a \emph{right-angled Artin group} is a group that
has a finite presentation each of whose relators is a commutator of
generators. It was sometimes called a {\em graph group} since its
presentation can be defined via a graph whose vertices are
generators and each of whose edges gives a relator that is the
commutator of two ends. The survey article \cite{Ch} by Charney
contains a good overview.

Graph braid groups that can be computed directly by hand usually
have presentations whose relators are all in the form of a
commutator that is not however necessarily a commutator of
generators. In fact, all of graph braid groups were presumed to be
right-angled Artin groups until it was known by Abrams and Ghrist in
\cite{AR} that the pure 2-braid groups of the complete graph $K_5$
and the complete bipartite graph $K_{3,3}$ are surface groups so
they are not right-angled Artin groups. Since then, it was a
reasonable conjecture that every planar-graph braid group is a
right-angled Artin group. Connelly and Doig in \cite{CD} proved that
every linear-tree braid group is a right-angled Artin group where a
tree is linear if it does not contain $T_0$ in Figure~\ref{fig1}(a).
Meanwhile, Crisp and Wiest in \cite{CW2} proved that every graph
braid group embeds in a right-angled Artin group. Conversely,
Sabalka in \cite{Sa} proved that every right-angled Artin group can
be realized as a subgroup of a graph braid group. Farley and Sabalka
in \cite{FS1} described an efficient method to obtain a presentation
for any graph braid group using the discrete Morse theory.

Recently, Farley and Sabalka in \cite{FS2} characterized trees whose
braid groups are right-angled Artin groups. They proved that for
braid indices 2 or 3, tree braid groups are always right-angled Artin
groups, and for braid index $\ge 4$, a braid group over a given tree
is a right-angled Artin group if and only if the tree does not
contain $T_0$. In this article, this result is generalized to
arbitrary graphs for braid index $\ge 5$. In fact, we will prove the
following two Theorems:

\begin{thma}
If a graph $\Gamma$ contains neither $T_0$ nor $S_0$ in
Figure~\ref{fig1}, then the graph $n$-braid group $B_n\Gamma$ is a
right-angled Artin group for braid index $n\ge 5$.
\end{thma}

\begin{thmb}
If a graph $\Gamma$ contains $T_0$ or $S_0$ in Figure~\ref{fig1},
then the graph $n$-braid group $B_n\Gamma$ is not a right-angled
Artin group for braid index $n\ge 5$.
\end{thmb}

\begin{figure}[ht]
\centering\label{fig1}
\subfigure[$T_0$]
{\includegraphics[height=1.6cm]{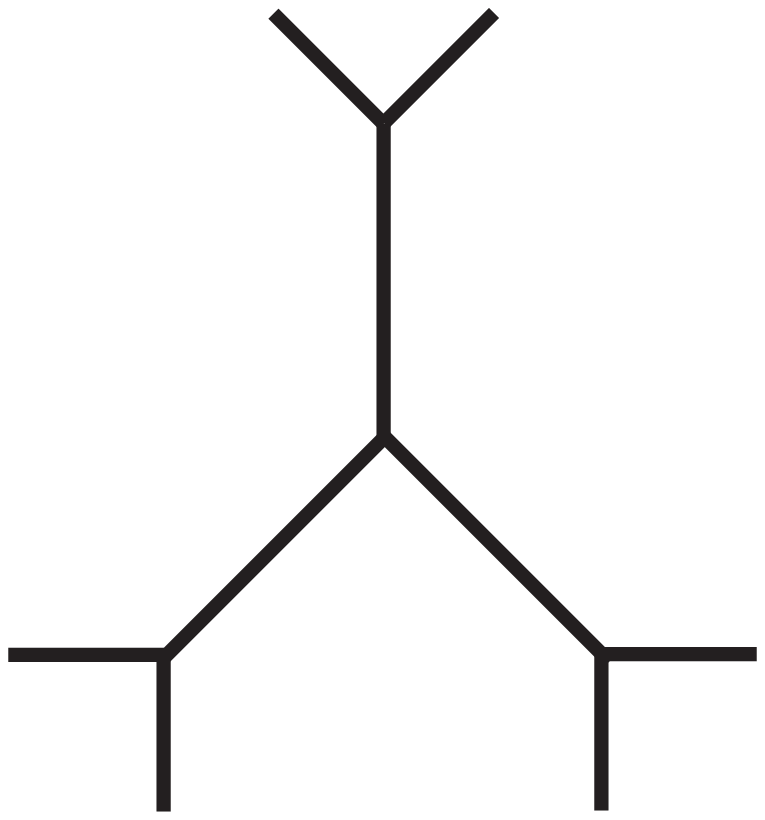}}\qquad
\subfigure[$S_0$]
{\includegraphics[height=1.6cm]{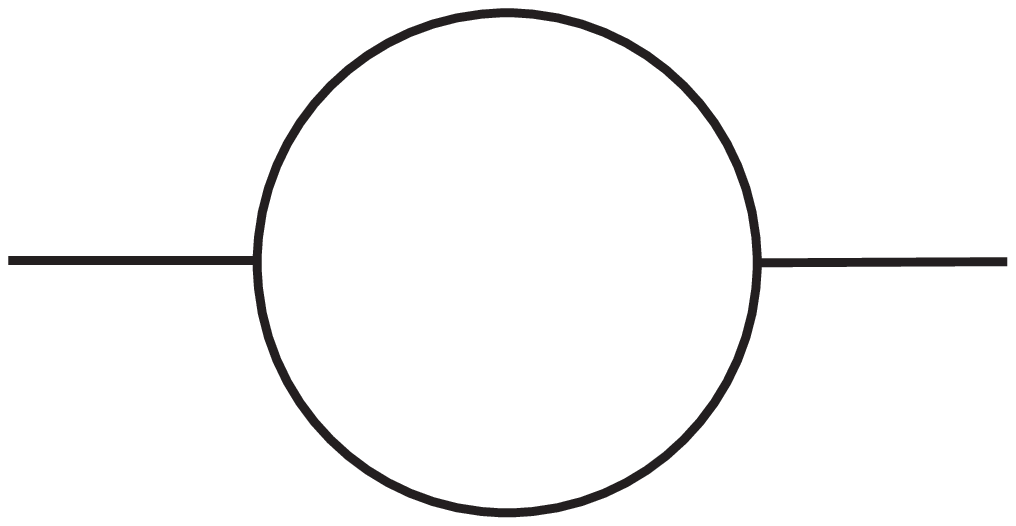}}
\caption{$T_0$ and $S_0$}
\end{figure}

This paper is organized as follows. In \S2, we explain how to
simplify configuration spaces using the discrete Morse theory. Basic
constructions and notations introduced will be used throughout the
article. As applications, we obtain a presentation of $B_n\Gamma$
for any graph $\Gamma$ containing neither $T_0$ nor $S_0$ and prove
Theorem~A.

For graphs $\Gamma$ and $\Gamma'$ and $n\ge n'$, suppose that there
is an embedding $i:\Gamma'\to\Gamma$ of graphs that induces a
surjection $i^*:H^*(B_n\Gamma;\mathbb Z_2)\to
H^*(B_{n'}\Gamma';\mathbb Z_2)$ of cohomology algebras whose kernel
is generated by homogeneous classes of degree 1 and 2, and moreover,
suppose that $B_{n'}\Gamma'$ is not a right-angled Artin group. Then
by combining results by Charney-Davis \cite{CD} and by
Farley-Sabalka \cite{FS2}, one can conclude that $B_n\Gamma$ is not
a right-angled Artin group either. This argument was introduced by
Farley and Sabalka in \cite{FS2} to show that an $n$-braid group
over a tree containing $T_0$ is not a right-angled Artin group for
$n\ge 4$. One of hypotheses of this argument does not hold if there
is no embedding $i:\Gamma'\to\Gamma$ inducing an injection
$i_*:H_1(B_{n'}\Gamma')\to H_1(B_n\Gamma)$. Furthermore, if
$\Gamma'$ has no vertices of valency 1, there is no room to
accommodate extra punctures when we construct an embedding
$i:\Gamma'\to\Gamma$ to obtain some result for $B_n\Gamma$ for all
braid index $n\ge 5$. In \S3, we circumvent these difficulties by dividing
the class of graphs containing $T_0$ but not $S_0$ into smaller
classes, for each of which the above argument is applicable. So
Theorem~B is proved for a graph $\Gamma$ that contains $T_0$ but not
$S_0$.

Given a finite presentation of a group $G$, let $F$ be the free
group over generators and $R$ be the normal subgroup generated by
relators. Suppose $R\subset [F,F]$. Then the inclusion induces a
homomorphism $\Phi_G : R/[F,R]\to [F,F]/[F,[F,F]]$ which can be
regarded as the dual of the cup product $H^1(G)\wedge H^1(G)\to
H^2(G)$ according to Matei-Suciu \cite{MS}. In \S4, we derive and
use a number of necessary conditions \ for $G$ to be a right-angled
Artin group. For example, if $G$ is a right-angled Artin group,
$\Phi_G$ is injective, which we employ for showing that $B_nS_0$ and
$B_n\Theta$ are not a right-angled Artin group where $\Theta$ is the
graph of the figurative shape $\Theta$. If $G$ is a right-angled
Artin group, the cohomology ring $H^*(G)$ is completely determined
by $\Phi_G$ and $H_1(G)$, and this fact also provides a useful
obstruction for $G$ to be a right-angled Artin group. A more
sophisticated application of this idea is to consider homomorphism
$h_*:R'/[F',R']\to R/[F,R]$ and $\bar h:[F',F']/[F',[F',F']]\to
[F,F]/[F,[F,F]]$ induced from an embedding $\Gamma'\to \Gamma$. Here
we assume that $R\subset [F,F]$ and $R'\subset [F',F']$ for
$B_n\Gamma=F/R$ and $B_n\Gamma'=F'/R'$. If we can directly show from
their configuration spaces that $\Phi_{B_n\Gamma'}$ is not injective
and $h_*:H_2(B_n\Gamma')\to H_2(B_n\Gamma)$ is injective, then the
commutativity $\Phi_{B_n\Gamma}h_*=\bar h\Phi_{B_n\Gamma'}$ implies
that $\Phi_{B_n\Gamma}$ is not injective and so $B_n\Gamma$ is not a
right-angled Artin group. Using this argument, we prove Theorem~B
for a planar graph $\Gamma$ containing $S_0$. For non-planar graphs,
we prove Theorem~B by showing that the first homology of its braid
group has a 2-torsion.

We conclude the paper with four conjectures.
Two of them are about extending our main results in this paper to graph
braid groups of braid index $\le 4$. The other two conjectures are
characterizations of planar graphs via graph braid groups.
One is that graph braid groups of planar graphs are commutator-related and the other is that the first homology groups of their graph braid groups are torsion-free. As mentioned above, a given
graph must be planar if the first homology of its graph braid group
is torsion-free. Finally we prove the converse when the braid index is 2.

\section{Discrete configuration spaces and their Morse complexes}
\label{s:two}

\subsection{Discrete configuration spaces}\label{ss21:configuration}
Let $\Gamma$ be a connected finite graph. We may regard $\Gamma$ as
a metric space by treating each edge as the unit interval $[0,1]$.
The {\em topological configuration space} of $\Gamma$ is defined by
$$C_n\Gamma=\{(x_1,\ldots,x_n)\in
        \Gamma^n \mid x_i\ne x_j\mbox{ if } i\neq j \}.$$
The \emph{unordered topological configuration space} of $\Gamma$ is
defined by $UC_n\Gamma=C_n\Gamma/S_n$ where the symmetric group
$S_n$ acts on $C_n\Gamma$ by permuting coordinates, or is
alternatively defined by
   $$UC_n\Gamma=\{\{x_1,\ldots,x_n\}\subset
        \Gamma \mid x_i\ne x_j\mbox{ if } i\neq j \}.$$
The topological configuration spaces are open complexes and it is
hard to make them finite complexes. We study the following alternative.

We now regard a graph $\Gamma$ as a 1-dimensional CW complex. Then
the $n$-fold Cartesian product $\Gamma^n$ of $\Gamma$ has a CW
complex structure and a cell in $\Gamma^n$ has the form
$(\sigma_1,\ldots,\sigma_n)$ where each $\sigma_i$ is either a
vertex or an edge of $\Gamma$. Let $\partial\sigma$ denote the set
of end vertices if $\sigma$ is an edge, or $\sigma$ itself if
$\sigma$ is a vertex. The \emph{ordered discrete configuration
space} of $\Gamma$ is defined by
         $$D_n\Gamma=\{(\sigma_1,\ldots,\sigma_n)\in
        \Gamma^n \mid \partial \sigma_i\cap\partial \sigma_j =
        \emptyset\mbox{ if } i\neq j \}.$$
The \emph{unordered discrete configuration space} of $\Gamma$ is
defined by $UD_n\Gamma=D_n\Gamma/S_n$ , or is alternatively defined
by
   $$UD_n\Gamma=\{\{\sigma_1,\ldots,\sigma_n\}\subset
        \Gamma \mid \partial \sigma_i\cap\partial \sigma_j =
        \emptyset\mbox{ if } i\neq j \}.$$

Abrams showed in \cite{Ab} that for any $n>1$ and any graph $\Gamma$
with at least $n$ vertices, $D_n\Gamma(UD_n\Gamma$, respectively) is
a deformation retract of the topological ordered (unordered,
respectively) configuration space if a graph is sufficiently
subdivided so that

\begin{itemize}
\item[(1)] Each path between two vertices of valency $\neq 2$ passes
through at least $n+1$ edges;
\item[(2)] Each loop from a vertex to itself which cannot be shrunk
to a point in $\Gamma$ passes through at least n + 1 edges.
\end{itemize}

Abrams conjectured the version that there are at least $n-1$ edges
in the condition (1) instead of  $n+1$ edges. He also observed that
the statement is false if there are fewer than $n-1$ edges. For the
sake of completeness, we will prove the conjecture in Theorem~\ref{thm:homotopy equiv}.

From now on, we assume that $\Gamma$ is always sufficiently
subdivided so that the discrete configuration spaces are deformation retracts of their topological counterparts.
The \emph{pure graph braid group}
$P_n\Gamma$ and the \emph{graph braid group} $B_n\Gamma$ of $\Gamma$
are the fundamental groups of the ordered and the unordered discrete
configuration spaces of $\Gamma$, that is,
    $$P_n\Gamma=\pi_1(D_n\Gamma)\quad\mbox{and}
    \quad B_n\Gamma=\pi_1(UD_n\Gamma)$$

Abrams showed in \cite{Ab} that discrete configuration spaces
$D_n\Gamma$ and $UD_n\Gamma$ are cubical complexes of non-positive
curvature and so locally CAT(0) spaces. In particular, $D_n\Gamma$
and $UD_n\Gamma$ are Eilenberg-MacLane spaces, and $P_n\Gamma$ and
$B_n\Gamma$ are torsion-free, and the word problems and the
conjugacy problems in $P_n\Gamma$ and $B_n\Gamma$ are solvable.

\subsection{Presentation of graph braid groups}\label{ss22:morse}
In order to compute a presentation of graph braid groups, Farley and
Sabalka in \cite{FS1} considered a complex that is as simple as
possible but is still homotopy equivalent to $UD_n\Gamma$. By using
the discrete Morse theory developed by Forman in~\cite{For}, they
proved that $UD_n\Gamma$ can be systemically collapsed to a complex,
called a \emph{Morse complex}, which is simple enough to compute
$\pi_1$. Mostly following \cite{FS1}, we quickly review how to
collapse $UD_n\Gamma$ to its Morse complex.

\paragraph{\bf Step I: Give an order of vertices of a graph $\Gamma$}
First choose a maximal tree $T$ in $\Gamma$. We also assume that $T$
is always sufficiently subdivided in the sense of
\S\ref{ss21:configuration}. Edges in $\Gamma-T$ are called
\emph{deleted edges}. Pick a vertex of valency 1 in $T$ to be a
basepoint and assign 0 to this vertex. Let $R$ be a regular
neighborhood of $T$ in a plane. Then the boundary $\partial R$ is a
simple closed curve in the plane. Starting from the base vertex 0, we
number unnumbered vertices of $T$ as we travel along $\partial R$
clockwise. Figure~\ref{fig2}(a) illustrates this procedure for the
graph $S_0$ and for $n=4$. There is one deleted edge $d$ to form a
maximal tree. All vertices in $\Gamma$ are numbered and so are
referred by the nonnegative integers.

Each edge $e$ in $T$ is oriented so that the initial vertex
$\iota(e)$ is larger than the terminal vertex $\tau(e)$. The edge
$e$ is denote by $\tau(e)\mbox{-}\iota(e)$. A (open, cubical) cell
$c$ in the unordered configuration space $UD_n\Gamma$ can be written
as $\{c_1,\ldots,c_n\}$ where each $c_j$ is either a vertex or an
edge in $\Gamma$. The cell $c$ is an $i$-cell if the number of edges
among $c_j$'s is $i$. For example, $\{6\mbox{-}10, d, 0\mbox{-}1,
5\}$ in Figure~\ref{fig2} represents a 3-cell in $UD_4S_0$. Let
$K_i$ denote the set of all $i$-cells of $UD_n\Gamma$.

\begin{figure}[ht]
\psfrag{*}{\small0}
\psfrag{e}{\small$d$}
\psfrag{1}{\small1}
\psfrag{2}{\small2}
\psfrag{3}{\small3}
\psfrag{4}{\small4}
\psfrag{5}{\small5}
\psfrag{6}{\small6}
\psfrag{7}{\small7}
\psfrag{8}{\small8}
\psfrag{9}{\small9}
\psfrag{10}{\small10}
\psfrag{11}{\small11}
\psfrag{12}{\small12}
\centering \subfigure[Numbering a maximal tree $T$ of $S_0$ for $n=4$]
{\includegraphics[scale=.45]{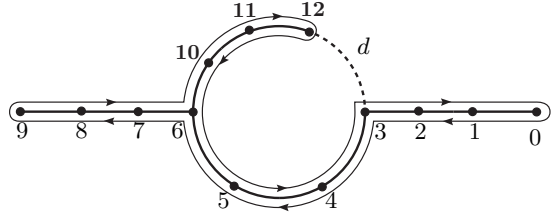}}
\subfigure[Orientations of edges in a maximal tree $T$ of $S_0$]
{\includegraphics[scale=.45]{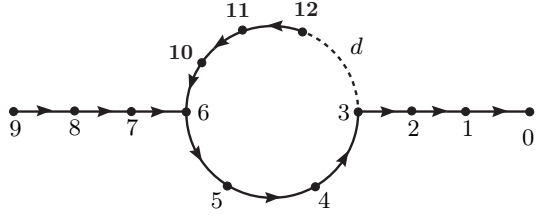}}
\caption{Maximal tree of $S_0$}\label{fig2}
\end{figure}

\paragraph{\bf Step II: Define a function $W:K_i\to K_{i+1}\cup\{\mbox{\rm void}\}$ for $i\ge -1$}
For a cell $c=\{c_1,\ldots,c_{n-1},v\} \in UD_n\Gamma$, a vertex $v$
is called \emph{blocked} in $c$ if $\tau(e)$ appears in $c$ (as a
vertex or as an end vertex of an edge) for the edge $e$ with
$\iota(e)=v$. Let $K_{-1}=\emptyset$. Define $W$ by induction on
$i$. Let $c=\{c_1,c_2,\ldots,c_n\}$ be an $i$-cell. If
$c\notin\im(W)$ and there are unblocked vertices in $c$ and, say,
$c_1$ is the smallest unblocked vertex then
$W(c)=\{v\mbox{-}c_1,c_2,\ldots,c_n\}$ where $v\mbox{-}c_1$ is the
edge in $T$ such that $v < c_1$. Otherwise, $W(c)=\mbox{void}$.
Then it is not hard to see that $W$ is well-defined, and each cell
in $W(K_i)-\{\mbox{void}\}$ has the unique preimage under $W$, and
there is no cell in $K_i$ that is both an image and a preimage of
other cells under $W$.

\paragraph{\bf Step III: Collapse $UD_n\Gamma$ to a Morse complex}
For each pair $(c, W(c))\in K_i\times(W(K_i)-\{\mbox{void}\})$, we
homotopically collapse the closure $\overline{W(c)}$  onto
$\overline{W(c)}-(W(c)\cup c)$ to obtain a Morse complex of
$UD_n\Gamma$.

For a pair $(c, W(c))\in K_i\times(W(K_i)-\{\mbox{void}\})$, cells
$c$ and $W(c)$ are said to be \emph{redundant} and
\emph{collapsible}, respectively. Redundant or collapsible cells
disappear in a Morse complex. A cell that survives in a Morse
complex is said to be \emph{critical}. Every cell in $\Gamma$ is one
of these three kinds that can be also characterized as follows: An
edge $e$ in a cell $c=\{c_1,\ldots,c_{n-1},e\}$ is \emph{
order-respecting} if  $e$ is not a deleted edge and there is no
vertex $v$ in $c$ such that $v$ is adjacent to $\tau(e)$ and
$\tau(e)<v<\iota(e)$. A cell is critical if it contains  neither
order-respecting edges nor unblocked vertices. A cell is collapsible
if it contains at least one order-respecting edge and each unblocked
vertex is larger than the initial vertex of some order-respecting
edge. A cell is redundant if it contains at least one unblocked
vertex that is smaller than the initial vertices of all
order-respecting edges. Notice that there is exactly one critical
0-cell $\{0,1,\ldots,n-1\}$.

For example, consider $UD_4S_0$ in Figure~\ref{fig2}. The cells
$\{0, 1, 2, 3\}$ and $\{6\mbox{-}10, 7, 8,  d\}$ are critical, and
the cells $\{1\mbox{-}2, 5, 6, d\}$ and $\{1\mbox{-}2, 5,
6\mbox{-}10, 7\}$ are collapsible, and the cells $\{8, 6\mbox{-}10,
11, d\}$ and $\{1, 2\mbox{-}3, 5, 6\mbox{-}10\}$ are redundant.

\begin{exa}
Morse complex of $UD_2S_0$.
\end{exa}

\begin{figure}[ht]
\psfrag{0}{\small0}
\psfrag{1}{\small1}
\psfrag{2}{\small2}
\psfrag{3}{\small3}
\psfrag{4}{\small4}
\psfrag{e}{\small$d$}
\psfrag{{0,1}}{\small$\{0,1\}$}
\psfrag{{0,2}}{\small$\{0,2\}$}
\psfrag{{0,3}}{\small$\{0,3\}$}
\psfrag{{0,4}}{\small$\{0,4\}$}
\psfrag{{1,2}}{\small$\{1,2\}$}
\psfrag{{1,3}}{\small$\{1,3\}$}
\psfrag{{1,4}}{\small$\{1,4\}$}
\psfrag{{2,3}}{\small$\{2,3\}$}
\psfrag{{2,4}}{\small$\{2,4\}$}
\psfrag{{3,4}}{\small$\{3,4\}$}
\psfrag{{0-1,3-4}}{\small$\{0\mbox{-}1,3\mbox{-}4\}$}
\psfrag{{0-1,2-3}}{\small$\{0\mbox{-}1,d\}$}
\psfrag{{1-2,3-4}}{\small$\{1\mbox{-}2,3\mbox{-}4\}$}
\psfrag{{e,4}}{\small$\{d,4\}$}
\psfrag{{e,0}}{\small$\{d,0\}$}
\psfrag{{1-3,2}}{\small$\{1\mbox{-}3,2\}$}
\centering
\subfigure[Maximal tree of $S_0$ for $n=2$]
{\includegraphics[height=1.8cm]{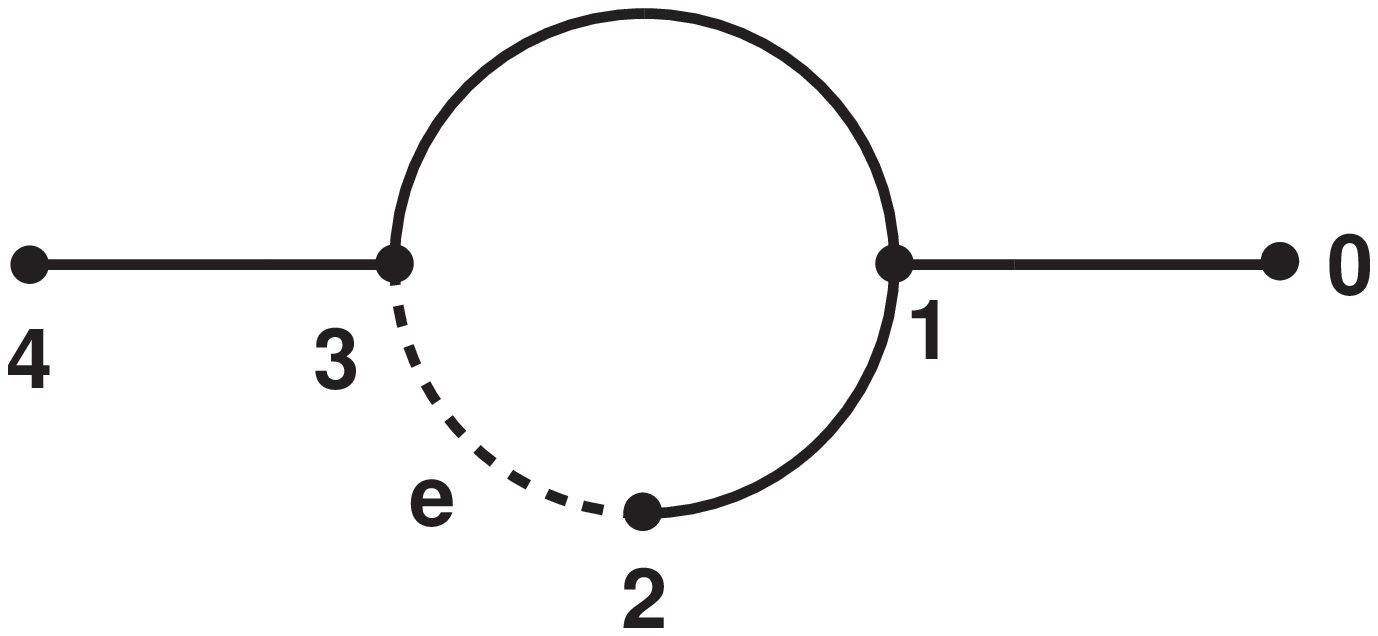}}\\
\subfigure[$UD_2S_0$ and gradient field for collapsing]
{\includegraphics[height=3.6cm]{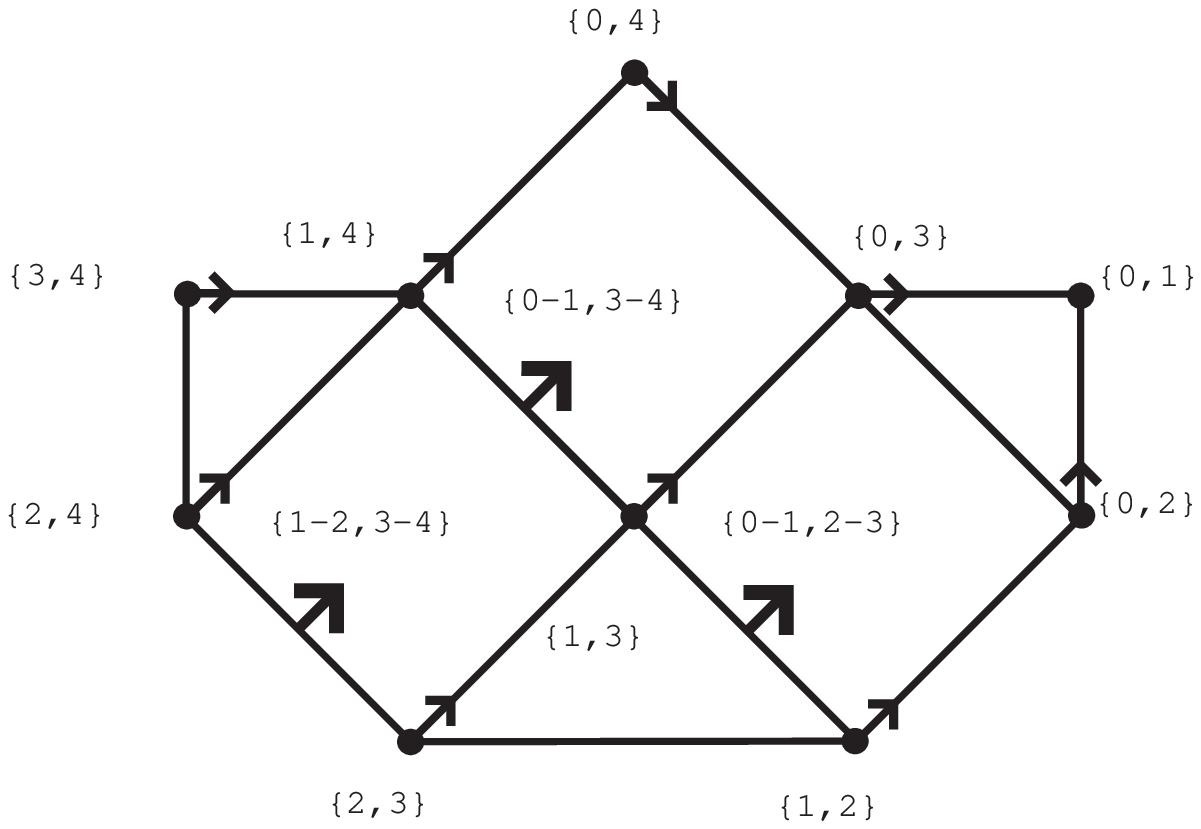}}\quad
\subfigure[Morse complex]
{\includegraphics[height=3.6cm]{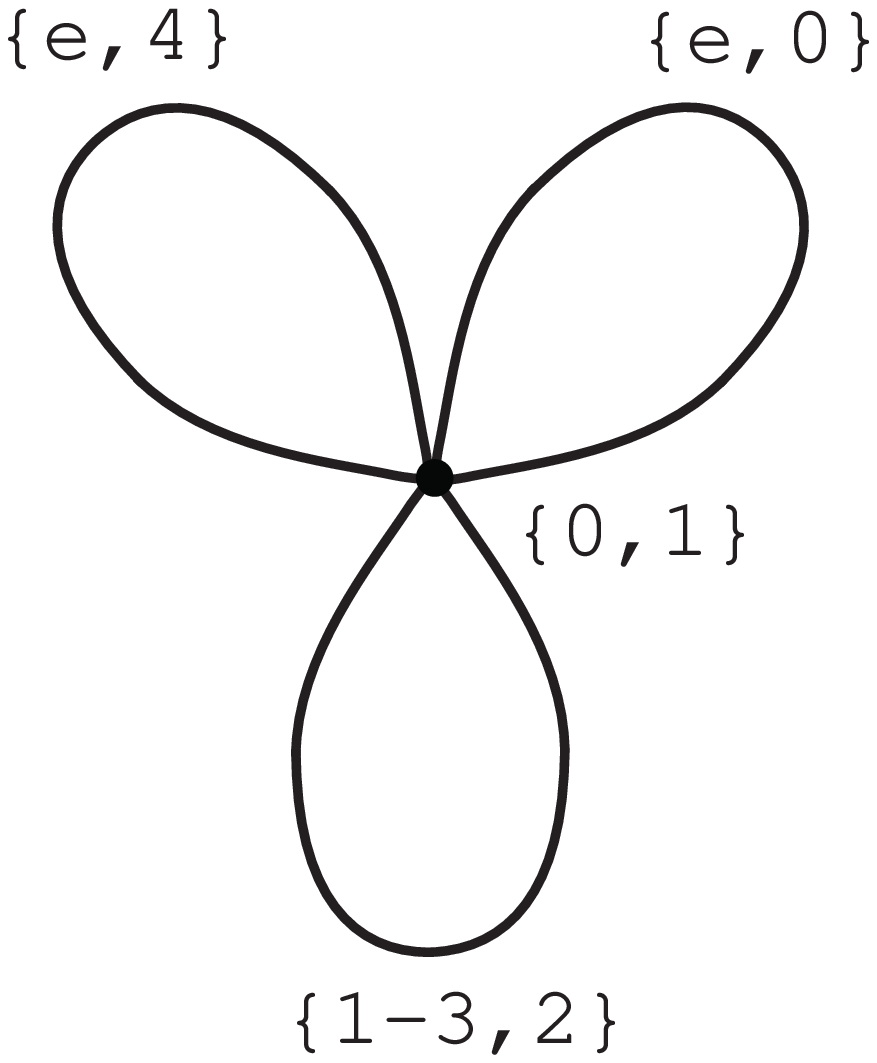}}
\caption{$S_0$}\label{fig3}
\end{figure}

The complex $UD_2S_0$ is a 2-dimensional complex as in
Figure~\ref{fig3}(b). The order of vertices of a maximal tree of
$S_0$ as in Figure~\ref{fig3}(a) determines $W$  as follows: For
0-cells, $W(\{0,2\})=\{0, 1\mbox{-}2\}$, $W(\{0,3\})=\{0,
1\mbox{-}3\}$, $W(\{0,4\})=\{0, 3\mbox{-}4\}$,
$W(\{1,2\})=\{0\mbox{-}1,2\}$, $W(\{1,3\})=\{0\mbox{-}1,3\}$,
$W(\{1,4\})=\{0\mbox{-}1,4\}$, $W(\{2,3\})=\{1\mbox{-}2,3\}$,
$W(\{2,4\})=\{1\mbox{-}2,4\}$, $W(\{3,4\})=\{1\mbox{-}3,4\}$, and
for 1-cells, $W(\{3\mbox{-}4,2\})=\{1\mbox{-}2, 3\mbox{-}4\}$,
$W(\{d,1\})=\{0\mbox{-}1, d\}$, $W(\{3\mbox{-}4,1\})=\{0\mbox{-}1,
3\mbox{-}4\}$.

After collapsing, a Morse complex of $UD_2S_0$ is a 1-dimensional
complex which has one 0-cell $\{0,1\}$ and three 1-cells $\{0,
d\},\{4, d\},\{1\mbox{-}3,2\}$ as depicted in Figure~\ref{fig3}(c).

The braid group $B_n\Gamma$ is now given by the fundamental group of
a Morse complex of $UD_n\Gamma$. Thus $B_n\Gamma$ has a presentation
whose generators are critical 1-cells and whose relators are
boundary words of critical 2-cells in terms of critical 1-cells. For
example, $B_2S_0$ is the free group generated by the three critical
1-cells above.

In order to rewrite a word in 1-cells of $UD_n\Gamma$ into a new
word in critical 1-cells, we use the rewriting homomorphism $\tilde
r$ from the free group on $K_1$ to the free group on the set of
critical 1-cells defined as follows: First define a homomorphism $r$
from the free group on $K_1$ to itself by $r(c)= 1$ if $c$ is
collapsible, $r(c)= c$ if $c$ is critical, and
$$r(c)=\{v\mbox{-}v_1,v_2,\ldots,v_{n-1},\iota(e)\}
\{v,v_2,\ldots,v_{n-1},e\}
\{v\mbox{-}v_1,v_2,\ldots,v_{n-1},\tau(e)\}^{-1}$$ if
$c=\{v_1,v_2,\ldots,v_{n-1},e\}$ is redundant such that $v_1$ is the
smallest unblocked vertex and $e$ is the edge in $c$. Forman's
discrete Morse theory in \cite{For} guarantees that there is a
nonnegative integer $k$ such that $r^k(c)=r^{k+1}(c)$ for all $c\in
K_1$. An essential idea is that a sequence of consecutive collapsing
never circles back like a gradient flow of a height function. Once
this fact is established, the finiteness of $K_1$ implies the
existence of such a $k$. Let $\tilde r=r^k$, then for any $c\in
K_1$, $\tilde r(c)$ is a word in critical 1-cells that is the image
of $c$ under the quotient map defined by collapsing $UD_n\Gamma$ unto its
Morse complex. We note that $k=0$ iff $c$ is critical, $k=1$ iff $c$ is collapsible, and $k\ge 2$ iff $c$ is redundant. In fact, it is impossible that $r(c)$ is a product of three critical cells and so $k>1$ if $c$ is redundant.

Using the rewriting map $\tilde r$, we can rewrite the boundary word
of a critical 2-cell in terms of critical 1-cells. Thus it is
possible to compute a presentation of $B_n\Gamma$ using a Morse
complex of $UD_n\Gamma$. However, the computation of $\tilde r$ is
usually tedious and the following lemma somewhat shortens it. In particular,
the lemma will be useful in the proof of Theorem~\ref{thm:homotopy equiv}.

\begin{lem}\label{lem:rewriting}
Let $c$ be a $1$-cell in $UD_n\Gamma$ and $e$ be an edge in $\Gamma$
such that $\iota(e)$ is an unblocked vertex in $c$. If $\iota(e)-\tau(e)=1$, then $\tilde r(c)=\tilde r(V_e(c))$ where $V_e(c)$ denotes the $1$-cell obtained from $c$ by replacing $\iota(e)$ by $\tau(e)$.
\end{lem}
\begin{proof}
We first remark that this result was proved by Farley and Sabalka in \cite{FS1} if $\iota(e)$ is the smallest unblocked vertex in $c$.
We proceed by induction on the minimal integer $k$ such that $r^k(c)=r^{k+1}(c)$. Since $c$ must be either collapsible or redundant, $k\ge 1$. If $k=1$, the $c$ is collapsible and $\tilde r(c)=\tilde r(V_e(c))=0$.

Assume that $k\ge2$. Let $e'$ be the edge in $c$ and $e_s$ be the edge in $\Gamma$ such that $\iota(e_s)$ is the smallest unblocked vertex in $c$. If $\iota(e)=\iota(e_s)$, then we are done by the result of Farley and Sabalka. Now assume $\iota(e)\neq\iota(e_s)$. Then we may write $r(c)=c_{e'}^{\iota}c_{e_s}^{\tau}(c_{e'}^{\tau})^{-1}$ where $c_{e'}^{\iota}$ and $c_{e'}^{\tau}$ are $1$-cells obtained from the collapsible 2-cell $W(c)$ by replacing an edge $e'$ by $\iota(e')$ and $\tau(e')$, respectively. Since $v$ is unblocked in $c_{e'}^{\iota}$, $c_{e_s}^{\tau}$ and  $c_{e'}^{\tau}$, all three $1$-cells $c_{e'}^{\iota}$, $c_{e_s}^{\tau}$ and  $c_{e'}^{\tau}$ are either collapsible or redundant. If $c_{e'}^{\iota}$ is collapsible, then $e_s$ is an order-respecting edge proceeded by no unblocked vertices, that are smaller than $\iota(e_s)$, in both $c_{e'}^{\iota}$ and $V_e(c_{e'}^{\iota})$ by our assumption. So $V_e(c_{e'}^{\iota})$ is collapsible and $\tilde r(c_{e'}^{\iota})=\tilde r(V_e(c_{e'}^{\iota}))$. On the other hand, if $c_{e'}^{\iota}$ is redundant, then $\tilde r(c_{e'}^{\iota})=\tilde r(V_e(c_{e'}^{\iota}))$ by induction because $c_{e'}^{\iota}$ requires one less iterations of $r$ to stabilize than $c$. Similarly we can show $\tilde r(c_{e_s}^{\tau})=\tilde r(V_e(c_{e_s}^{\tau}))$ and $\tilde r(c_{e'}^{\tau})=\tilde r(V_e(c_{e'}^{\tau}))$. Since $\iota(e_s)$ is still the smallest unblocked vertex in $V_e(c)$, $r(V_e(c))=V_e(c_{e'}^{\iota})V_e(c_{e_s}^{\tau})V_e(c_{e'}^{\tau})^{-1}$. Thus $\tilde r(c)=\tilde r(r(c))= \tilde r(c_{e'}^{\iota}c_{e_s}^{\tau}(c_{e'}^{\tau})^{-1})= \tilde r(V_e(c_{e'}^{\iota})V_e(c_{e_s}^{\tau})V_e(c_{e'}^{\tau})^{-1})=\tilde r(r(V_e(c)))=\tilde r(V_e(c))$.
\end{proof}

\begin{exa}\label{ex:S_0}
$B_4S_0$ is a right-angled Artin group.
\end{exa}

\begin{proof}
By Abrams' conjecture proved in Theorem~\ref{thm:homotopy equiv}, $B_4S_0$ is the fundamental group of $UD_4S_0$ with a subdivision, a maximal tree, and an order as given in Figure~\ref{fig2}.
Then $UD_4S_0$ has 10 critical 1-cells $\{d,0,1,2\}$,
$\{d,0,1,4\}$, $\{d,0,4,5\}$, $\{d,4,5,6\}$,
$\{6\mbox{-}10,0,1,7\}$, $\{6\mbox{-}10,0,7,11\}$,
$\{6\mbox{-}10,0,7,8\}$, $\{6\mbox{-}10,7,11,12\}$,
$\{6\mbox{-}10,7,8,11\}$, $\{6\mbox{-}10,7,8,9\}$, and $4$ critical
2-cells $\{d,6\mbox{-}10,0,7\}$, $\{d,6\mbox{-}10,4,7\}$,
$\{d,6\mbox{-}10,7,11\}$, $\{d,6\mbox{-}10,7,8\}$.

The boundary word of the critical 2-cell $\{d,6\mbox{-}10,0,7\}$ is
the product
$$\{6\mbox{-}10,0,7,12\}\{d,0,6,7\}
\{6\mbox{-}10,0,3,7\}^{-1}\{d,0,7,10\}^{-1}.$$ By Lemma~\ref{lem:rewriting},
$\tilde r(\{d,0,6,7\})=\tilde r(\{d,0,5,7\})=\tilde r(\{d,0,4,7\})=\tilde r(\{d,0,4,6\})=\tilde r(\{d,0,4,5\})=\{d,0,4,5\}$. Similarly, we have $\tilde
r(\{6\mbox{-}10,0,7,12\})=\{6\mbox{-}10,0,7,11\}$, $\tilde
r(\{6\mbox{-}10,0,3,7\})=\{6\mbox{-}10,0,1,7\}$, and $\tilde
r(\{d,0,7,10\})=\{d,0,4,5\}$. Consequently, the boundary word is
rewritten in terms of critical 1-cells as follows: $$\tilde
r(\partial(\{d,6\mbox{-}10,0,7\}))
=\{6\mbox{-}10,0,7,11\}\{d,0,4,5\}\{6\mbox{-}10,0,1,7\}^{-1}\{d,0,4,5\}^{-1}.$$
Apply the rewriting map, we can obtain the boundary words of other
critical 2-cells in terms of critical 1-cells as follows:
\begin{align*}
\tilde r(\partial(\{d,6\mbox{-}10,4,7\}))&
=\{6\mbox{-}10,0,7,11\}\{d,4,5,6\}\{6\mbox{-}10,0,1,7\}^{-1}\{d,4,5,6\}^{-1},\\
\tilde r(\partial(\{d,6\mbox{-}10,7,11\}))&
=\{6\mbox{-}10,7,11,12\}\{d,4,5,6\}\{6\mbox{-}10,0,7,11\}^{-1}\{d,4,5,6\}^{-1},\\
\tilde r(\partial(\{d,6\mbox{-}10,7,8\}))&
=\{6\mbox{-}10,7,8,11\}\{d,4,5,6\}\{6\mbox{-}10,0,7,8\}^{-1}\{d,4,5,6\}^{-1}.
\end{align*}
We eliminate the generators $\{6\mbox{-}10,0,7,11\}$,
$\{6\mbox{-}10,7,11,12\}$ and $\{6\mbox{-}10,7,8,11\}$ by applying a
Tietze transformation. Then there are seven generators
\begin{gather*}
\{d,0,1,2\},\{d,0,1,4\},\{d,0,4,5\},\{d,4,5,6\},\\\{6\mbox{-}10,0,1,7\},
\{6\mbox{-}10,0,7,8\},\{6\mbox{-}10,7,8,9\}.
\end{gather*}
and one relation
$[\{d,0,4,5\}^{-1}\{d,4,5,6\},\{6\mbox{-}10,0,1,7\}]$. And we can
introduce a new generator $g$ by adding the relation
$g=\{d,0,4,5\}^{-1}\{d,4,5,6\}$ and then eliminate the generator
$\{d,4,5,6\}$ by applying a Tietze transformation. Then we obtain
the presentation of $B_4S_0$ with seven generators $g$,
$\{d,0,1,2\}$, $\{d,0,1,4\}$, $\{d,0,4,5\}$,
$\{6\mbox{-}10,0,1,7\}$, $\{6\mbox{-}10,0,7,8\}$,
$\{6\mbox{-}10,7,8,9\}$ and one relator $[g,\{6\mbox{-}10,0,1,7\}]$.
Thus $B_4(S_0)$ is a right-angled Artin group. Note that $B_n(S_0)$
is not a right-angled Artin group for $n\ge5$, as we will see in
\S\ref{s:four}.
\end{proof}

We now explain how to obtain a geometric graph braid corresponding to a critical 1-cell. Given a graph $\Gamma$ ordered by a planar embedding of its maximal tree, the 1-skeleton of the Morse complex is a bouquet consisting of critical 1-cells. So the subcomplex consisting of all 0-cells and all collapsible 1-cells in $UD_n\Gamma$ forms a maximal tree $S$ of the 1-skeleton of $UD_n\Gamma$. Let $c=\{e,v_1,\cdots,v_{n-1}\}$ is a critical 1-cell in $UD_n\Gamma$ and $0_n=\{0,1,\cdots,n-1\}$ be the base point of $UD_n\Gamma$. Then there are unique edge paths $\alpha$ from $0_n$ and $\{\iota(e),v_1,\cdots,v_{n-1}\}$ and $\beta$ from $\{\tau(e),v_1,\cdots,v_{n-1}\}$ and $0_n$ in the tree $S$. The edge loop $\alpha e\beta$ is a typical generator of $B_n\Gamma$. The edge loop uniquely corresponds to a geometric $n$-braid on $\Gamma$ given by the motion of $n$ vertices in $0_n$ along $\alpha e\beta$.

For example, the critical 1-cell $\{6\mbox{-}10,7,8,11\}$ of $UD_4S_0$ in Example~\ref{ex:S_0} determines the edge loop given by
\begin{align*}
&\{0,1,2,3\}\to\{0,1,2,4\}\to\{0,1,2,5\}\to\cdots\to\{0,1,2,11\}\to\{0,1,3,11\}\to\cdots\\
&\to\{0,1,10,11\}\to\cdots\to\{0,8,10,11\}\to\{1,8,10,11\}\to\cdots\to\{7,8,10,11\}\\
&\to\{6,7,8,11\}\to\{5,7,8,11\}\to\cdots\to\{0,7,8,11\}\to\{0,6,8,11\}\to\cdots\\
&\to\{0,1,8,11\}\to\{0,1,7,11\}\to\cdots\to\{0,1,2,11\}\to\cdots\to\{0,1,2,3\}
\end{align*}
where the 1-cell $\{6\mbox{-}10,7,8,11\}$ gives the edge path $\{7,8,10,11\}\to\{6,7,8,11\}$. This loop corresponds to the geometric 4-braid in Figure~\ref{fig4}(a).

Similarly the critical 1-cell $\{d,0,1,2\}(=\{0,1,2,12\}\to\{0,1,2,3\})$ of $UD_4S_0$, determines the edge loop given by
$$\{0,1,2,3\}\to\{0,1,2,4\}\to\cdots\to\{0,1,2,12\}\to\{0,1,2,3\}$$
and corresponds to the geometric 4-braid in Figure~\ref{fig4} (b).

\begin{figure}[ht]
\subfigure[$\{6\mbox{-}10,7,8,11\}$]
{\includegraphics[height=2.8cm]{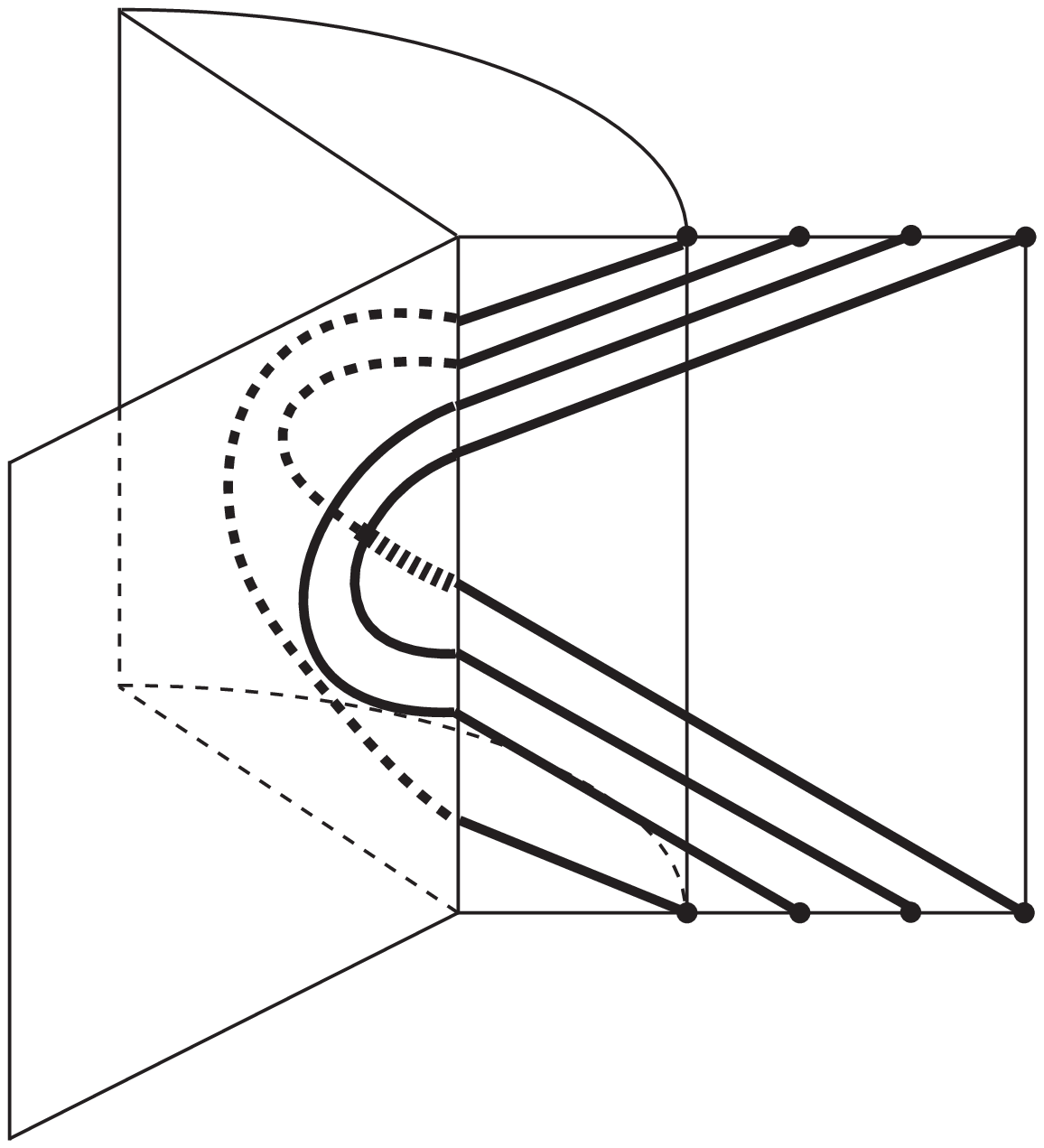}}\qquad\qquad
\subfigure[$\{d,0,1,2\}$]
{\includegraphics[height=2.5cm]{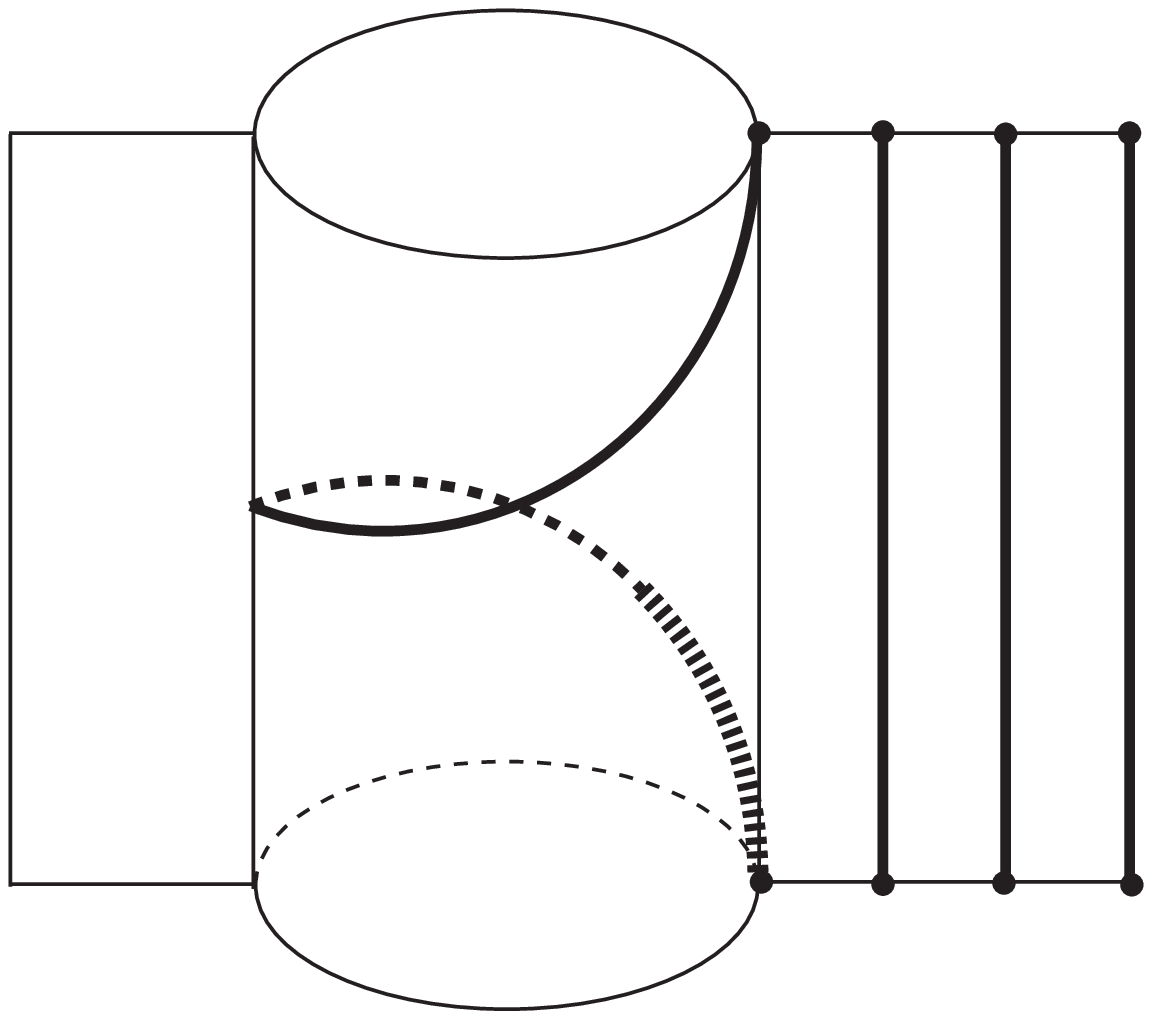}}
\caption{Geomtric graph braids on $S_0$}
\label{fig4}
\end{figure}

Since geometric braids permute the $n$ vertices in $0_n$, there is a natural homomorphism $\pi$ from $B_n\Gamma$ into the symmetric group $S_n$.  If the vertex $i$ in $0_n$ is regarded as a puncture labeled by $n-i$ for $0\le i\le n-1$,
then $\pi(\{6\mbox{-}10,7,8,11\})=(4,3,2)$ and $\pi(\{d,0,1,2\})$ is the identity.
The homomorphism $\pi$ obviously depends on the choice of labels but it is well-defined up to conjugation. Thus for a graph $\Gamma$ ordered by a planar embedding of its maximal tree, there is a well-defined homomorphism $$\pi_*:H_1(B_n\Gamma)\to H_1(S_n)=\mathbb Z_2$$ which will be useful in \S\ref{ss44:nonplanar}.

When we work with an arbitrary graph of an arbitrary braid index
$n$, it is convenient to represent cells of $UD_n(\Gamma)$ by using
the following notations adapted from \cite{FS1}. Let $A$ be a vertex
of valency $\mu\ge 3$ in a maximal tree of $\Gamma$. Starting from
the branch clockwise next to the branch joining to the base vertex, we
number branches incident to $A$ clockwise. Let $\vec a$ be a vector
$(a_1,\ldots,a_{\mu-1})$ of nonnegative integers and let $|\vec
a|=\sum_{i=1}^{\mu-1} a_i$. Then for $1\le k\le \mu-1$, $A_k(\vec
a)$ denotes the set consisting of one edge $e$ with $\tau(e)=A$ that
lies on the $k$-th branch together with $a_i$ blocked vertices that
lies on the $i$-th branch if $i\neq k$ and $a_k-1$ blocked vertices
on the $k$-th branch. For $1\le s\le n$, $0_s$ denotes the set
$\{0,1,\ldots,s-1\}$ of $s$ consecutive vertices from the base vertex.
Let $\dot A(\vec a)$ denote the set of vertices consisting of $A$
together with $a_i$ blocked vertices that lies on the $i$-th branch
and let $A(\vec a)=\dot A(\vec a)-\{A\}$. Then $A(\vec a)$ can be
obtained from $A_k(\vec a)$ by replacing an edge $e$ with
$\iota(e)$. Every critical $i$-cell is represented by the following
union:
$$ A^1_{k_1}(\vec a^1)\cup
\ldots\cup A^p_{k_p}(\vec a^p)\cup
\{d_1,\ldots,d_q\}\cup\{v_1,\ldots,v_r\}\cup 0_s,$$ where
$A^1,\ldots,A^p$ are vertices of valency $\ge 3$, and
$d_1,\ldots,d_q$ are deleted edges, and $v_1,\ldots,v_r$ are blocked
vertices blocked by deleted edges. Furthermore, since $s$ is
uniquely determined by $s=n-(|\vec a^1|+\cdots+|\vec a^p|+q+r)$, we
will omit $0_s$ in the notation.

We now prove the conjecture by Abrams for braid index $n\ge 3$. Abrams himself prove the conjecture for $n=2$ in \cite{Ab}.
\begin{thm}\label{thm:homotopy equiv}
Suppose that $\Gamma$ is a graph with at least $n(\ge 3)$ vertices such that each path
joining two vertices of valency $\ne 2$ contains at least $n-1$ edges, and each
simple loop contains at least $n+1$ edges.
Then
$UD_n\Gamma$($D_n\Gamma$, respetively) is a deformation retract of $UC_n\Gamma$($C_n\Gamma$, respetively).
\end{thm}
\begin{proof}
We first choose a maximal tree $T$ of $\Gamma$ so that one of end vertices of every deleted edge has valency $\ge 3$ in $\Gamma$ and the base vertex 0 has valency 1 in $T$ and the path between 0 and the nearest vertex of valency $\ge 3$ in $T$ contains at least $n-1$ edges. Then we give an order on vertices of $T$ by choosing a planar embedding of $T$. Let $\Gamma'$ be obtained
by subdividing $\Gamma$ so that each path joining two vertices of valency $\ge 3$ contains at least $n+1$ edges. When we subdivide, we do not touch any deleted edge.
We note that the assumption $n\ge3$ is needed here.
A maximal tree $T'$ of $\Gamma'$ and its planar embedding are obtained by subdividing $T$ so that an order on $T'$ is assigned accordingly. Since $\Gamma'$ is a subdivision of $\Gamma$, vertices of valency $\ge 3$ in $\Gamma$ and in $\Gamma'$ coincide and deleted edges in $\Gamma$ and in $\Gamma'$ also coincide. We use the same notation to denote two coinciding objects.
Let $K_i$ (and $K'_i$, respectively) be the set of all critical $i$-cells in $UD_n\Gamma$ (and $UD_n\Gamma'$). Each cell in $K_i$ containing non-order-respecting edges given by $A_k(\vec a)$'s and deleted edges $d$'s can be identified with the $i$-cell in $K'_i$ containing the non-order-respecting edges and the deleted edges that are denoted by the same notations. Thus $K_i$ can be naturally regarded as a subset of $K'_i$ for each $i$. We remark that a non-order-respecting edge in $K'_i$ determined by $A_k(\vec a)$ may be a subdivided part of the correspond non-order-respecting edge in $K_i$.

Due to our requirement on the base vertex, $K_0=K'_0$ consists of a single critical 0-cell. Let $\tilde r$ ($\tilde r'$, respectively) be the rewriting map of 1-cells in terms of critical 1-cells in $UD_n\Gamma$ (and $UD_n\Gamma'$).
The fundamental group $\pi_1(UD_n\Gamma)$ ($\pi_1(UD_n\Gamma')$, respectively) is generated over $K_1$ ($K'_1$, respectively) and related by $\tilde r(\partial c)$ for each $c\in K_2$ ($K'_2$, respectively). Our first goal is to obtain an isomorphism between $\pi_1(UD_n\Gamma)$ and $\pi_1(UD_n\Gamma')$ which can be established by applying Tietze transformations if we show the following two facts:
\begin{itemize}
\item[(1)] For each $c\in K_2$, $\tilde r(\partial c)=\tilde r'(\partial c)$;
\item[(2)] There is a bijection $h: K'_2-K_2\to K'_1-K_1$ such that for each $c\in K'_2-K_2$, the word $\tilde r'(\partial c)$ contains $h(c)$ or $h(c)^{-1}$ exactly once.
\end{itemize}

For a vertex $v$ in a tree $T$ ordered by its planar embedding, the location of $v$ is completely determined by the triple $(Y_T(v),\beta_T(v),\delta_T(v))$ where $Y_T(v)$ is the nearest smaller vertex of valency $\ge 3$ or the base vertex 0 if no such vertex exists, and $v$ lies on the $\beta_T(v)$-th branch of $Y_T(v)$ clockwise from the branch to the base vertex, and $\delta_T(v)$ is the number of edges in the shortest path from $v$ to $Y_T(v)$. If $v$ itself is of  valency $\ge 3$ or the base vertex, then $\beta_T(v)=\delta_T(v)=0$. A 1-cell $b'$ in $UD_n\Gamma'$ with a non-order-respecting edge or a deleted edge is said to be {\em realizable in $\Gamma$} if there is a 1-cell $b$ in $UD_n\Gamma$ with the same non-order-respecting edge or the same deleted edge such that for each $v'\in b'$, $(Y_{T'}(v'),\beta_{T'}(v'),\delta_{T'}(v'))=(Y_T(v),\beta_T(v),\delta_T(v))$ for some $v\in b$.

For $c\in K_2$, we try to compute $\tilde r'(\partial c)$. Whenever we encounter a redundant 1-cell $b'$ that is not realizable in $\Gamma$ on the course of this computation, $b'$ must contain exactly one vertex $v'$ such that $\delta_{T'}(v')>\delta_T(v)$ for any vertex $v$ of $T$ satisfying $Y_T(v)=Y_{T'}(v')$ and $\beta_T(v)=\beta_{T'}(v')$. We move $v'$ forward by applying $V_e$ in Lemma~\ref{lem:rewriting} as many times as needed until $b'$ becomes realizable in $\Gamma$. It is now clear that $\tilde r$ and $\tilde r'$ yield the same output. This prove the part (1).

To show the part (2), we first characterize critical cells in $K'_2-K_2$ and $K'_1-K_1$.
A pair of vertices $A, B$ of valency $\ge 3$ in $\Gamma$  are said to be {\em restrictive} if the shortest path joining $A$ and $B$ in $\Gamma$ contains exactly $n-1$ edges, and every vertex on the interior of the path has valency 2 in $\Gamma$, and there is a deleted edge $d$ on the path such that its one end vertex is $B$. In order to have $K'_2-K_2\ne\emptyset$ and $K'_1-K_1\ne\emptyset$, it is necessary that there is at least one restrictive pair in $\Gamma$.  Let $\mu$ be the valency of $A$ in $T$ and assume that $d$ is on the $k_0$-th branch at $A$ for some $1\le k_0 \le\mu$. Let $\vec\delta_i$ be the $\mu$-dimensional unit vector whose $i$-th component equals 1.

There are three kinds of critical 2-cells in $K'_2-K_2$ that can be described by using the notation of a restrictive pair(See Figure~\ref{fig5}):
\begin{enumerate}
\item[(i)] $d\cup A_k((n-2)\vec\delta_{k_0}+\vec\delta_k)$ for $k_0<k\le\mu$;
\item[(ii)] $d\cup d'\cup A((n-2)\vec\delta_{k_0})$ for a deleted edge $d'$ with an end vertex $A$;
\item[(iii)] $d\cup A_{k_0}((n-2)\vec\delta_{k_0}+\vec\delta_k)$ for some $1\le k< k_0$.
\end{enumerate}

\begin{figure}
\psfrag{A}{\small$A$}
\psfrag{B}{\small$B$}
\psfrag{d}{\small$d$}
\psfrag{n1}{\small$n-1$ vertices}
\psfrag{n2}{\small$n-2$ vertices}
\psfrag{d1}{\small$d'$}
\psfrag{0}{\small$0$}
\centering
\subfigure[Restrictive pair]
{\includegraphics[scale=.29]{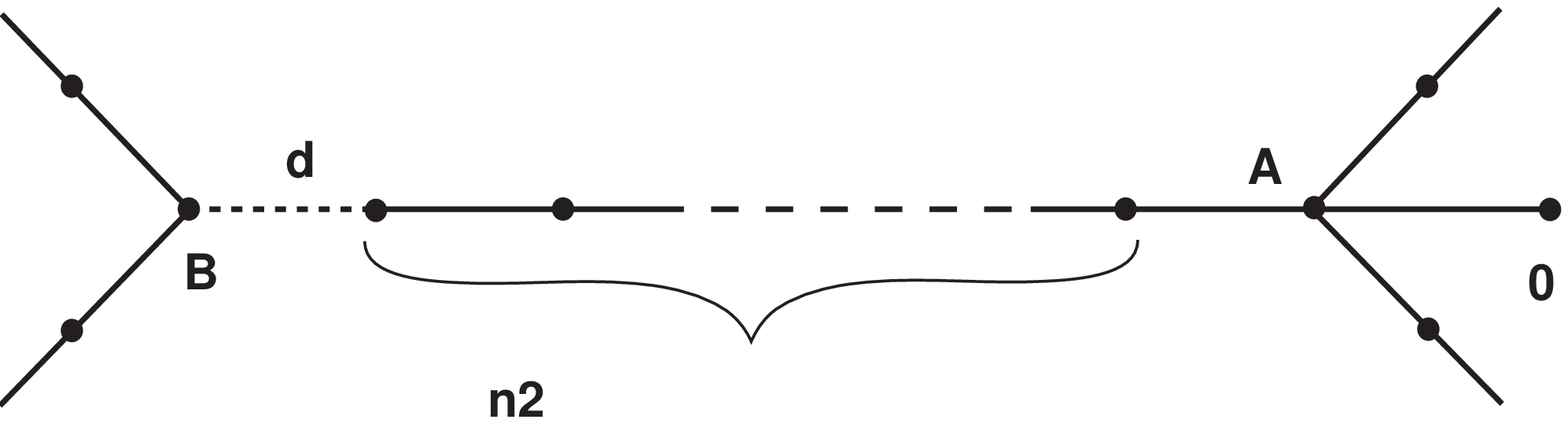}}\qquad
\subfigure[$A_k((n-1)\vec\delta_{k_0}+\vec\delta_k)$]
{\includegraphics[scale=.29]{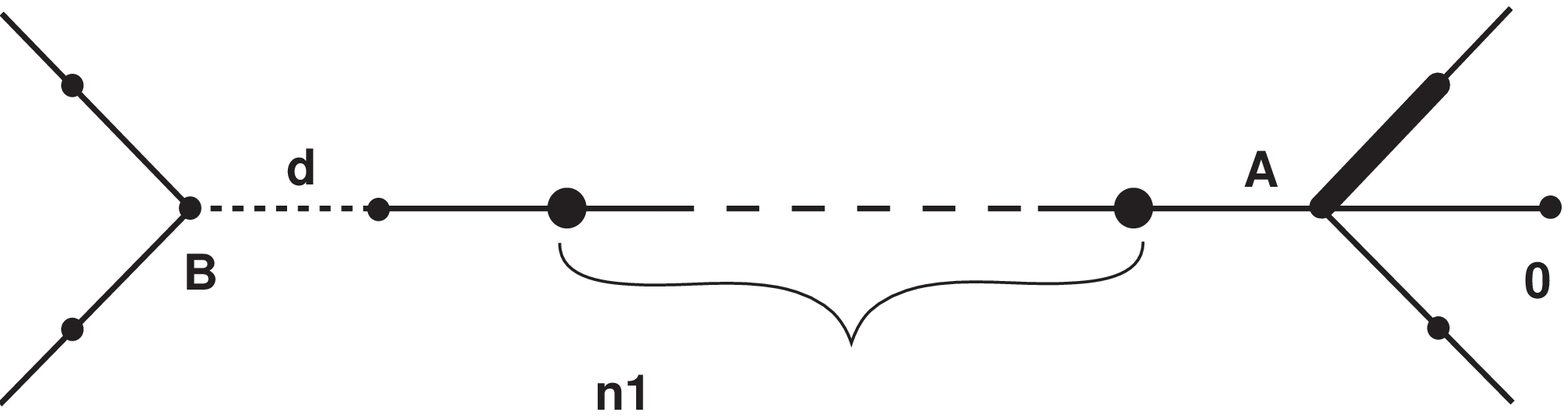}}\\
\subfigure[$A_{k_0}((n-1)\vec\delta_{k_0}+\vec\delta_k)$]
{\includegraphics[scale=.29]{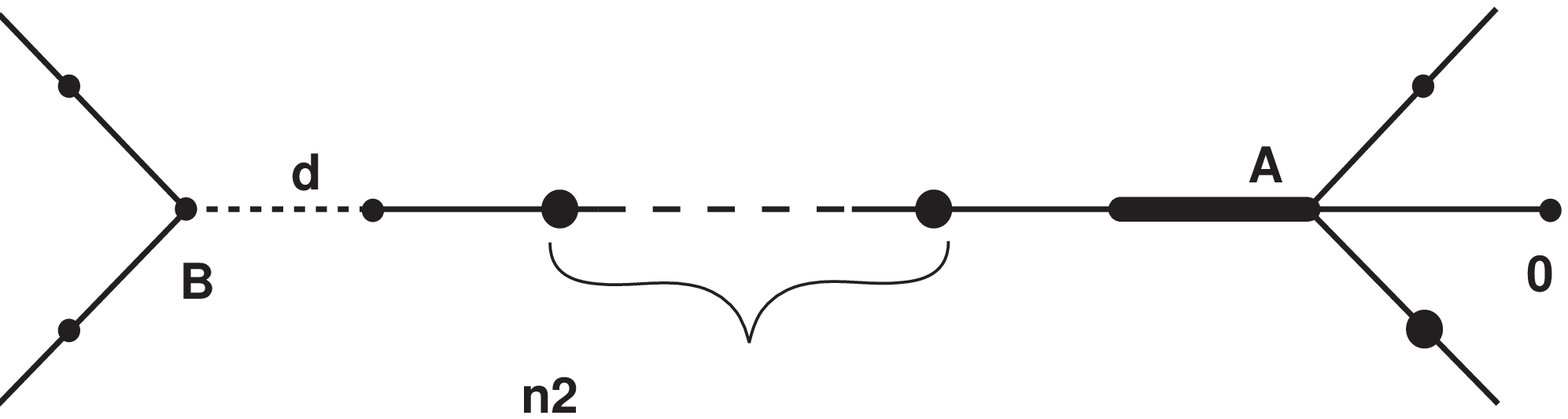}}\qquad
\subfigure[$d'\cup A((n-1)\vec\delta_{k_0})$]
{\includegraphics[scale=.29]{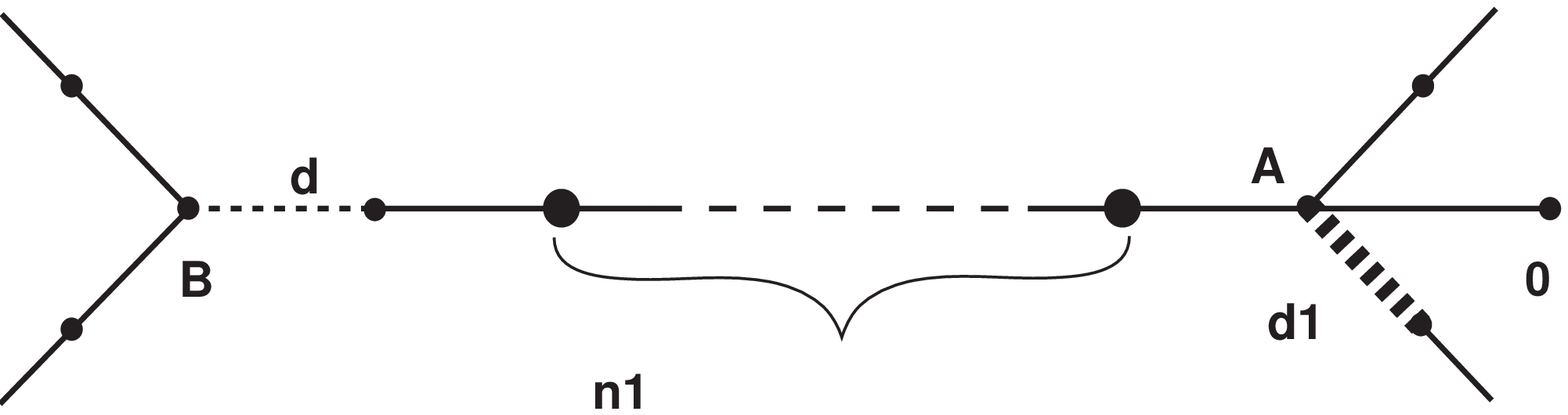}}
\caption{Critical 1-cells in $K'_1-K_1$}\label{fig5}
\end{figure}

There are also corresponding three kinds of critical 1-cells in $K'_1-K_1$:
\begin{enumerate}
\item[(i)] $A_k((n-1)\vec\delta_{k_0}+\vec\delta_k)$ for $k_0<k\le\mu$;
\item[(ii)] $d'\cup A((n-1)\vec\delta_{k_0})$ for a deleted edge $d'$ with an end vertex $A$;
\item[(iii)] $A_{k_0}((n-1)\vec\delta_{k_0}+\vec\delta_k)$ for some $1\le k< k_0$.
\end{enumerate}
We define a bijection $h: K'_2-K_2\to K'_1-K_1$ by sending a critical 2-cell in $K'_2-K_2$ to the corresponding critical 1-cell in $K'_1-K_1$. It is not hard to see that the word $\tilde r'(\partial c)$ contains $h(c)$ or $h(c)^{-1}$ exactly once. For example, $\partial(d\cup A_k((n-2)\vec\delta_{k_0}+\vec\delta_k))$ contains $\{v_d\}\cup A_k((n-2)\vec\delta_{k_0}+\vec\delta_k)$ or its inverse exactly once where $v_d$ is the end vertex of $d$ that is not $B$. But
$$\tilde r'(\{v_d\}\cup A_k((n-2)\vec\delta_{k_0}+\vec\delta_k))=\tilde r'(A_k((n-1)\vec\delta_{k_0}+\vec\delta_k))$$
by Lemma~\ref{lem:rewriting}.

In order to realize this isomorphism of fundamental groups by a continuous map, we need to consider an intermediate cubical complex $\overline{UD}_n$ obtained by subdividing $UD_n\Gamma$ as follows: Whenever an edge $e$ in $\Gamma$ is subdivided into $e_1,\ldots,e_\ell$ in $\Gamma'$, we replace each cell $\{e,x_1\ldots x_i\}$ containing $e$ by $(2\ell-1)$ cells $\{e_1,x_1\ldots x_i\}$, $\{v_1,x_1\ldots x_i\}$, $\{e_2,x_1\ldots x_i\},\ldots, \{v_{\ell-1},x_1\ldots x_i\}$, $\{e_\ell,x_1\ldots x_i\}$ where $v_j=\tau(e_j)=\iota(e_{j+1})$. Then there is an obvious cubical embedding $\bar i:\overline{UD}_n\hookrightarrow UD_n\Gamma'$. We remark that $UD_n\Gamma'$ have more cells such as one containing more than one vertices or edges produced by a subdivision of an edge. The cubical complex $\overline{UD}_n$ still has the obvious 0-cell as the basepoint.
A critical cell $c$ in $UD_n\Gamma$ may become a union $\sigma$ of cells $\overline{UD}_n$. Except the critical cell corresponding to $c$ in $UD_n\Gamma'$, all other cells in $\bar i(\sigma)$ inevitably contain an order-respecting edge preceded by an unblocked vertex and so they are redundant in $UD_n\Gamma'$.
Thus the isomorphism $:\pi_1(UD_n\Gamma)\to\pi_1(UD_n\Gamma')$ established above is induced from the composite of the identity map $:UD_n\Gamma\to\overline{UD}_n$ and the inclusion $\bar i:\overline{UD}_n\to UD_n\Gamma'$.

Since both $\overline{UD}_n$ and $UD_n\Gamma'$ are $K(\pi,1)$-spaces, $\bar i$ is a homotopy equivalence.
As Abrams showed in \cite{Ab}, $UD_n\Gamma'$ is a deformation retract of $UC_n\Gamma$.
Thus the composition of two inclusions $\bar i:\overline{UD}_n\hookrightarrow UD_n\Gamma'$ and $UD_n\Gamma'\hookrightarrow UC_n\Gamma$ is a homotopy equivalence. Since involved cubical complexes are all finite, this composition has the homotopy extension property and so $\overline{UD}_n$ is a deformation retract of $UC_n\Gamma$.
Since $\overline{UD}_n$ and $UD_n\Gamma$ share the same underlying space, we complete a proof for the unordered case.

Recall the covering $D_n\Gamma\to UC_n\Gamma=D_n\Gamma/S_n$ for the symmetric group $S_n$. Let $\bar p:\overline{D}_n\to\overline{UD}_n$ be the covering with the deck transformation group $S_n$. It is clear that the cover $\overline{D}_n$ can be obtained by subdividing $D_n\Gamma$ in a similar way to subdividing $UD_n\Gamma$ to obtain $\overline{UD}_n$ so that there is a cubical embedding $\tilde i: \overline{D}_n\to D_n\Gamma'$. Then we naturally have $\bar i\bar p=p'\tilde i$ where $p':D_n\Gamma'\to UD_n\Gamma'$ is the $S_n$-covering. Then $\tilde i_*:\pi_1(\overline{D}_n)\to\pi_1(D_n\Gamma')$ is injective since $\bar i_*$ is an isomorphism. Considering indices of subgroups, we have
$$[\pi_1(UD_n\Gamma):\im p'_*][\pi_1(D_n\Gamma'):\im\tilde i_*]=[\pi_1(UD_n\Gamma'):\im\bar i_*][\pi_1(\overline{UD}_n):\im\bar p_*]$$ and so $[\pi_1(D_n\Gamma'):\im\tilde i_*]=1$. Thus $\tilde i_*$ is an isomorphism. Since both $\overline{D}_n$ and
$D_n\Gamma'$ are $K(\pi,1)$-spaces, $\tilde i$ is a homotopy equivalence. By a similar argument to the unordered case, we obtain the ordered case.
\end{proof}

It is clear from the above proof that the existence of restrictive pairs is not only necessary but also sufficient to guarantee $K'_2-K_2\ne\emptyset$ and $K'_1-K_1\ne\emptyset$ perhaps except the case $n=2$. It is also clear that if each path in $\Gamma$ joining two vertices of valency $\ge 3$ contains at least $n$ edges, then $K'_2-K_2=K'_1-K_1=\emptyset$.

\subsection{Proof of Theorem A}\label{ss23:thma}
We say that a graph $\Gamma$ \emph{contains} another graph $\Gamma'$
if a subdivision of $\Gamma'$ is a subgraph of a subdivision of
$\Gamma$.

\begin{thma}
If a graph $\Gamma$ does not contain $T_0$ and $S_0$ in
Figure~\ref{fig1}, then $B_n(\Gamma)$ is a right-angled Artin group
for any braid index.
\end{thma}

\begin{proof}
Suppose that a graph $\Gamma$ contains neither $T_0$ nor $S_0$.
Since $\Gamma$ does not contain $T_0$, a maximal tree of $\Gamma$
must be a linear tree, that is, the tree contains a simple path
containing all vertices of valency $\ge 3$. Moreover since $\Gamma$
does not contain $S_0$, every circuit cannot contain more than one
vertex of valency $\ge 3$ and so circuits form bouquets. Thus
$\Gamma$ must look like a graph in Figure~\ref{fig6} that we call a
{\em linear star-bouquet}.

\begin{figure}[ht]
\centering
\includegraphics[height=1.8cm]{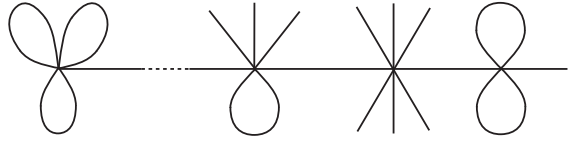}
\caption{Graph $\Gamma$ with no $T_0$ nor $S_0$}\label{fig6}
\end{figure}

If we choose a maximal tree of $\Gamma$ and give an order as in
Figure~\ref{fig7}, the complex $UD_n\Gamma$ has two kinds of
critical 1-cells: $$d_p^i\cup A^i(\vec a),\quad A_k^i(\vec a)$$ and
four kinds of critical 2-cells: For $i<j$
$$ A_k^i(\vec a)\cup A_\ell^j(\vec b),\quad d_p^i\cup d_q^j\cup A^i(\vec a)\cup A^j(\vec b),\quad A_k^i(\vec a)\cup d_q^j\cup A^j(\vec b),\quad d_p^i\cup A_\ell^j(\vec b)\cup A^i(\vec a).$$
Here $A^i$ is the $i$-th vertex of valency $\ge 3$ in the maximal
tree and $d_j^i$ is the $j$-th of deleted edges ends at the vertex
$A^i$. See Figure~\ref{fig7}.

\begin{figure}[ht]
\psfrag{0}{\small0}
\psfrag{A1}{\small$A^1$}
\psfrag{A2}{\small$A^2$}
\psfrag{A3}{\small$A^3$}
\psfrag{Am}{\small$A^m$}
\psfrag{e11}{\small$d^1_1$}
\psfrag{e12}{\small$d^1_2$}
\psfrag{e31}{\small$d^3_1$}
\psfrag{em1}{\small$d^m_1$}
\psfrag{em2}{\small$d^m_2$}
\psfrag{em3}{\small$d^m_3$}
\centering
\includegraphics[height=2cm]{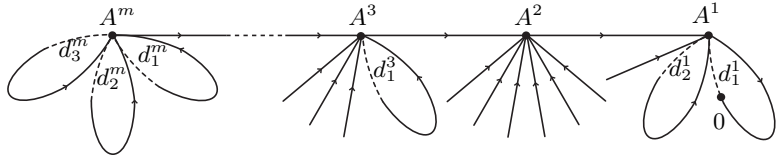}
\caption{Maximal tree of $\Gamma$}\label{fig7}
\end{figure}

Relations of $B_n\Gamma$ are obtained by rewriting boundary words of
critical 2-cells as words in critical 1-cells. Let $\mu$ be the
valency of $A^i$ in the maximal tree. The boundary word of the
critical 2-cell $A_k^i(\vec a)\cup A_\ell^j(\vec b)$ is given by the
product of 1-cells $(A_k^i(\vec a)\cup A^j(\vec b))((\dot A^i(\vec
a)-A^i(\vec\delta_k))\cup A_\ell^j(\vec b))(A_k^i(\vec a)\cup (\dot
A^j(\vec b)-A^j(\vec\delta_\ell)))^{-1}(A^i(\vec a)\cup A_\ell^j(\vec
b))^{-1}$ where $\vec\delta_k$ denotes the $k$-th coordinate unit
vector. Since
$$\tilde r(A^i(\vec a)\cup A_\ell^j(\vec b))=\tilde r((\dot A^i(\vec a)-A^i(\vec\delta_k))\cup A_\ell^j(\vec b))=
A_\ell^j(\vec b),$$
$$\mbox{and }\tilde r(A_k^i(\vec a)\cup
A^j(\vec b))=\tilde r(A_k^i(\vec a)\cup
(\dot A^j(\vec b)-A^j(\vec\delta_\ell)))=A_k^i(\vec a+
|\vec b|\vec\delta_{\mu-1}),$$ the relation becomes
$$[A_k^i(\vec a+|\vec b|\vec\delta_{\mu-1}),
A_\ell^j(\vec b)]$$ which is a commutator of critical 1-cells.
Similarly, we can obtain relations from the boundary words of the
remaining three types of critical 2-cells. Then they are all
commutators of critical 1-cells as follows:
\begin{align*}&[d_p^i\cup
A^i(\vec a+(|\vec b|+1)\vec\delta_{\mu-1}), d_q^j\cup A^j(\vec b)],\\
&[A_k^i(\vec a+(|\vec b|+1)\vec\delta_{\mu-1}), d_q^j\cup A^j(\vec b)],\quad [d_p^i\cup A^i(\vec a+|\vec b|\vec\delta_{\mu-1}), A_\ell^j(\vec b)].
\end{align*}
Consequently, $B_n\Gamma$ is a right-angled Artin group.
\end{proof}

\section{Graphs containing $T_0$ but not $S_0$}\label{s:three}
Let $K$ be a finite simplicial complex, then \emph{the exterior face
algebra $\Lambda[K]$ of $K$ over a given field} is defined by
$$\Lambda[K]=\Lambda[v_1,\ldots, v_m]/I.$$
where $\Lambda[v_1,\ldots, v_m]$ is the exterior algebra generated
by vertices of $K$ over a field and $I$ is the ideal generated by
the products $v_{i_1}\cdots v_{i_\ell}$ such that
$\{v_{i_1},\ldots,v_{i_\ell}\}$ does not form a simplex in $K$.
The field $\mathbb Z_2$ will be used in this article. In this case, the algebra $\Lambda[K]$ can be regarded as a Stanley-Reisner ring and it was
shown by Bruns and Gubeladze in \cite{BG} that $\Lambda[K]$ completely determines
a finite simplicial complex $K$. That is, if
$\Lambda[K]$ and $\Lambda[K']$ are isomorphic for finite simplicial
complexes $K$ and $K'$, then there is a simplicial homeomorphism
between $K$ and $K'$.

A simplicial
complex $K$ is called a \emph{flag complex} if every complete
subgraph in 1-skeleton spans a simplex in $K$. The following
proposition provides our guiding principle in this section.

\begin{pro}\label{pro:flag}
{\rm(1)(Charney-Davis \cite{CW1})}
 Let $G$ be a right-angled Artin group. The Eilenberg-MacLane
space $K(G,1)$ is a full subcomplex of a high-fold torus. In
particular, the cohomology algebra $H^*(G;\mathbb Z_2)$ is an
exterior face
algebra of a flag complex.\\
{\rm(2)(Farley-Sabalka \cite{FS2})} Let $K$ and $K'$ be finite
simplicial complexes. If $\phi : \Lambda(K) \to \Lambda(K')$ is a
degree-preserving surjection, $K$ is a flag complex and $\ker(\phi)$
is generated by homogeneous elements of degree 1 and 2, then $K'$ is
also a flag complex.
\end{pro}

Let $\Gamma$ and $\Gamma'$ be graphs such that $\Gamma$ contains
$\Gamma'$. If homology groups $H_k(B_n\Gamma)$ and
$H_k(B_{n'}\Gamma')$ are torsion free for $k\ge 0$ and $n'\le n$ and
there exists an embedding $i:\Gamma'\to \Gamma$ which induces an
injection $i_*:H_k(B_{n'}\Gamma')\to H_k(B_n\Gamma)$ sending
generators to generators, then the induced map
$i^*:H^*(B_n\Gamma;\mathbb Z_2)\to H^*(B_{n'}\Gamma';\mathbb Z_2)$
is surjective by the universal coefficient theorem. And if
$H^*(B_{n'}\Gamma';\mathbb Z_2)$ is an exterior face algebra of a
non-flag complex and the induced map $i^*$ is degree-preserving and
$\ker(i^*)$ is generated by homogeneous elements of degree 1 and 2,
then $H^*(B_n\Gamma;\mathbb Z_2)$ is not an exterior face algebra of
a flag complex by Proposition~\ref{pro:flag}(2). So $B_n\Gamma$ is
not a right-angled Artin group by Proposition~\ref{pro:flag}(1). In
the case that $\Gamma'=T_0$ and $\Gamma$ is a tree $T$ containing
$T_0$, Farley and Sabalka employed this argument in \cite{FS2} to
prove that $B_nT$ is not a right-angled Artin group for $n\ge 4$. We
try to do the same for the case that $\Gamma'=T_0$ and $\Gamma$ is a
graph containing $T_0$. There are several difficulties as we now
explain.

If a graph contains $S_0$ as well as $T_0$ like the graph $T'$ given
in Figure~\ref{fig8} then one of hypotheses of the above argument
fails to hold. In fact, one can show that there
is no embedding $i:T_0\to T'$ which induces an injection
$i_*:H_1(B_4T_0)\to H_1(B_nT')$ for $n\ge 4$. For a graph $\Gamma$
containing $S_0$, we introduce some other argument in \S\ref{s:four}
to prove that $B_n\Gamma$ is not a right-angled Artin group.

Let $\mathcal G$ be the set consisting of graphs that contain $T_0$
but do not contain $S_0$. Let $T_1$, $T_2$ and $T_3$ be graphs
depicted in Figure~\ref{fig8}. Then they are graphs in $\mathcal G$.
One can also show that for each $k=0,1,2$, there is
no embedding $T_k\to T_{k+1}$ which induces an injection
$H_1(B_4T_k)\to H_1(B_nT_{k+1})$ for $n\ge4$. In order to apply the
above argument, we further need to divide $\mathcal G$ into four
subclasses: For each $k=0,1,2$, $\mathcal G_k$ consists of graphs in
$\mathcal G$ that contain $T_k$ but do not contain $T_{k+1}$ and
$\mathcal G_3$ consists of graphs containing $T_3$ but not
containing $S_0$.

\begin{figure}[ht]
\psfrag{S1}{\small$T'$}
\psfrag{T1}{\small$T_1$}
\psfrag{T2}{\small$T_2$}
\psfrag{T3}{\small$T_3$}
\psfrag{T4}{\small$T_3'$}
\psfrag{T5}{\small$T_3''$}
\psfrag{T6}{\small$T_3'''$}
\centering
{\includegraphics[height=4.5cm]{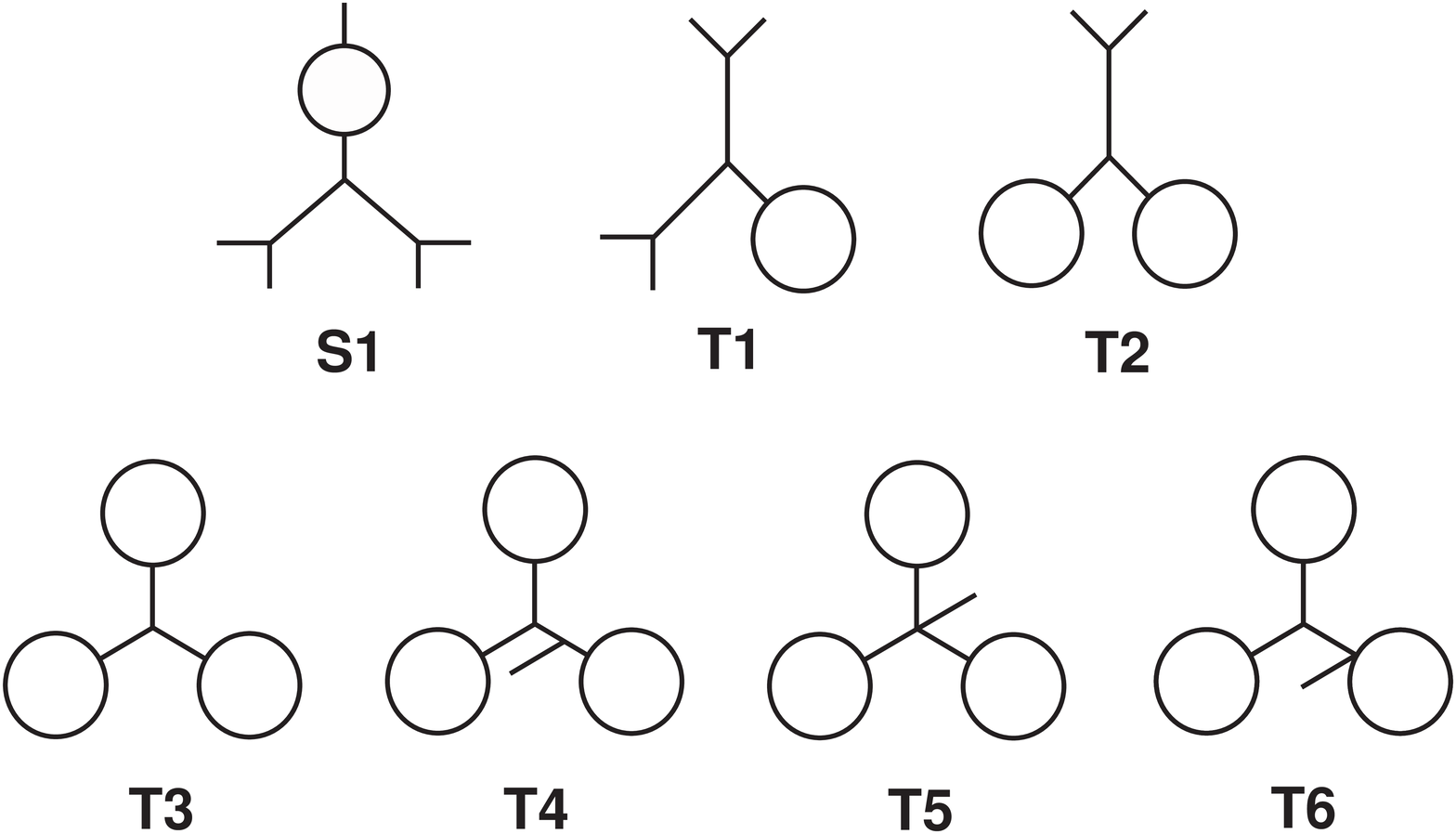}}
\caption{}\label{fig8}
\end{figure}

However, we still have one technical problem. The graph $T_3$ itself
is in $\mathcal G_3$. In order to show $B_nT_3$ is not a
right-Angled Artin group for $n\ge 2$, we need to find a map
$h:T_3\to T_3$ that induces an injection $h_*:H_*(UD_2T_3)\to
H_*(UD_nT_3)$ for $n>2$. The injectivity of $h_*$ on homologies
forces $h$ to be surjective due to the absence of vertices of
valency 1 in $T_3$. Then it is impossible for a surjection $h$ to
induce a cubical map $UD_2T_3\to UD_nT_3$ for $n>2$ since there is
no room to place extra $n-2$ vertices. In \S\ref{ss34:T_3}, we will
show $B_nT_3$ is not a right-angled Artin group for $n\ge2$ by using
some other method. And the subclass $\mathcal G_3$ is modified so
that it consists of graphs that contain $T_3'$, $T_3''$, or $T_3'''$
depicted in Figure~\ref{fig8} but do not contain $S_0$. Then
$\mathcal G$=$\mathcal G_0\cup\mathcal G_1\cup\mathcal
G_2\cup\mathcal G_3\cup\{T_3\}$. The goal of this section is to
prove the following:

\begin{thm}
Let $\Gamma$ be a graph in $\mathcal G$, that is, it contains $T_0$
and does not contain $S_0$. Then $B_n\Gamma$ is not a right-angled
Artin group for $n\ge 4$.
\end{thm}

Except for $T_3$, this theorem will be proved by the argument
mentioned above. In \S\ref{ss31:homology}, we prove that for
$\Gamma\in\mathcal G$, the homology groups of the Morse complex of
$UD_n\Gamma$ are free abelian groups generated by critical
$k$-cells. In \S\ref{ss32:cohomology}, we prove that the cohomology
algebras of $UD_4T_1$, $UD_3T_2$, $UD_2T_3'$, $UD_2T_3''$, and
$UD_2T_3'''$ are exterior face algebras of non-flag complexes. In
\S\ref{ss33:embedding}, we prove that for $\Gamma\in\mathcal G_k$
for some $k=0,1,2,3$, there is an embedding $i:P_k\to \Gamma$ which
induce a surjection $i^*$ on cohomology algebras such that $i^*$ is
degree-preserving and $\ker(i^*)$ is generated by homogeneous
elements of degree 1 and 2, where $P_k=T_k$ for $k=0,1,2$ and $P_3$
is one of $T_3'$, $T_3''$, and $T_3'''$.

\subsection{Homologies of $UD_n\Gamma$}\label{ss31:homology}
Let $\Gamma$ be a graph in $\mathcal G$. Since $\Gamma$ does not
contain $S_0$, every circuit cannot contain more than one vertex of
valency $\ge 3$ and so circuits form bouquets. Moreover since
$\Gamma$ contains $T_0$ but not $S_0$, the maximal tree of $\Gamma$
is non-linear tree. Thus $\Gamma$ looks like a graph in
Figure~\ref{fig9}(a).

\begin{figure}[ht]
\centering
\psfrag{0}{\small0}
\subfigure[Graph containing $T_0$ but not $S_0$]
{\includegraphics[height=4cm]{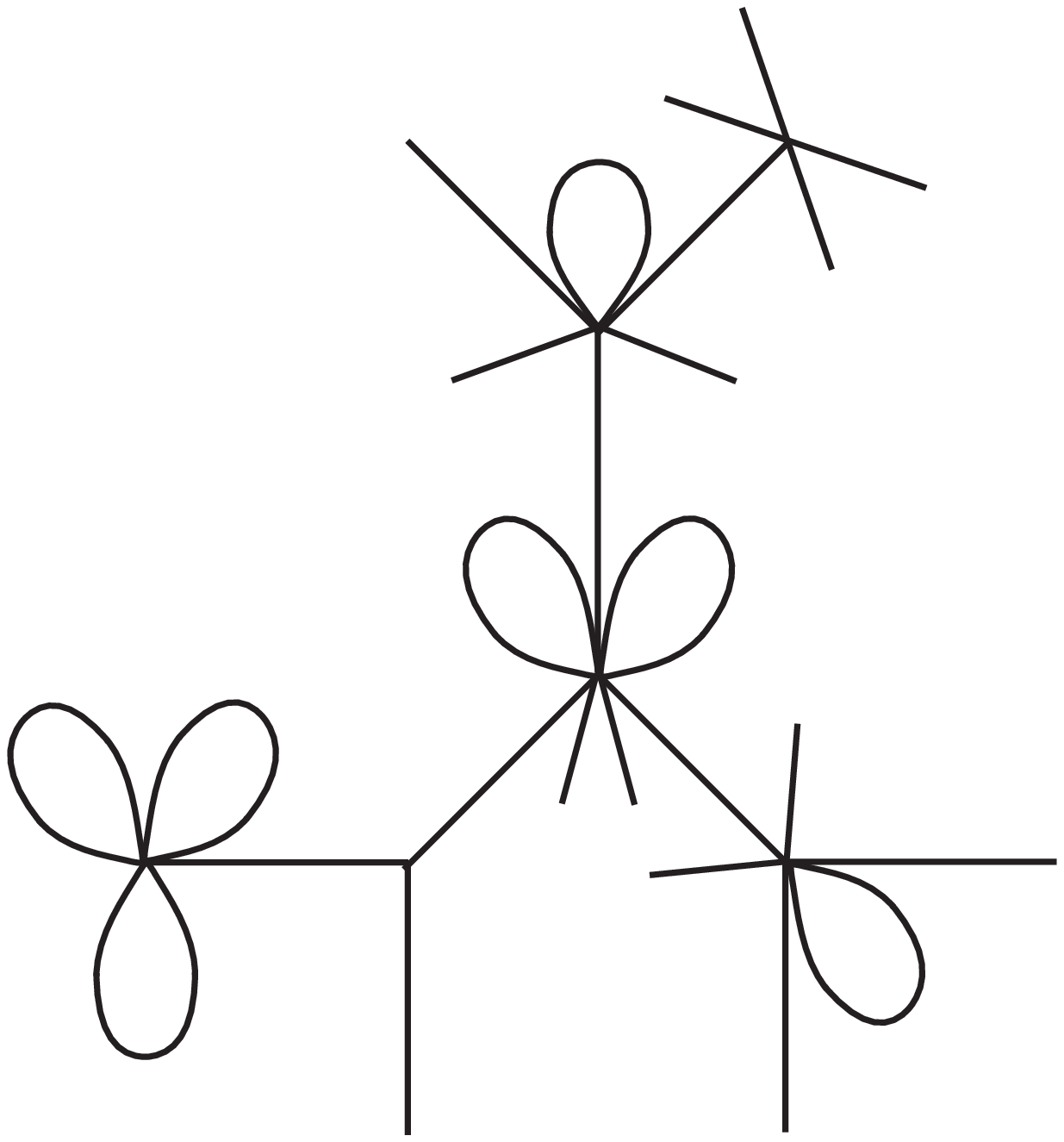}}\qquad
\subfigure[Maximal tree and order]
{\includegraphics[height=4cm]{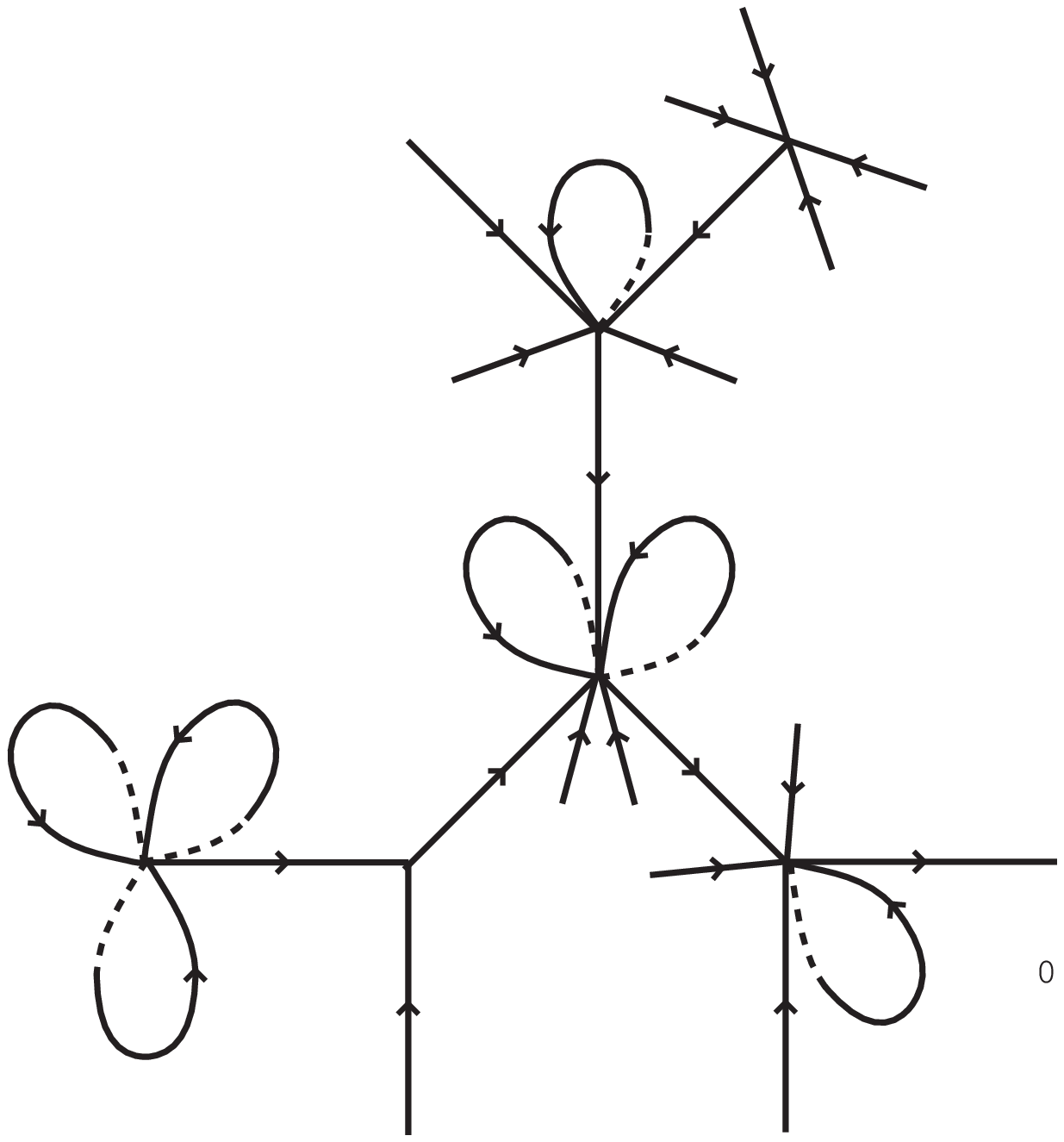}}
\caption{Graph containing $T_0$ but not $S_0$}\label{fig9}
\end{figure}

Given a graph $\Gamma$, we want to compute the homology of the braid
group $B_n\Gamma$ which is the same as the homology of $UD_n\Gamma$.
So we may use its Morse complex instead. Let
$(C_i(UD_n\Gamma),\partial)$ be the (cubical) cellular chain complex
of $UD_n\Gamma$. A choice of a maximal tree and an order on it
determine a Morse complex, as described in \S\ref{ss22:morse}. Let
$M_i(UD_n\Gamma)$ be the free abelian group generated by critical
$i$-cells.  Let $c=\{e_1,e_2,\ldots,e_i,v_{i+1},\ldots,v_n\}$ be an
$i$-cell of $UD_n\Gamma$ where $e_1,\ldots,e_i$ are edges with
$\iota(e_1)<\iota(e_2)<\cdots<\iota(e_i)$ and $v_{i+1},\ldots,v_n$
are vertices of $\Gamma$.
$$\partial(c)=\sum_{k=1}^i(-1)^k(\partial_k^\iota(c)-\partial_k^\tau(c))$$
where
$$\partial_k^\iota(c)=\{e_1, \ldots, e_{k-1}, e_{k+1}, \ldots,
e_i,v_{i+1},\ldots,v_n,\iota(e_k)\},$$
$$\partial_k^\tau(c)=\{e_1, \ldots, e_{k-1}, e_{k+1}, \ldots, e_i,
v_{i+1},\ldots,v_n, \tau(e_k)\}.$$

Consider an abelianized version of the rewriting map $\tilde r$ in
\S\ref{ss22:morse}. Let $R:C_i(UD_n\Gamma)\allowbreak\to
C_i(UD_n\Gamma)$ be a homomorphism defined by $R(c)=0$ if $c$ is a
collapsible $i$-cell, by $R(c)=c$ if $c$ is critical, and by
$R(c)=\pm
\partial W(c)+c$ if $c$ is redundant where the sign is chosen so
that the coefficient of $c$ in $\partial W(c)$ is $-1$. By
\cite{For}, there is a nonnegative integer $m$ such that
$R^m=R^{m+1}$ and let $\widetilde R=R^m$. Then $\widetilde R(c)$ is
in $M_i(UD_n\Gamma)$ and we have a homomorphism $\widetilde R :
C_i(UD_n\Gamma)\to M_i(UD_n\Gamma)$. This definition is little
different from the similar map defined by Forman in~\cite{For}. But
it easy to see that two maps are essentially the same. Define a map
$\widetilde\partial:M_i(UD_n\Gamma)\to M_{i-1}(UD_n\Gamma)$ by
$\widetilde\partial(c) = \widetilde R\partial(c)$. Then
$(M_i(UD_n\Gamma),\widetilde\partial)$ forms a chain complex.
However, the inclusion $M_*(UD_n\Gamma)\hookrightarrow
C_*(UD_n\Gamma)$ is not a chain map. Instead, consider a
homomorphism $\varepsilon:M_i(UD_n\Gamma)\to C_i(UD_n\Gamma)$
defined as follows: For a (critical) $i$-cell $c$, $\varepsilon(c)$
is obtained from $c$ by minimally adding collapsible $i$-cells until
it becomes closed in the sense that for each redundant $(i-1)$-cell
$c'$ in the boundary of every $i$-cell summand in $\varepsilon(c)$,
$W(c')$ already appears in $\varepsilon(c)$. Then $\varepsilon$ is a
chain map that is a chain homotopy inverse of $\widetilde R$. Thus
$(M_i(UD_n\Gamma),\widetilde\partial)$ and
$(C_i(UD_n\Gamma),\partial)$ have the same chain homotopy type and
so their homology groups are isomorphic.

It should be more efficient to compute the homology groups of
$UD_n\Gamma$ by using its Morse complex if we appropriately choose a
maximal tree and an order on $\Gamma$ in $\mathcal G$. The following
lemma that generalizes the corresponding Farley's result in
\cite{Far}, is useful in the sense that it significantly reduces the
amount of computation to evaluate the rewriting map $\widetilde R$.

\begin{lem}\label{lem:redundant1}
Let $\Gamma$ be a graph. If a cell $c$ in $UD_n\Gamma$ contains an
order-respecting edge $e$ satisfying:
\begin{itemize}
\item[(1)] Every vertex $v$ in $c$ with $\tau(e)<v<\iota(e)$ is blocked;
\item[(2)] For any edge $e'$ in $c$ with $\tau(e)<\iota(e')<\iota(e)$, $\tau(e)<\tau(e')<\iota(e)$,
\end{itemize}
then $\widetilde R(c)=0$.
\end{lem}

\begin{proof}
The cell $c$ can be either collapsible or redundant. If $c$ is
collapsible, then $R(c)=0$ and we are done. So we assume that $c$ is
redundant. Then $c$ has the smallest unblocked vertex $v_s$. Let
$e_s$ be the edge in $\Gamma$ such that $\iota(e_s)=v_s$. We use
induction on the minimal integer $k$ such that $R^k(c)=R^{k+1}(c)$.
Since $c$ is redundant, $k\ge2$. Suppose that $k=2$. Then each cell
$c'$ occurring in $R(c)=\pm\partial W(c) +c$ contains at least one
of $e$ and $e_s$. The condition (1) and (2) guarantee that $e$
remains order-respecting. If $c'$ does not contain $e$, $c'$
contains $e_s$ and $e_s$ must be order-respecting in this case by
the choice of $v_s$. Since $k=1$, $c'$ can no longer be redundant.
Thus $c'$ is collapsible and so $R^2(c)=0$.

Suppose that $k \ge 3$. Each cell $c'$ occurring in $R(c)$ also
contains at least one of $e$ and $e_s$. The condition (1) and (2)
guarantee that $e$ remains order-respecting and satisfies (1) and
(2). If $c'$ contains $e_s$ and every vertex $v$ in $c'$ with
$\tau(e_s)<v<\iota(e_s)$ is blocked, then $c'$ is collapsible. If
$c'$ contains $e_s$ and an unblocked vertex $v$ in $c'$ with
$\tau(e_s)<v<\iota(e_s)$, then the unblocked vertex is newly formed
and $c'$ contains $e$ that is order-respecting and still satisfies
the conditions (1) and (2) because no edges lying between
$\tau(e_s)$ and $\iota(e_s)$ are touched. Obviously, $c'$ stabilizes
at least one-application faster than $c$ does under iteration of
$R$. By induction, $\widetilde R(c')=0$. If $c'$ does not contain
$e_s$ then $c'$ contains the order-respecting edge $e$ that
satisfies the conditions (1) and (2) because no edges lying between
$\tau(e_s)$ and $\iota(e_s)$ are untouched. This also implies
$\widetilde R(c')=0$ by induction. Since $\widetilde R(c')=0$ for
each cell $c'$ occurring in $R(c)$, we have $\widetilde R(c)=0$.
\end{proof}

For a graph $\Gamma$, let $c=\{c_1,\ldots,c_{n-1},v_s\}$ be an
$i$-cell in $UD_n\Gamma$. Define a function $V:K_i\to K_i$ by
$V(c)=\{c_1,\ldots,c_{n-1},\tau(e_s)\}$ if $c$ is redundant, and
$v_s$ is the smallest unblocked vertex in $c$, and $e_s$ is the edge
of $\Gamma$ such that $\iota(e_s)=v_s$ and by $V(c)=c$ otherwise.
The function $V$ should stabilize to a function $\widetilde V:K_i\to
K_i$ under iteration, that is, $\widetilde V=V^k$ for some
non-negative integer $k$ such that $V^k=V^{k+1}$.

\begin{lem}\label{lem:v}
Let $\Gamma$ be a graph not containing $S_0$. Suppose that we choose
a maximal tree of $\Gamma$ and give an order as in
Figure~\ref{fig9}(b). Let $c$ be a cell in $UD_n\Gamma$ given by
$c=\{c_{1},\ldots,c_{n-1},e\}$ where $e$ is an edge and each $c_i$
are either a vertex or an edge. Then $\widetilde R\widetilde
V(c^\iota)=\widetilde R\widetilde V(c^\tau)$ where
$c^\iota=\{c_{1},\ldots,c_{n-1},\iota(e)\}$ and
$c^\tau=\{c_{1},\ldots,c_{n-1},\tau(e)\}$.
\end{lem}
\begin{proof}
If $e$ is not a deleted edge or if $e$ is a deleted edge and there
are no edges between $\iota(e)$ and $\tau(e)$, then $\iota(e)$
eventually becomes the smallest unblocked vertex and moves via the
original location of $\tau(e)$ under iteration of $V$ and so
$\widetilde V(c^\iota)=\widetilde V(c^\tau)$.

If $e$ is a deleted edge and there is an edge $e'$ between
$\iota(e)$ and $\tau(e)$, then due to our choice of a maximal tree
and an order, $e'$ is an order-respecting edge satisfying the
hypothesis of Lemma~\ref{lem:redundant1} and so $\widetilde
R\widetilde V(c^\iota)=\widetilde R\widetilde V(c^\tau)=0$
\end{proof}

Farley proved in \cite{Far} that
$\widetilde\partial:M_i(UD_n\Gamma)\to M_{i-1}(UD_n\Gamma)$ is zero
for any tree $\Gamma$. We now try to extend this result to graphs
containing $T_0$ but not $S_0$.

\begin{lem}\label{lem:redundant}
Let $\Gamma$ be a graph not containing $S_0$. Suppose that we choose
a maximal tree of $\Gamma$ and give an order as in
Figure~\ref{fig9}(b). Then $\widetilde R(c)=\widetilde RV(c)$ for a
redundant cell $c$ in $UD_n\Gamma$. Consequently, $\widetilde R=
\widetilde R\widetilde V$.
\end{lem}

\begin{proof}
Suppose that $v_s$ is the smallest unblocked vertex in $c$ and $e_s$
is an edge of $\Gamma$ such that $\iota(e_s)=v_s$. Then
$R(c)=\pm\partial W(c)+c=\sum_e (c_e^\iota-c_e^\tau)+V(c)$, where
$c_e^\iota$ and $c_e^\tau$ are cells obtained from $W(c)$ by
replacing an edge $e\ne e_s$ in $W(c)$ by $\iota(e)$ and $\tau(e)$,
respectively. It suffices to show that $\widetilde
R(c_e^\iota-c_e^\tau)=0$ for each edge $e\ne e_s$ in $c$. We use
induction on the minimal integer $k$ such that $R^k(c)=R^{k+1}(c)$.
Since $c$ is redundant, $k\ge2$. Suppose that $k=2$. Then both
$c^\iota_e$ and $c^\tau_e$ must be collapsible and so $\widetilde
R(c_e^\iota)=\widetilde R(c_e^\tau)=0$.

Suppose that $k\ge 3$. Among six possible relative positions of
edges $e$ and $e_s$, $\tau(e_s)<\tau(e)<\iota(e_s)<\iota(e)$ and
$\tau(e)<\tau(e_s)<\iota(e)<\iota(e_s)$ are impossible due to our
choice of the maximal tree and the order of $\Gamma$ as in
Figure~\ref{fig9}(b). If $\tau(e_s)<\iota(e_s)<\tau(e)<\iota(e)$,
$\tau(e)<\tau(e_s)<\iota(e_s)<\iota(e)$, or
$\tau(e)<\iota(e)<\tau(e_s)<\iota(e_s)$, then both $c_e^\iota$ and
$c_e^\tau$ have the order-respecting edge $e_s$ satisfying the
hypothesis of Lemma~\ref{lem:redundant1} and so $\widetilde
R(c_e^\iota)=\widetilde R(c_e^\tau)=0$.

In the remaining case $\tau(e_s)<\tau(e)<\iota(e)<\iota(e_s)$, both
$c^\iota_e$ and $c^\tau_e$ are redundant and stabilize faster than
$c$ does. By induction, $\widetilde R(c^\iota_e)=\widetilde R
V(c^\iota_e)$ and $\widetilde R(c^\tau_e)=\widetilde R V(c^\tau_e)$.
If $V(c^\iota_e)$ has a unblocked vertex $v$ such that
$\tau(e_s)<v<\iota(e_s)$, then $V(c^\iota_e)$ is redundant and so
$\widetilde RV(c^\iota_e)=\widetilde R V^2(c^\iota_e)$ by induction.
Thus $\widetilde R V^2(c^\iota_e)=\widetilde R(c^\iota_e)$. Let $m$
be the smallest integer such that every vertex $v$ in
$V^m(c^\iota_e)$ with $\tau(e_s)<v<\iota(e_s)$ is blocked. Then
$\widetilde R(c^\iota_e)=\cdots=\widetilde
RV(V^{m-1}(c^\iota_e))=\widetilde R V^m(c^\iota_e)=\widetilde R
\widetilde V(c^\iota_e)$ by induction. Similarly, $\widetilde
R(c^\tau_e)=\widetilde R \widetilde V(c^\tau_e)$. Thus $\widetilde
R(c^\iota_e)= \widetilde R\widetilde V(c^\iota_e)=\widetilde R
\widetilde V (c^\tau_e)= \widetilde R(c^\tau_e)$ by
Lemma~\ref{lem:v}.
\end{proof}

\begin{thm}\label{thm:homology}
Let $\Gamma$ be a graph not containing $S_0$. Then there is a Morse
complex $(M_*(UD_n\Gamma),\widetilde\partial)$ whose boundary maps
are all zero and so the $i$-th homology $H_i(M_*(UD_n\Gamma))$ is
the free abelian group over critical $i$-cells.
\end{thm}

\begin{proof}
Suppose that we choose a maximal tree of $\Gamma$ and give an order
as in Figure~\ref{fig9}(b). Let $c$ be a critical $i$-cell of
$UD_n\Gamma$. The boundary homomorphism $\widetilde\partial$ of the
Morse chain complex is given by $\widetilde\partial(c)=\widetilde
R(\sum_{k=1}^i(-1)^k(\partial_k^\iota(c)-\partial_k^\tau(c)))$. By
Lemma~\ref{lem:v} and Lemma~\ref{lem:redundant}, $\widetilde
R(\partial_k^\iota(c))=\widetilde R\widetilde
V(\partial_k^\iota(c))=\widetilde R\widetilde
V(\partial_k^\tau(c))=\widetilde R(\partial_k^\tau(c))$. Thus
$\widetilde R(\partial_k^\iota(c)-\partial_k^\tau(c))=0$ for each
$k$ and so $\widetilde
\partial(c)=0$.
\end{proof}

Let $\Gamma$ and $\Gamma'$ be graphs. We are interested in an
embedding $i:\Gamma'\to\Gamma$ of graphs that induces homomorphisms
$\tilde i_*:H_*(M_*(UD_{n'}\Gamma'))\to H_*(M_*(UD_n\Gamma))$ for
$n'\le n$. If $\Gamma'$ has a vertex of valency 1, then we can
choose maximal trees $T$ and $T'$ of $\Gamma$ and $\Gamma'$,
respectively and give an order on  $\Gamma$ and $\Gamma'$ as
described in \S\ref{ss22:morse} so that they satisfy the following
conditions:

\begin{itemize}
\item[(1)] The base vertex $0'$ of $T'$ have valency one in $\Gamma'$;
\item[(2)] The edge path joining $i(0')$ and the base vertex 0 in $T$ passes through exactly $n-n'$ edges, none of which are images of edges of
$\Gamma'$;
\item[(3)] The order is preserved under $i$, that is, $v_1<v_2$ iff
$i(v_1)<i(v_2)$ for all vertices $v_1,v_2$ in $\Gamma'$;
\item[(4)] The image of an edge $e$ on $\Gamma'$ is a deleted edge on $\Gamma$
if and only if $e$ is a deleted edge on $\Gamma'$.
\end{itemize}
By the conditions (1) and (2), the embedding $i$ naturally induces
an embedding $i:UD_{n'}\Gamma'\to UD_n\Gamma$ for $n'\le n$ defined
by$$i(\{\sigma_1,\sigma_2,\ldots,\sigma_{n'}\})=
\{i(\sigma_1),i(\sigma_2),\ldots,i(\sigma_{n'})\}\cup 0_{n-n'}.$$ So
$i$ preserves $i$-cells and in fact, it is a chain map of cubical
chain complexes of $UD_{n'}\Gamma'$ and $UD_n\Gamma$. Moreover, by
the conditions (1), (2), and (3), $i$ and $W$ commute. So $i$
preserves critical $i$-cells and it induces a chain map $\tilde i$
between Morse complexes, that is, $\widetilde\partial\tilde i=\tilde
i\widetilde\partial :M_i(UD_{n'}\Gamma')\to M_{i-1}(UD_n\Gamma)$.
Thus the diagram
\begin{displaymath}
\xymatrixcolsep{0pc}\xymatrixrowsep{1pc} \xymatrix{
 & M_i(UD_{n'} \Gamma') \ar@{->}[dl]_{\tilde i} \ar@{->}[rr]^{\widetilde\partial'}\ar@<.5ex>[dd]^(.3){\varepsilon}|\hole & & M_{i-1}(UD_{n'} \Gamma')\ar@{->}[dl]_{\tilde i}\ar@<.5ex>[dd]^{\varepsilon}\\
 M_i(UD_n\Gamma)\ar@{->}[rr]^(.6){\widetilde\partial}\ar@<.5ex>[dd]^{\varepsilon} & & M_{i-1}(UD_n\Gamma) \ar@<.5ex>[dd]^(.3){\varepsilon}\\
 & C_i(UD_{n'}\Gamma')\ar@<.5ex>[uu]^(.3){\widetilde R}|\hole\ar@{->}[dl]_i \ar@{->}'[r]^-{\partial'}[rr] & & C_{i-1}(UD_{n'}\Gamma')\ar@<.5ex>[uu]^{\widetilde R}\ar@{->}[dl]_i\\
 C_i(UD_n\Gamma) \ar@<.5ex>[uu]^{\widetilde R}\ar@{->}[rr]^{\partial} & & C_{i-1}(UD_n\Gamma)\ar@<.5ex>[uu]^(.3){\widetilde R}}
\end{displaymath}
commutes. Since the chain map $\varepsilon:M_*(UD_n\Gamma)\to
C_*(UD_n\Gamma)$ induces an isomorphism between homologies of two
chain complexes, the diagram

\begin{displaymath}
\begin{CD}
H_*(M_*(UD_{n'}\Gamma'))@>{\tilde i_*}>>H_*(M_*(UD_n\Gamma))\\
@A{\cong}A{\widetilde R_*}A @V{\cong}V{\varepsilon_*}V\\
H_*(C_*(UD_{n'}\Gamma'))@>{i_*}>>H_*(C_*(UD_n\Gamma))\\
\end{CD}
\end{displaymath}
commutes. Although the homomorphism $\tilde i_*$ may be intractable
in general, it is manageable if we appropriately choose maximal
trees and orders for a given embedding $i:\Gamma'\to\Gamma$. Since
$i_*=\varepsilon_*\tilde i_*\widetilde R_*$, many properties such as
injectivity are shared by $i_*$ and $\tilde i_*$, and moreover they
do not depend on the choice of maximal trees and orders.

\subsection{Cohomology algebra of graphs without $S_0$}\label{ss32:cohomology}
In~\cite{FS2}, Farley and Sabalka employed an indirect construction
to compute the cohomology algebra of the braid group of a tree with
coefficients in $\mathbb Z_2$. They defined an equivalence relation
on the unordered configuration space of a tree so that its quotient
space is a subcomplex of a high-fold torus and the epimorphism of
cohomology algebras induced from the quotient map carries the
structure of the cohomology algebra of the subcomplex of the
high-fold torus over to that of the cohomology algebra of the braid
group. In fact, they proved from this construction that the
cohomology algebra of $B_4T_0$ is an exterior face algebra of a
non-flag complex.

Let $\Gamma$ be a graph that contains $T_0$ and does not contain
$S_0$. In this section, we take a more direct route to compute the
cohomology algebra of $B_n\Gamma$. The two approaches start quiet
differently but they end up producing the essentially same
formulation. So our approach explains a motivation of the
construction by Farley and Sabalka.

Let $(C^i(UD_n\Gamma;\mathbb Z_2),\delta)$ be the cubical cellular
cochain complex with coefficients in $\mathbb Z_2$ of $UD_n\Gamma$,
and $(M^i(UD_n\Gamma;\mathbb Z_2),\tilde\delta)$ be the Morse
cochain complex with coefficients in $\mathbb Z_2$ which is dual to
the Morse chain complex $(M_i(UD_n\Gamma),\tilde \partial)$. Let
$\widetilde R^*$ and $\varepsilon^*$ denote the dual of the chain
maps $\widetilde R$ and $\varepsilon$ discussed in
\S\ref{ss31:homology}. Consider the following commutative diagram.
\begin{displaymath}
\xymatrix{
M^i(UD_n\Gamma;\mathbb Z_2)\ar@{->}[r]^{\tilde \delta}\ar@<.5ex>[d]^{\widetilde R^*} & M^{i+1}(UD_n\Gamma;\mathbb Z_2)\ar@<.5ex>[d]^{\widetilde R^*}\\
C^i(UD_n\Gamma;\mathbb
Z_2)\ar@{->}[r]^{\delta}\ar@<.5ex>[u]^{\varepsilon^*} &
C^{i+1}(UD_n\Gamma;\mathbb Z_2)\ar@<.5ex>[u]^{\varepsilon^*} }
\end{displaymath}

$C_i(UD_n\Gamma;\mathbb Z_2)$ is generated by $i$-cells in
$UD_n\Gamma$ and so $C^i(UD_n\Gamma;\mathbb Z_2)$ is generated by
their dual $c^*$ where $c^*$ is evaluated to 1 on $c$ and to 0 on
other $i$-cells. Similarly, $M^i(UD_n\Gamma;\mathbb Z_2)$ is
generated by $c^*$ dual to critical $i$-cells $c$. Note that the
domain of dual cells should be taken into account although they are
written in the same notation.

If we choose a maximal tree of $\Gamma$ and give an order as in
Figure~\ref{fig9}(b), the boundary map $\widetilde\partial_i$ of the
Morse complex $(M_i(UD_n\Gamma),\tilde
\partial)$ is zero for all $i$ by Theorem~\ref{thm:homology} and so
is the coboundary map $\tilde\delta_i$ of $M^*(UD_n\Gamma;\mathbb
Z_2)$.  Then $c^*$ is a cocycle and so $\widetilde R^*(c^*)$ is a
cocycle of $C^*(UD_n\Gamma;\mathbb Z_2)$. Since homology groups
$H_*(M_*(UD_n\Gamma))$ are free abelian groups generated by critical
cells, the cohomology algebra $H^*(M_*(UD_n\Gamma);\mathbb Z_2)$ is
generated by cohomology classes $[c^*]$. Since
$(C^*(UD_n\Gamma),\delta)$ and $(M^*(UD_n\Gamma),\tilde\delta)$ are
chain homotopy equivalent via $\widetilde R^*$, cohomology classes
$[\widetilde R^*(c^*)]$ for critical cells $c$ generate the
cohomology algebra $H^*(C_*(UD_n\Gamma);\mathbb Z_2)$. Therefore the
algebra structure $H^*(B_n\Gamma;\mathbb Z_2)$ is determined by the
cup product $[\widetilde R^*(c^*)]\cup[\widetilde R^*(c'^*)]$ in
$H^*(UD_n\Gamma;\mathbb Z_2)$ for critical cells $c$ and $c'$. One
may try to compute cup products within the Morse cochain complex,
then all the complications are concentrated at the coboundary map
$\tilde\delta$. In this section, we rather want to disperse the
complication so that we can quickly decide whether a cup product
vanishes or not.

For an $i$-cell $a$ in $UD_n\Gamma$, $\widetilde
R^*(c^*)(a)=c^*(\widetilde R(a))$, which equals to 1 if $c$ occurs
in $\widetilde R(a)$, or 0 otherwise. By Lemma~\ref{lem:redundant},
$\widetilde R \widetilde V (a)=\widetilde R(a)$. Since $\widetilde
V(a)$ is either critical or collapsible, $\widetilde R(a)$ consists
of one critical cell or is 0. So $\widetilde
R^*(c^*)=\sum_{\widetilde R(a)=c}a^*$. $\widetilde R(a)=\widetilde
R\widetilde V(a)=c$ if and only if $\widetilde V(a)=c$. Therefore,
$\widetilde R^*(c^*)=\sum_{\widetilde V(a)=c}a^*$. Note that
$i$-cells $a$ satisfying $\widetilde V(a)=c$ can be characterized as
follows: Suppose $c=\{e_1,\ldots,e_i,v_{i+1},\ldots,v_n\}$. Then $a$
contains the same edges as $c$ and each connected component of
$\Gamma-\bar e_1\cup\cdots\cup\bar e_i$ contains the same number of
vertices from $c$ as from $a$ where $\bar e$ denotes the closure of
the open cell $e$. This precisely gives the equivalence relation
considered by Farley and Sabalka in \cite{FS2}. Let
$c=\{e_1,\ldots,e_i,v_{i+1},\ldots,v_n\}$ be a critical $i$-cell and
$c'=\{e_1',\ldots,e_j',v_{j+1}',\ldots,v_n'\}$ be a critical
$j$-cell in $UD_n\Gamma$. Then
$$[\widetilde R^*(c^*)]\cup[\widetilde R^*(c'^*)]
=[\widetilde R^*(c^*)\cup \widetilde R^*(c'^*)] =[\sum_{\widetilde
V(a)=c}a^*\cup \sum_{\widetilde V(b)=c'}b^*]$$ Since $UD_n\Gamma$ is
a cubical complex, the cup product in $H^*(UD_n\Gamma;\mathbb Z_2)$
can be geometrically formulated as
$$\sum_{\widetilde V(a)=c}a^*\cup\sum_{\widetilde V(b)=c'}b^*
=\sum p^*,$$ where the sum on the right side is taken over all
$(i+j)$-cells $p$ in $UD_n\Gamma$ satisfying the following
conditions.
\begin{itemize}
  \item [(1)] $e_1,\ldots,e_i,e_1',\ldots,e_j'$ are edges in $p$;
  \item [(2)] If an $i$-cell $a$ is an $i$-dimensional
  face of $p$, and $e_1,\ldots,e_i$ are edges of $a$, then
$\widetilde V(a)=c$;
  \item [(3)] If a $j$-cell $b$ is a $j$-dimensional
  face of $p$, and $e'_1,\ldots,e'_j$ are edges of $b$, then
$\widetilde V(b)=c'$.
\end{itemize}

In general, the cup product in the cohomlogy algebra of the Morse
cochain complex can be given by
$$[c^*]\cup[c'^*]=[\varepsilon^*\widetilde R^*(c^*)]\cup
[\varepsilon^*\widetilde R^*(c'^*)]=\varepsilon^*(\sum p^*).$$ In
our situation, if $\sum p^*$ contains a summand $q^*$ for a critical
$(i+j)$-cell $q$, then $\sum p^*$ is equal to $\sum_{\widetilde
V(x)=q}x^*=[\widetilde R^*(q^*)]$ and so $[c^*]\cup[c'^*]=[q^*]$.
Otherwise, we continue to modify $\sum p^*$ up to coboundary until
it contains such a summand or it vanishes.

In order to reduce the amount of computation, we give several
sufficient conditions for a cup product to vanish, which can be
direct consequences of the conditions for $\sum p^*$ above. If a
critical $i$-cell $c=\{e_1,\ldots,e_i,v_{i+1},\ldots,v_n\}$ and a
critical $j$-cell $c'=\{e_1',\ldots,e_j',v_{j+1}',\ldots,v_n'\}$
satisfy one of the following three conditions, then the cup product
$[c^*]\cup[c'^*]=0$ since each condition violates the corresponding
condition given above.
\begin{itemize}
\item[(1)] There are edges $e$ in $c$ and $e'$ in $c'$ such that
$\bar e \cap\bar e'\neq\emptyset$;
\item[(2)] Every $i$-cell $a$ with $\widetilde V(a)=c$ contains neither
$\iota(e_\ell')$ nor $\tau(e_\ell')$ for some $\ell=1,\ldots,j$;
\item[(3)] Every $j$-cell $b$ with $\widetilde V(b)=c'$ contains neither
$\iota(e_\ell)$ nor $\tau(e_\ell)$ for some $\ell=1,\ldots,i$.
\end{itemize}
There are critical cells $c$ and $c'$ that satisfy none of the
vanishing conditions, but the cup product $[c^*]\cup[c'^*]$ is still
0. It is not easy to compute the cohomology algebra of the unordered
configuration space of an arbitrary graph that contains $T_0$ but
not $S_0$. But it is possible to compute cohomology algebras of
graph braid groups of low braid indices if a graph is simple enough
like $T_1$, $T_2$, $T_3'$, $T_3''$ and $T_3'''$.

The following lemma shows that the cohomology algebras of $UD_4T_1$,
$UD_3T_2$, $UD_2T_3'$, $UD_2T_3''$, and $UD_2T_3'''$ are exterior
face algebras of non-flag complexes. Consequently none of braid
groups $B_4T_1$, $B_3T_2$, $B_2T_3'$, $B_2T_3''$, and $B_2T_3'''$ is
a right-angled Artin group.

\begin{lem}\label{lem:cohomology_algebra}
The cohomology algebra $H^*(B_n\Gamma;\mathbb Z_2)$ is isomorphic to
the exterior face algebra $\Lambda(K)$ where the braid index $n$,
the graph $\Gamma$, and the simplicial complex $K$ are given by rows
of the following table:
$$
\begin{tabular}{|c|c|c|}\hline
$n$ & $\Gamma$ & $K$\\\hline
$4$ & $T_1$ & the graph in Figure~\ref{fig10}(a) with $9$ extra vertices\\
$3$ & $T_2$ & the graph in Figure~\ref{fig10}(b) with $5$ extra vertices\\
$2$ & $T'_3$ & the graph in Figure~\ref{fig10}(c) with $5$ extra vertices\\
$2$ & $T''_3$ & the graph in Figure~\ref{fig10}(c) with $6$ extra vertices\\
$2$ & $T'''_3$ & the graph in Figure~\ref{fig10}(c) with $6$ extra vertices\\\hline
\end{tabular}
$$
\begin{figure}[ht]
\psfrag{A2(1,2)}{\small$\{3\mbox{-}7,4,0,8\}^*$}
\psfrag{A2(2,2)}{\small$\{3\mbox{-}7,4,5,8\}^*$}
\psfrag{A2(3,1)}{\small$\{3\mbox{-}7,4,5,6\}^*$}
\psfrag{A2(1,3)}{\small$\{3\mbox{-}7,4,8,9\}^*$}
\psfrag{B2(1,3)}{\small$\{9\mbox{-}17,10,18,19\}^*$}
\psfrag{B2(2,1)}{\small$\{9\mbox{-}17,10,0,11\}^*$}
\psfrag{B2(3,1)}{\small$\{9\mbox{-}17,10,11,12\}^*$}
\psfrag{B2(2,2)}{\small$\{9\mbox{-}17,10,11,18\}^*$}
\psfrag{B2(1,1)}{\small$\{9\mbox{-}17,10,0,1\}^*$}
\psfrag{C2(1,1)}{\small$\{19\mbox{-}23,20,0,1\}^*$}
\psfrag{C2(1,2)}{\small$\{19\mbox{-}23,20,0,24\}^*$}
\psfrag{C2(2,1)}{\small$\{19\mbox{-}23,20,0,21\}^*$}
\psfrag{e1}{\small$\{d_1,0,1,2\}^*$}
\psfrag{e1{13}}{\small$\{d_1,0,1,13\}^*$}
\psfrag{3A2(1,2)}{\small$\{2\mbox{-}5,3,6\}^*$}
\psfrag{3B2(1,2)}{\small$\{6\mbox{-}12,7,13\}^*$}
\psfrag{3B2(2,1)}{\small$\{6\mbox{-}12,7,8\}^*$}
\psfrag{3e1}{\small$\{d_1,0,1\}^*$}
\psfrag{3e2}{\small$\{d_2,0,1\}^*$}
\psfrag{e1{9}}{\small$\{d_1,0,9\}^*$}
\psfrag{e2{13}}{\small$\{d_2,0,14\}^*$}
\psfrag{e1{3}}{\small$\{d_1,3\}^*$}
\psfrag{e2}{\small$\{d_2,0\}^*$}
\psfrag{e3}{\small$\{d_3,0\}^*$}
\subfigure[]
{\includegraphics[height=3.2cm]{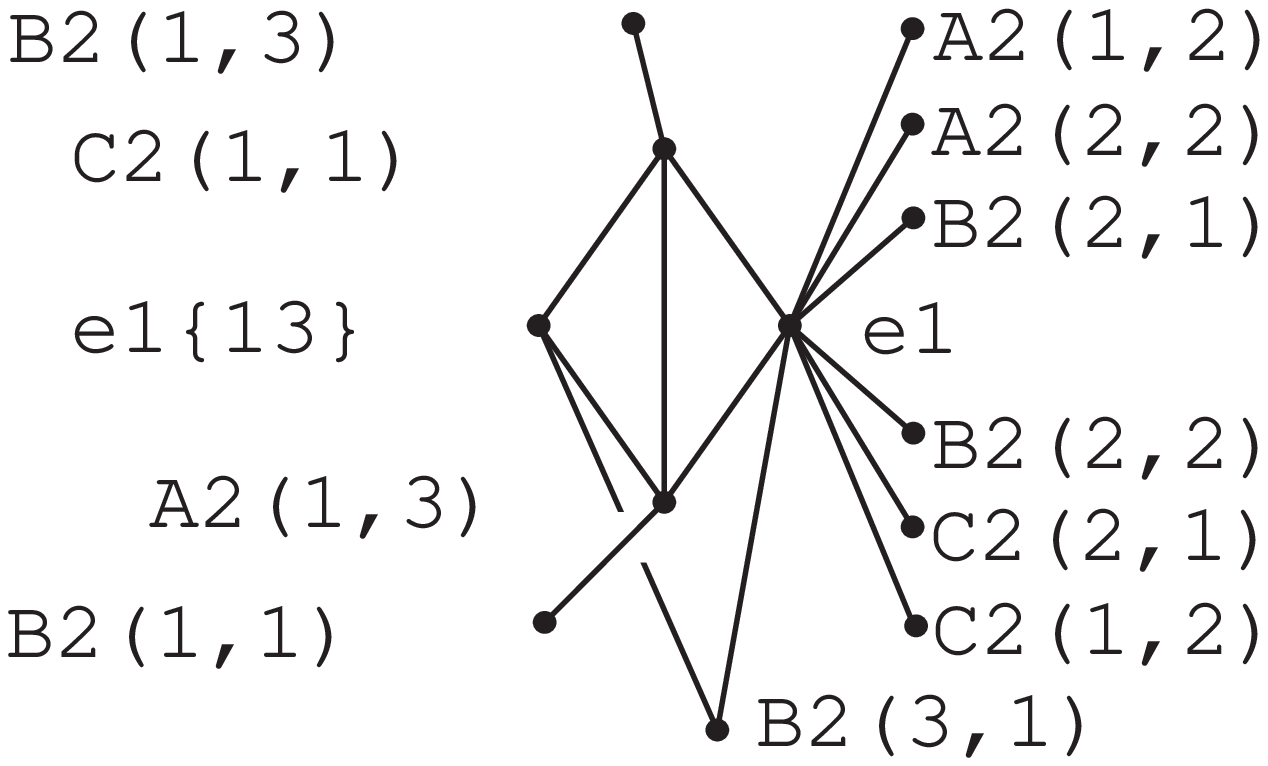}}\qquad\subfigure[]
{\includegraphics[height=3.2cm]{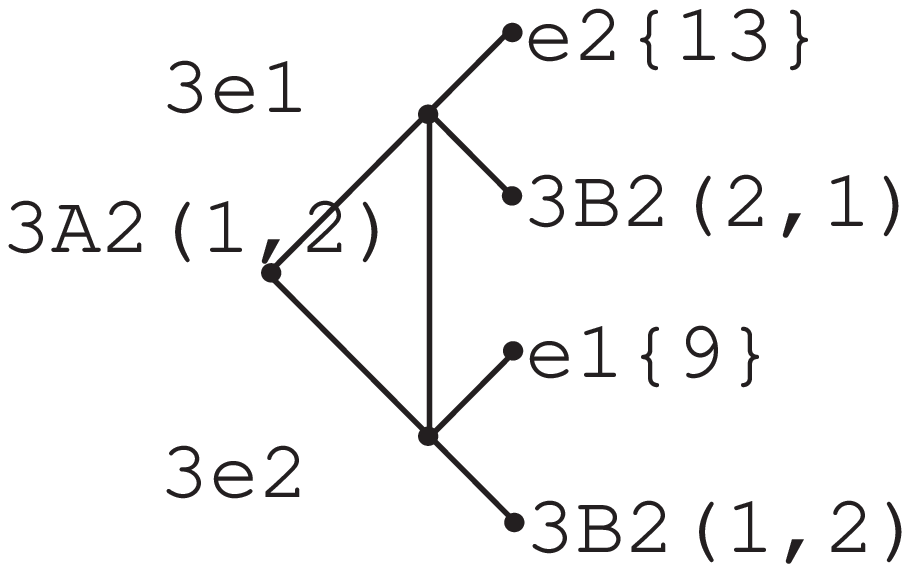}}\\ \subfigure[]
{\includegraphics[height=1.3cm]{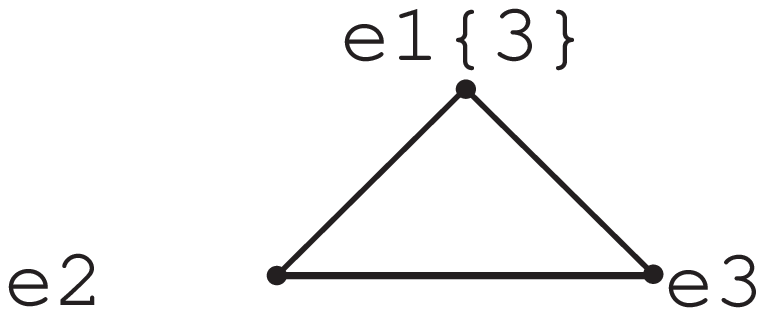}}
\caption{Complex $K$ labeled by 1-cocycles with no isolated vertices}
\label{fig10}
\end{figure}
\end{lem}

\begin{proof}
We only discuss $H^*(UD_3T_2;\mathbb Z_2)$. The rest are similar
using maximal trees and orders given in Figure~\ref{fig11}. In the
figure, vertices of valency $\ge 3$ labeled by their numbering, and
the base vertex 0 and deleted edges $d_i$ are indicated.

\begin{figure}[ht]
\psfrag{A}{\small$A$}
\psfrag{B}{\small$B$}
\psfrag{C}{\small$C$}
\psfrag{0}{\small0}
\psfrag{2}{\small2}
\psfrag{3}{\small3}
\psfrag{4}{\small4}
\psfrag{5}{\small5}
\psfrag{6}{\small6}
\psfrag{7}{\small7}
\psfrag{8}{\small8}
\psfrag{9}{\small9}
\psfrag{10}{\small10}
\psfrag{11}{\small11}
\psfrag{12}{\small12}
\psfrag{13}{\small13}
\psfrag{14}{\small14}
\psfrag{16}{\small16}
\psfrag{19}{\small19}
\psfrag{22}{\small22}
\psfrag{25}{\small25}
\psfrag{e1}{\small$d_1$}
\psfrag{e2}{\small$d_2$}
\psfrag{e3}{\small$d_3$}
\subfigure[\small$T_1$]
{\includegraphics[height=3cm]{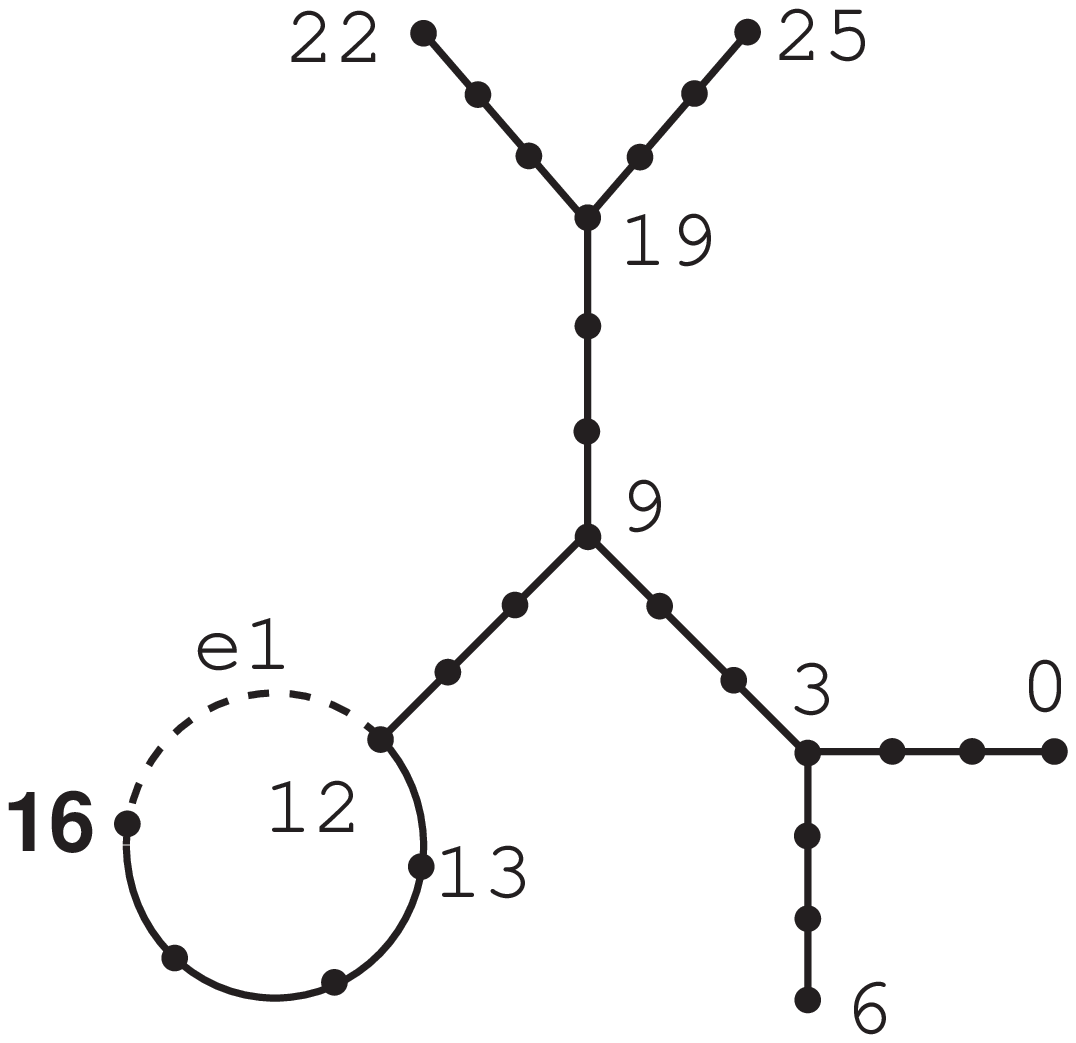}}\qquad\quad
\subfigure[$T_2$]
{\includegraphics[height=3cm]{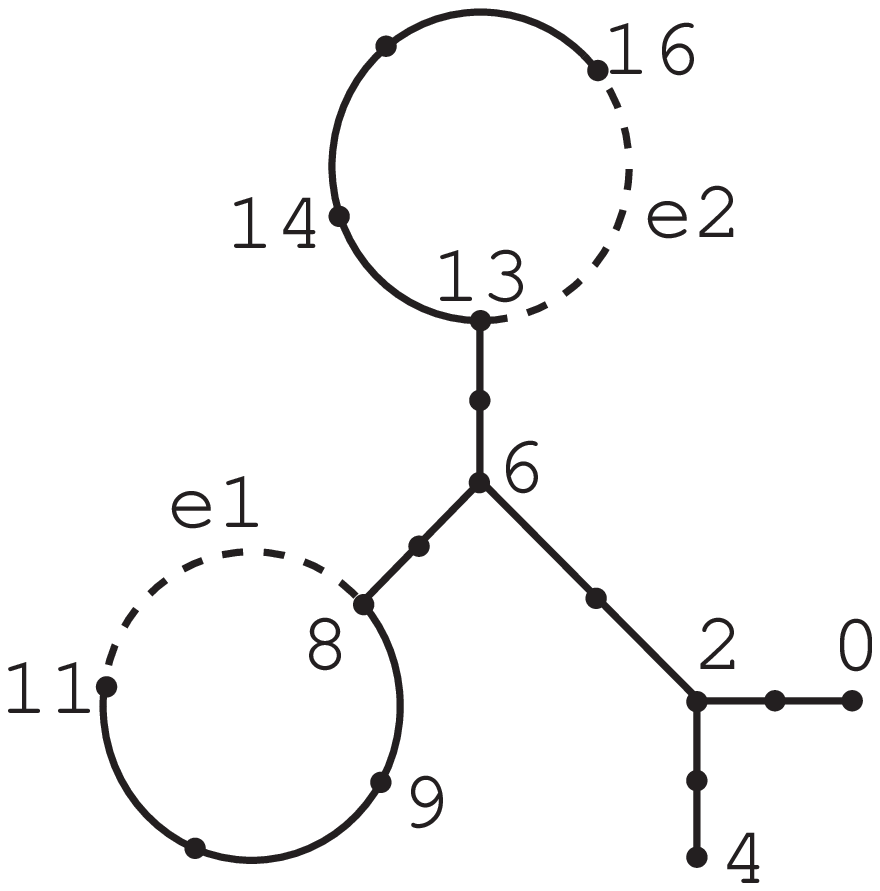}}\\
\subfigure[$T_3'$]
{\includegraphics[height=2.5cm]{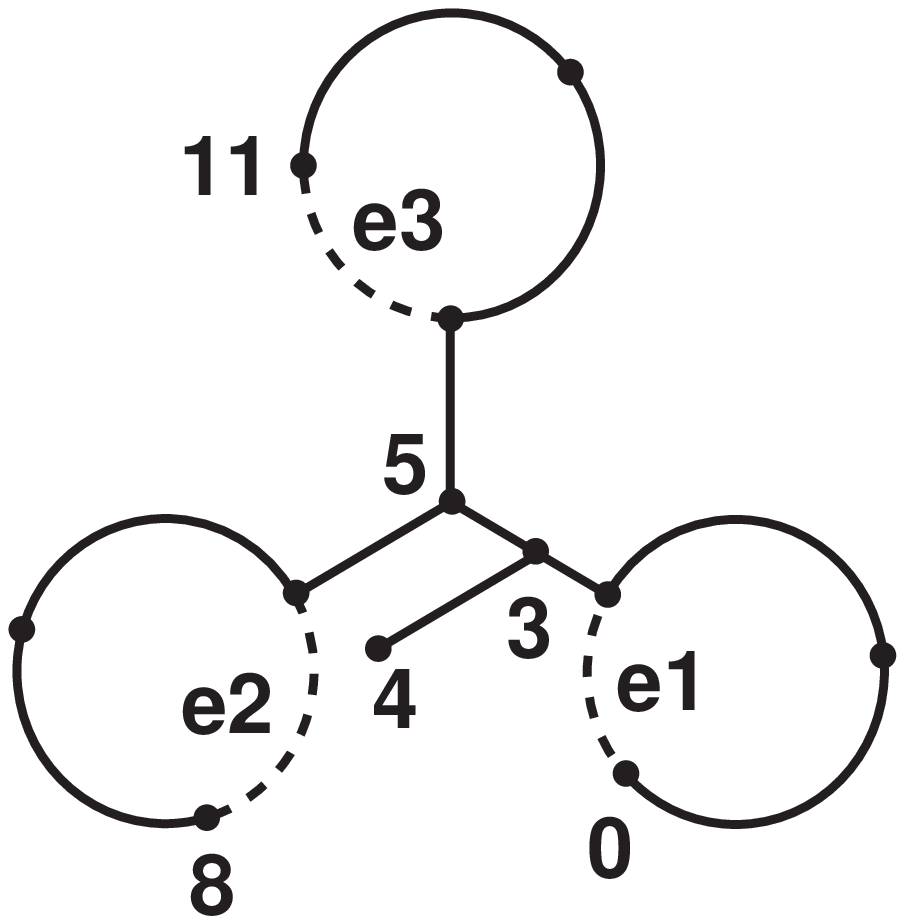}}\quad
\subfigure[$T_3''$]
{\includegraphics[height=2.5cm]{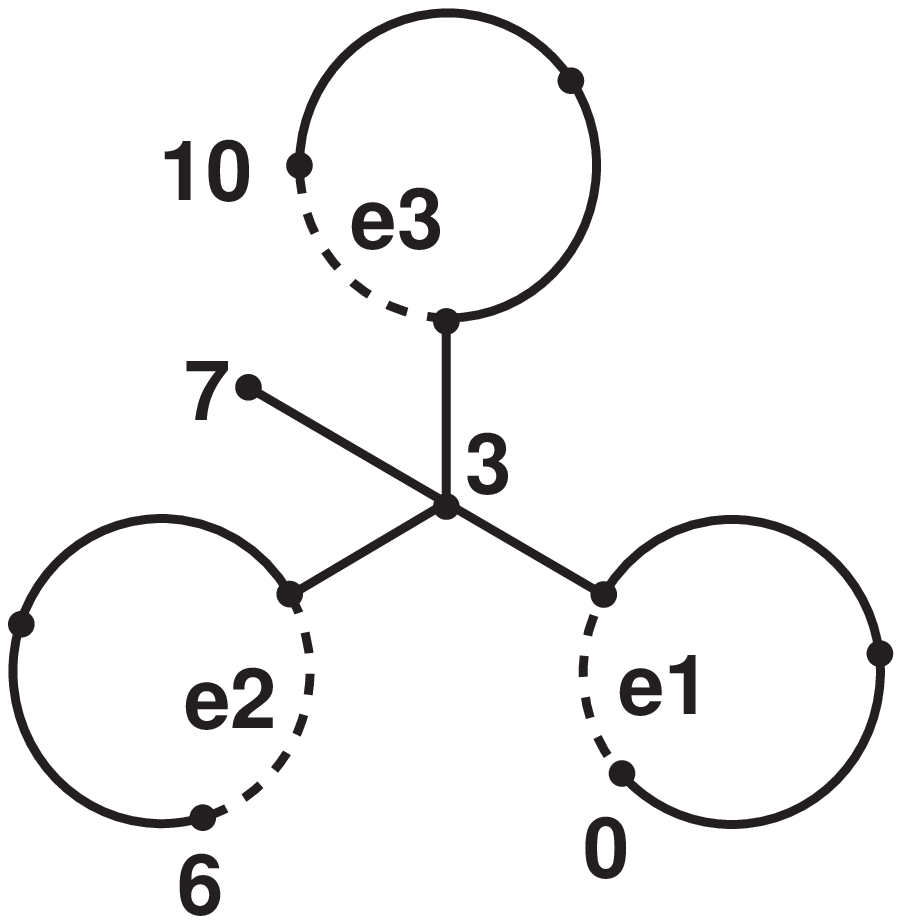}}\quad
\subfigure[$T_3'''$]
{\includegraphics[height=2.5cm]{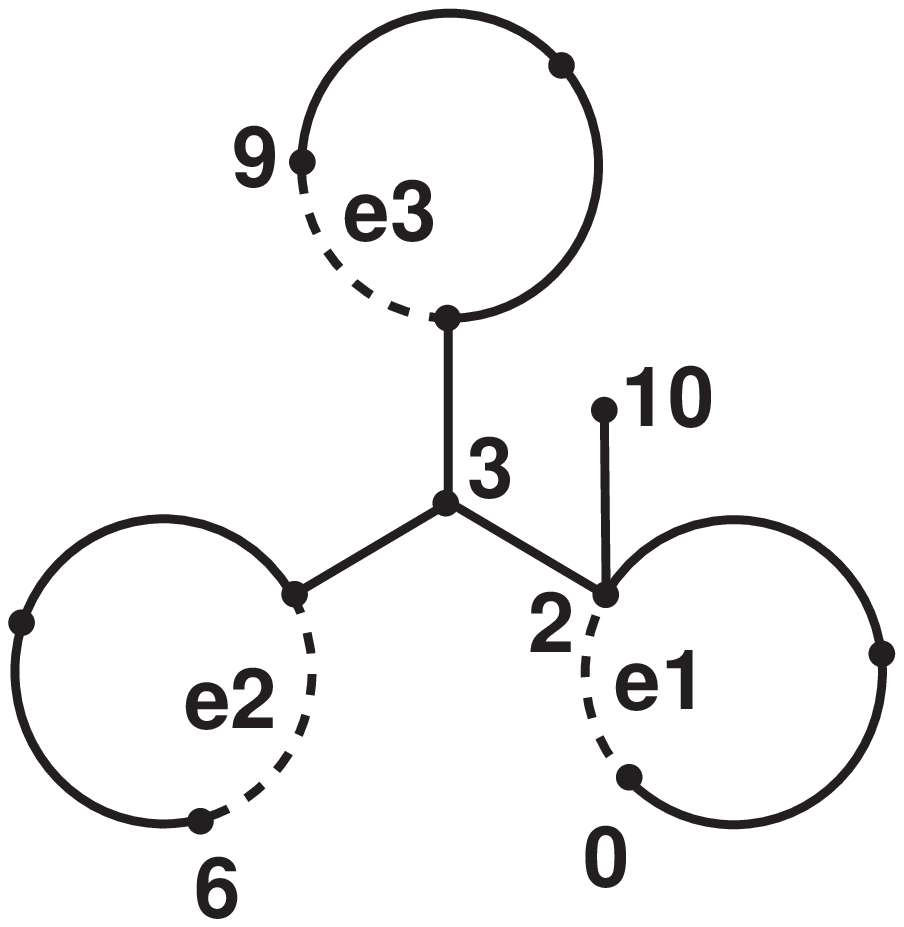}}
\caption{Maximal trees with numbering on vertices}
\label{fig11}
\end{figure}

If we choose a maximal tree of $T_2$ and give an order as in
Figure~\ref{fig11}(b), then the complex $UD_3T_2$ has twelve
critical 1-cells $\{2\mbox{-}5,3,0\}$, $\{2\mbox{-}5,3,4\}$,
$\{2\mbox{-}5,3,6\}$, $\{6\mbox{-}12,\allowbreak 7,0\}$,
$\{6\mbox{-}12,7,8\}$, $\{6\mbox{-}12,7,13\}$, $\{d_1,0,1\}$,
$\{d_1,0,9\}$, $\{d_1,9,10\}$, $\{d_2,0,1\}$, $\{d_2,0,14\}$,
$\{d_2,\allowbreak 14,15\}$, and seven critical 2-cells
$\{d_1,d_2,0\}$, $\{d_1,d_2,9\}$, $\{d_1,d_2,14\}$,
$\{2\mbox{-}5,d_1,3\}$, $\{2\mbox{-}5,d_2,3\}$,
$\{6\mbox{-}12,d_1,7\}$, $\{6\mbox{-}12,d_2,7\}$. And there are no
critical $i$-cells for $i\ge 3$. Thus, the cohomology algebra
$H^*(UD_3T_2;\mathbb Z_2)$ is completely determined by the cup
products of cohomology classes that are dual to critical 1-cells.
Let $c=\{e_1,v_2,v_3\}$ and $c'=\{e'_1,v'_2,v'_3\}$ be critical
1-cells. Recall the formula
$$[\widetilde R^*(c^*)]\cup[\widetilde R^*(c'^*)]=[\sum_{\widetilde
V(a)=c}a^*\cup\sum_{\widetilde V(b)=c'}b^*]=[\sum p^*],$$ where the
sum on the right side is taken over all 2-cells $p$ in $UD_3T_2$
such that
\begin{itemize}
  \item [(1)] $e_1, e'_1\in p$;
  \item [(2)] $\widetilde V(a)=c$ for a 1-dimensional face $a$ of
  $p$ containing $e_1$;
  \item [(3)] $\widetilde V(b)=c'$ for a 1-dimensional face $b$ of
   $p$ containing $e'_1$.
\end{itemize}

If $c$ and $c'$ satisfy one of the vanishing conditions explained
earlier, then the cup product vanishes. For example, consider the
case $c=\{2\mbox{-}5,3,6\}$ and $c'=\{d_1,0,9\}$. Then $p$ must be a
2-cell that contains edges 2-5 and $d_1$. If $a$ is a face of $p$
containing the edge 2-5 and $\widetilde V(a)=c$, then $p$ is either
$\{2\mbox{-}5,d_1,3\}$ or $\{2\mbox{-}5,d_1,4\}$. On the other hand,
if $b$ is a face of $p$ containing $d_1$ and $\widetilde
V(b)=\{d_1,0,9\}$, then $p$ is either $\{2\mbox{-}5,d_1,9\}$ or
$\{2\mbox{-}5,d_1,10\}$. Thus no such 2-cell $p$ exist, and so
$[\widetilde R^*(c^*)]\cup[\widetilde R^*(c'^*)]=0$, and so
$[\{2\mbox{-}5,3,6\}^*]\cup[\{d_1,0,9\}^*]=0$. Similarly, the
following cup products vanish:
 \begin{align*}
 &[\{2\mbox{-}5,3,6\}^*]\cup[\{6\mbox{-}12,7,0\}^*],
 &&[\{2\mbox{-}5,3,6\}^*]\cup[\{d_2,0,14\}^*],\\
 &[\{6\mbox{-}12,7,0\}^*]\cup[\{d_1,0,9\}^*],
 &&[\{6\mbox{-}12,7,13\}^*]\cup[\{d_1,0,9\}^*],\\
 &[\{6\mbox{-}12,7,13\}^*]\cup[\{d_2,0,14\}^*],
 &&[\{d_1,0,9\}^*]\cup[\{d_2,0,14\}^*].
 \end{align*}

If $\sum p^*$ contains a summand $q*$ for critical 2-cell $q$, then
$[c^*]\cup [c'^*]=[q^*]$ as mentioned before. For example, consider
the case $c=\{2\mbox{-}5,3,6\}$ and $c'=\{d_1,0,1\}$. Then
\begin{align*}
[\widetilde R^*(c^*)]\cup[\widetilde R^*(c'^*)]
&=[\sum_{\substack{3\le i\le 4\\6\le j\le 16}}
\{2\mbox{-}5,i,j\}^*\cup\sum_{k,\ell\in A}\{d_1,k,\ell\}^*]\\
&=[\sum_{\widetilde V(x)=\{2\mbox{-}5,d_1,3\}}x^*]
=[\widetilde R^*(\{2\mbox{-}5,d_1,3\}^*)].
\end{align*}
where $A=\{0,\ldots,7,12,\ldots,16\}$. Thus
$[\{2\mbox{-}5,3,6\}^*]\cup[\{d_1,0,1\}^*]=[\{2\mbox{-}5,d_1,3\}^*]$.

The following cup products are obtained similarly:
\begin{align*}
[\{2\mbox{-}5,3,6\}^*]\cup[\{d_2,0,1\}^*]&=[\{2\mbox{-}5,d_2,3\}^*],\\
[\{6\mbox{-}12,7,8\}^*]\cup[\{d_1,0,1\}^*]&=[\{6\mbox{-}12,d_1,7\}^*],\\
[\{6\mbox{-}12,7,13\}^*]\cup[\{d_2,0,1\}^*]&=[\{6\mbox{-}12,d_2,7\}^*],\\
[\{d_1,0,1\}^*]\cup[\{d_2,0,1\}^*]&=[\{d_1,d_2,0\}^*],\\
[\{d_1,0,1\}^*]\cup[\{d_2,0,14\}^*]&=[\{d_1,d_2,14\}^*],\\
[\{d_1,0,9\}^*]\cup[\{d_2,0,1\}^*]&=[\{d_1,d_2,9\}^*].
\end{align*}

If $\sum p^*$ contains no dual to any critical 2-cell, then we
continue to modify $\sum p^*$ up to coboundary until it contains
such a summand or it vanishes. For the case $c=\{6\mbox{-}12,7,0\}$
and $c'=\{d_1,0,1\}$,
\begin{align*}
[\widetilde R^*(c^*)]\cup[\widetilde
R^*(c'^*)]=&[\{6\mbox{-}12,d_1,0\}^*+\{6\mbox{-}12,d_1,1\}^*+\{6\mbox{-}12,d_1,2\}^*\\
                   &+\{6\mbox{-}12,d_1,3\}^*+\{6\mbox{-}12,d_1,4\}^*+\{6\mbox{-}12,d_1,5\}^*].
\end{align*}

Since $\{6\mbox{-}12,d_1,0\}$, $\{6\mbox{-}12,d_1,1\}$,
$\{6\mbox{-}12,d_1,2\}$, $\{6\mbox{-}12,d_1,3\}$,
$\{6\mbox{-}12,d_1,4\}$ and $\{6\mbox{-}12,d_1,5\}$ are not
critical, we consider the cochain $\alpha^*=\sum_{0\le i,j\le
7}\{d_1,i,j\}^*$ in $C^1(UD_3T_2;\mathbb Z_2)$. Then
$$\delta(\alpha^*)=\widetilde
R^*(\{6\mbox{-}12,7,0\}^*)\cup\widetilde
R^*(\{d_1,0,1\}^*)+\widetilde R^*(\{6\mbox{-}12,7,8\}^*)\cup\widetilde R^*(\{d_1,0,1\}^*).$$ So,
$$[\{6\mbox{-}12,7,0\}^*]\cup[\{d_1,0,1\}^*]
=[\{6\mbox{-}12,7,8\}^*]\cup[\{d_1,0,1\}^*]=[\{6\mbox{-}12,d_1,7\}^*].$$ For the case $c=\{6\mbox{-}12,7,8\}$ and $c'=\{d_1,0,9\}$,
$$[\widetilde R^*(c^*)]\cup[\widetilde R^*(c'^*)]
=[\{6\mbox{-}12,d_1,9\}^*+\{6\mbox{-}12,d_1,10\}^*].$$ Since
$\{6\mbox{-}12,d_1,9\}$ and $\{6\mbox{-}12,d_1,10\}$ are not
critical, we consider the cochain $\beta^*=\sum_{\substack{0\le i\le
7\\9\le j\le 10}}\{d_1,i,j\}^*$ in $C^1(UD_3T_2;\mathbb Z_2)$. Then
$\delta(\beta^*)=\{6\mbox{-}12,d_1,9\}^*+\{6\mbox{-}12,d_1,10\}^*$.
Thus $[\{6\mbox{-}12,7,8\}^*]\cup[\{d_1,0,9\}^*]=0$. Similarly,
$$[\widetilde R^*(\{6\mbox{-}12,7,13\}^*)]\cup[\widetilde
R^*(\{d_1,0,1\}^*)]=[\delta(\gamma^*)]=0$$ for the cochain
$\gamma^*=\sum_{12\le i,j\le 16}\{d_1,i,j\}^*$in
$C^1(UD_3T_2;\mathbb Z_2)$. In order to see the cohomology algebra
$H^*(UD_3T_2;\mathbb Z_2)$ as an exterior face algebra of a complex,
we need to change a basis by replacing $[\{6\mbox{-}12,7,0\}^*]$ by
$[\{6\mbox{-}12,7,0\}^*]+[\{6\mbox{-}12,7,8\}^*]$, due to the
relation resulted from the coboundary $\delta(\alpha^*)$. Then
$$([\{6\mbox{-}12,7,0\}^*]+[\{6\mbox{-}12,7,8\}^*])\cup[\{d_1,0,1\}^*]=0.$$
Consequently, $H^*(UD_3T_2;\mathbb Z_2)$ is isomorphic to an
exterior face algebra of a graph in Figure~\ref{fig9}(b) together
with 5 extra vertices which is a non-flag complex. In the figure,
each vertex is labeled by a 1-cocycle dual to a critical 1-cell.
\end{proof}

It should be noted that the approach in this section cannot be used
to compute the cohomology algebra of graph braid groups unless the
boundary maps $\tilde\partial$ of the Morse chain complex are all
zero and $\widetilde R\widetilde V(c)=\widetilde R(c)$.

\subsection{Graphs that is not $T_3$, and contains $T_0$ but not $S_0$}\label{ss33:embedding}
Recall that for $k=0,1,2$, the set $\mathcal G_k$ consists of graphs
that contain $T_k$ but do not contain $T_{k+1}$ and $\mathcal G_3$
consists of graphs containing $T_3'$, $T_3''$ or $T_3'''$ but not
$S_0$. If a graph $\Gamma$ is in $\cup_{k=0}^3\mathcal G_k$, we want
to show $B_n\Gamma$ is not a right-angled Artin group. In the view
of Proposition~\ref{pro:flag} and
Lemma~\ref{lem:cohomology_algebra}, the only remaining ingredient is
the following lemma.

\begin{lem}\label{lem:ker}
Let $\Gamma$ be a graph in $\mathcal G_k$ for some $k=0,1,2,3$. Then
there is an embedding $i:P_k\to \Gamma$ that induces a
degree-preserving epimorphism $i^*:H^*(B_n\Gamma;\mathbb Z_2)\to
H^*(B_{n_k}P_k,\mathbb Z _2)$ of cohomology algebras such that
$\ker(i^*)$ is generated by homogeneous elements of degree 1 and 2
where $P_j=T_j$ for $j=0,1,2$ and $P_3$ is one of $T_3',T_3''$ and
$T_3'''$, and $n_0=4$, $n_1=4$, $n_2=3$, $n_3=2$ and $n\ge n_i$.
\end{lem}

\begin{proof}
It is easy to find an embedding $i:P_k\to \Gamma$ satisfying the
conditions discussed right after Theorem~\ref{thm:homology}. Then
the induced embedding $i:UD_{n_k}P_k\to UD_n\Gamma$ for $n\ge n_k$
preserves cell types and induces a chain map $\tilde
i:M_*(UD_{n_k}P_k)\to M_*(UD_n\Gamma)$ between Morse chain
complexes. By Theorem~\ref{thm:homology}, the homology groups
$H_*(M_*(UD_{n_k}P_k))$ and $H_*(M_*(UD_n\Gamma))$ are free abelian
groups that are generated by critical cells. Then the induced map
$\tilde i_*:H_*(M_*(UD_{n_k}P_k))\to H_*(M_*(UD_n\Gamma))$ is
injective and sends generators to generators.  By the universal
coefficient theorem, the induced map $\tilde
i^*:H^*(M_*(UD_n\Gamma);\mathbb Z_2)\to H^*(M_*(UD_{n_k}P_k);\mathbb
Z_2)$ of cohomology algebras is surjective and sends the generator
$[c^*]$ to $[\tilde i^*(c^*)]$ if $c\in\im(i)$, or to 0 otherwise.
And $\tilde i^*$ is clearly a degree-preserving algebra
homomorphism.

We now show that $\ker(\tilde i^*)$ is generated by homogeneous
elements of degree 1 and 2. Let $c$ be a critical $j$-cell in
$UD_n\Gamma$ for $j\ge 3$. Since there is no critical $j$-cell in
$UD_{n_k}P_k$ with $j\ge 3$, $H^j(UD_{n_k}P_k)=0 $ and so
$[c^*]\in\ker(\tilde i^*)$ since $\tilde i^*$ is degree-preserving.
By induction on the degree of $c^*$, it suffices to show that
$[c^*]$ is divisible by a cohomology class $[a^*]\in\ker(\tilde
i^*)$ for some critical 1-cell $a$ in $UD_n\Gamma$. Since the
embedding $i$ preserves cell type and there is no critical $j$-cell
for $j\ge 3$, $c\notin \im(i)$. Thus there is either an edge or a
vertex in $c$ that is not in $\im(i)$. If $c$ contains an edge $e'$
that is not in $\im(i)$, then a 1-cell $b$ can be obtained by taking
a 1-dimensional face of $c$ containing $e'$. Since $c$ is critical,
$c$ contains no order-respecting edges and so $b$ contains no
order-respecting edge, either. Thus $\widetilde V(b)$ is a critical
1-cell. Let $a=\widetilde V(b)$. Then $[\widetilde R^*(c^*)]$ is
divisible by $[\widetilde R^*(a^*)]$ by the construction of $b$.
Thus $[c^*]$ is divided by $[a^*]$. Since $\widetilde V(b)$ contains
the edge $e'$, $a\notin \im(i)$. Therefore $[a^*]\in\ker(\tilde
i^*)$.

If $c$ contains a vertex $v'$ that is not in $\im(i)$, there is an
edge $e'$ that is not order-respecting and $v'$ is blocked by $e'$.
Recall that we say that a vertex $v$ is blocked by an edge $e$ in a
cell $c$ if there is no vertices in $\Gamma-c$ that are smaller than
$v$ and larger than an either end of $e$ that is smaller than $v$.
Let $b$ be a 1-dimensional face of $c$ containing $e'$. Then
$[\widetilde R(c^*)]$ is divisible by $[\widetilde R^*(a^*)]$ where
$a=\widetilde V(b)$. Since $v'$ is blocked by $e'$, $v'$ is still in
the critical 1-cell $\widetilde V(b)$ and so $a\notin \im(i)$. Thus
$[a^*]\in\ker(\tilde i^*)$ and $[c^*]$ is divisible by $[a^*]
\in\ker(\tilde i^*)$.
\end{proof}

\begin{cor}
Using the notation of Lemma~\ref{lem:ker}, the braid group
$B_n\Gamma$ is not a right-angled Artin group for $n\ge n_k$ if a
graph $\Gamma$ is in $\mathcal G_k$ for $k=0,1,2,3$.
\end{cor}
\begin{proof}
Immediate from Proposition~\ref{pro:flag} together with
Lemma~\ref{lem:cohomology_algebra} and \ref{lem:ker}.
\end{proof}

\subsection{Braid group of $T_3$}\label{ss34:T_3}
By a computation similar to those in the previous section, it is not
hard to see that the cohomology algebra $H^*(B_2T_3;\mathbb Z_2)$ is
the exterior face algebra of the complex consisting of a triangle
together with 4 extra vertices. As explained earlier, this
computation cannot be carried over to the braid group over $T_3$ of
arbitrarily large braid indices due to the nonexistence of vertices
of valency 1 in $T_3$. We introduce another tool that, in the rest
of the article, will provide several necessary conditions for a
group to be a right-angled Artin group.

 A group $G$ is called a
\emph{commutator-related group} if it has a finite presentation
$\langle x_1,\cdots,x_n \mid r_1,\cdots, r_m\rangle$ such that each
relator $r_j$ belongs to the commutator subgroup $[F,F]$ of the free
group $F$ generated by $x_1,\ldots,x_n$. Let $R$ be the normal
subgroup of $F$ generated by $r_1,\ldots, r_m$. Then we can define a
homomorphism
$$\Phi_G : R / [F,R] \to [F,F] / [F,[F,F]]$$
induced by the inclusion $R \to [F,F]$. Then $H_2(G)$ is identified
with $R/[F,R]$ by Hopf's isomorphism, $H_1(G)$ is canonically
identified with $F/[F,F]$, and $H_1(G)\wedge H_1(G)$ is identified
with $[F,F]/[F,[F,F]]$ by sending $(a[F,F])\wedge (b[F,F])$ onto the
coset $[a,b][F,[F,F]]$ for $a,b\in F$. So $\Phi_G$ can be
regarded as a homomorphism
$$\Phi_G:H_2(G)\to H_1(G)\wedge H_1(G).$$
induced by $\Phi_G$. The following proposition says the homomorphism
$\Phi_G$ is related to something more familiar.

\begin{pro}{\rm (Matei-Suciu \cite{MS})}\label{pro:dual}
If $G$ is a commutator-related group then $\Phi_G:H_2(G)\to
H_1(G)\wedge H_1(G)$ is the dual of the cup product $H^1(G)\wedge
H^1(G)\to H^2(G)$.
\end{pro}

If $G$ is a right-angled Artin group, then the cohomology ring
$H^*(G)$ is an exterior face ring of a flag complex by
Proposition~\ref{pro:flag}(1). So $H^*(G)$ is completely determined
by the cup product $H^1(G)\wedge H^1(G) \to H^2(G)$. Since a
right-angled Artin group is a commutator-related group and the cup
product is the dual of $\Phi_G$, the cohomology ring $H^*(G)$ is
completely determined by the homomorphism $\Phi_G$.

We show that $B_nT_3$ is a commutator-related group for $n\ge 2$
and find a presentation such that $\Phi_G$ sends its relators to
commutators of its generators. Then we can compute the cup product
$H^1(G)\wedge H^1(G)\to H^2(G)$ directly from the presentation and
the homomorphism $\Phi_G$. Suppose $B_nT_3$ is a right-angled Artin
group. Then we can compute the rank of $H_3(B_nT_3)$ by counting certain triples of generators given by the duality of Proposition~\ref{pro:dual}.

On the other hand, we also
compute the rank of $H_3(B_nT_3)$ using the Morse complex of
$UD_nT_3$ since $UD_nT_3$ is a $K(B_nT_3,1)$ space. If two ranks are
different, then we conclude that the contradiction is caused by the
assumption that $B_nT_3$ is a right-angled Artin group. This
strategy proves the following:

\begin{thm}\label{thm:T_3}
Let $T_3$ be a graph in Figure~\ref{fig8} and $n\ge 2$. The group
$B_nT_3$ is not a right-angled Artin group
\end{thm}

\begin{proof}
\begin{figure}[ht]
\psfrag{*}{\small0}
\psfrag{A}{\small$A$}
\psfrag{e1}{\small$d_1$}
\psfrag{e2}{\small$d_2$}
\psfrag{e3}{\small$d_3$}
\centering
\includegraphics[height=2.7cm]{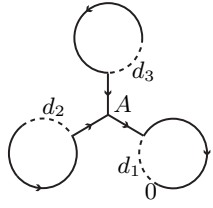}
\caption{Maximal tree of $T_3$ and deleted edges}\label{fig12}
\end{figure}

If we choose a maximal tree of $T_3$ and give an order as in
Figure~\ref{fig12}, the complex $UD_nT_3$ has two kinds of critical
1-cells: $$ d_i(k),\quad A_2(\vec a),$$ and two kinds of critical
2-cells: for $i\neq j$
$$A_2(\vec a)\cup d_i(k),\quad d_i(k)\cup d_j(\ell),$$ and two kinds of
critical 3-cells: for $i\neq j$ $$A_2(\vec a)\cup d_i(k)\cup
d_j(\ell),\quad d_1(k)\cup d_2(\ell)\cup d_3(m)$$ Here $i,j\in \{
1,2,3 \}$ and $d_i(k)$ denotes the set consisting of $d_i$ together
with $k$ blocked vertices from $\tau(d_i)$ if $i=2,3$ and
$\iota(d_i)$ if $i=1$.

We consider the Morse complex $M_i(UD_nT_3)$ of $UD_nT_3$ determined
by the above critical $i$-cells. By Theorem~\ref{thm:homology},
$H_i(M_*(UD_nT_3))$ is the free group generated by the critical
$i$-cells. We first consider the number $N_{A,i,j}$ of all critical
3-cells of the form $A_2(\vec a)\cup d_i(k)\cup d_j(\ell)$ that
consists of $d_i$, $d_j$ and $A_2(1,1)$ together with $n-4$ blocked
vertices. Each blocked vertices is placed on $T_3-(\{\bar d_i,\bar
d_j\} \cup A_2(1,1))$ that consists of five connected components
where $\bar d_i$ is the closure of $d_i$. Thus $N_{A,i,j}$ is the
number of ways to choose $n-4$ objects from $5$ objects with
repetitions and so $N_{A,1,2} = N_{A,2,3} =N_{A,3,1}
=\binom{5+(n-4)-1}{n-4}
 = \binom n 4 = n(n-1)(n-2)(n-3)/24$. Now
consider the number $N_{1,2,3}$ of all critical 3-cells of the form
$d_1(k)\cup d_2(\ell)\cup d_3(m)$ that consists of  $d_1$, $d_2$ and
$d_3$ together with $n-3$ blocked vertices placed on $T_3-\{\bar
d_1, \bar d_2, \bar d_3\}$ consisting of four connected components.
So $N_{1,2,3}=\binom{4+(n-3)-1}{n-3}=\binom n3= n(n-1)(n-2)/6$. Thus
$\rk(H_3(B_nT_3))=\rk(M_3(UD_nT_3))
=n(n-1)(n-2)(n-3)/8+n(n-1)(n-2)/6$.

On the other hand, $B_nT_3$ has the presentation whose generators
are the critical 1-cells and whose relators are the rewritten
boundary words of critical 2-cells given as follows:
\begin{align*}
\tilde r(\partial(A_2(\vec a)\cup d_1(k)))&=[A_2(\vec a),d_1(|\vec a|+k)],\\
\tilde r(\partial(A_2(\vec a)\cup d_2(k)))\\
=[(\Pi_{a_2-1}^1&(k))^{-1}(A_2(k+1,a_2))^{-1}
A_2(\vec a+(k+1)\vec\delta_1)\Pi_{a_2-1}^1(k),d_2(k)],\\
\tilde r(\partial(A_2(\vec a)\cup d_3(k)))&=[A_2(\vec a+(k+1)\vec\delta_2),d_3(k)],\\
\tilde r(\partial(d_1(k)\cup d_2(\ell)))&=[d_1(k+\ell+1),d_2(\ell)],\\
\tilde r(\partial(d_1(k)\cup d_3(\ell)))&=[d_1(k+\ell+1),d_3(\ell)],\\
\tilde r(\partial(d_2(k)\cup d_3(\ell)))&=[(\Pi_{\ell+1}^1(k))^{-1}d_3(\ell)\Pi_{\ell+1}^1(k),d_2(k)].
\end{align*}
Here $\vec a=(a_1,a_2)$, $\Pi_i^1(k)=A_2(k+1,i)A_2(k+1,i-1)\cdots
A_2(k+1,1)$, and $\vec\delta_k$ is the $k$-th coordinate unit vector.
Thus $B_nT_3$ is a commutator-related group and then we apply
Proposition~\ref{pro:dual} to assert that $[a,b]$ is in $\im(\Phi_{B_nT_3})$
up to $[F,[F,F]]$ for generators $a,b$ of $B_nT_3$ if and only if $a^*\cup b^*\ne 0$ in $H^2(B_nT_3)$ for the cohomology classes $a^*,b^*$ dual to the homology classes represented by $a,b$.

Suppose that $B_nT_3$ has a presentation $\langle x_1,\cdots,x_s \mid r_1,\cdots, r_t\rangle$ such that for each $i=1,\ldots,t$, $\Phi_{B_nT_3}(r_i)$ is a commutator of generators $x_1,\ldots,x_s$ up to $[F,[F,F]]$. Then $H^2(B_nT_3;\mathbb Z_2)$ is generated by $x_i^*\cup x_j^*$ for $1\le i,j\le s$ and there are no relations among $x_i^*\cup x_j^*$ that are nontrivial. Let $L$ be the graph with vertices $x_1,\ldots,x_s$ such that there is an edge between $x_i$ and $x_j$ if $x_i^*\cup x_j^*\ne 0$. Then the exterior face algebra $\Lambda(L)$ is isomorphic to the cohomology algebra $H^*(B_nT_3)$ modulo degree 3 or higher cohomology classes.

Now suppose that $B_nT_3$ is a right-angled Artin group. Then $H^*(B_nT_3)$ is an exterior face algebra of a flag complex $K$. Then for the 1-skeleton $K^1$ of $K$,
$\Lambda(K^1)$ is also isomorphic to $H^*(B_nT_3)$ modulo degree 3 or higher cohomology classes. As noted in the introduction of \S\ref{s:three}, a simplicial complex is completely determined by its exterior face algebras over $\mathbb Z_2$ and so $L$ is simplicially homeomorphic to $K^1$. Thus the number of triangles in $L$ is equal to that in $K$. Since $\rk(H_3(B_nT_3))$ is the number of triangles in $K^1$, $\rk(H_3(B_nT_3))$ is equal to the number of unordered triples $\{x_i,x_j,x_k\}$ such that $[x_i,x_j],[x_j,x_k],[x_k,x_i]\in\im(\Phi_{B_nT_3})$.

The presentation obtained from the Morse complex of $UD_nT_3$ is not yet
in the form described above because
\begin{multline*}
\Phi_{B_nT_3}(\tilde r(\partial(A_2(\vec a)\cup d_2(k))))\\
=[(A_2(k+1,a_2))^{-1}A_2(\vec
a+(k+1)\vec\delta_1),d_2(k)]\pmod{[F,[F,F]]}
\end{multline*}
which is not a commutator of generators. So we modify the
presentation by Tietze transformations.

For $a_1,a_2 >0$, we introduce new generators $\overline A(\vec a)$
by adding the relations $\overline A(\vec a+\vec\delta_1)=(A_2(\vec
a))^{-1} A_2(\vec a+\vec\delta_1)$ and $\overline A(1,a_2)=A_2(1,a_2)$,
and then eliminate the generators $A_2(\vec a)$. So we obtain a
presentation of $B_nT_3$ with two kinds $d_i(k)$ and $\overline
A(\vec a)$ of generators , and six kinds of relators as follows:
\begin{align*}
&[\overline A(\vec a),d_1(|\vec a|+k)],
&&[(\overline \Pi_{a_2-1}^1(k))^{-1}
\overline A(\vec a+(k+1)\vec\delta_1)
\overline \Pi_{a_2-1}^1(k),d_2(k)],\\
&[\overline A(\vec a+(k+1)\vec\delta_2),d_3(k)],
&&[d_1(k+\ell+1),d_2(\ell)],\\
&[d_1(k+\ell+1),d_3(\ell)],
&&[(\overline \Pi_{\ell+1}^1(k))^{-1}d_3(\ell)\overline \Pi_{\ell+1}^1(k),d_2(k)].
\end{align*}
Here $\overline \Pi_i^1(k)=\overline A(1,i)\cdots \overline
A(k+1,i)\overline A(1,i-1)\cdots \overline A(k+1,i-1)\cdots
\overline A(1,1)\cdots \overline A(k+1,1)$. Then the images of the
above relators under $\Phi_{B_nT_3}$ are given as follows:
\begin{align*}
&[\overline A(\vec a),d_1(|\vec a|+k)],
&&[\overline A(\vec a+(k+1)\vec\delta_1),d_2(k)],\\\displaybreak[2]
&[\overline A(\vec a+(k+1)\vec\delta_2),d_3(k)],
&&[d_1(k+\ell+1),d_2(\ell)],\\
&[d_1(k+\ell+1),d_3(\ell)],
&&[d_3(\ell),d_2(k)] \quad \pmod{[F,[F,F]]}.
\end{align*}

As explained earlier, if the braid group $B_nT_3$ were a
right-angled Artin group then $\rk(H_3(B_nT_3))$ is equal to the
number of unordered triples $\{x_1, x_2, x_3\}$ of generators such
that all of mutual commutators $[x_i,x_j]$ appear in the above
images. It is sufficient to consider the following four cases:
$$\{\overline A(\vec a),d_1(k),d_2(\ell)\},
\{\overline A(\vec a),d_2(k),d_3(\ell)\},
\{\overline A(\vec a),d_1(k),d_3(\ell)\},
\{d_1(k),d_2(\ell),d_3(m)\}.$$

Let $\overline N_{A,i,j}$ be the number of triples $\{\overline
A(\vec a),d_i(k),d_j(\ell)\}$ in which mutual commutators are all in
$\im(\Phi_{B_nT_3})$. In order for the commutators $[\overline
A(\vec a),d_i(k)]$, $[d_i(k),d_j(\ell)]$ and $[\overline A(\vec
a),d_j(\ell)]$ to appear in the images, the parameters must satisfy
the following inequalities:
\begin{align*}
n-1\ge k\ge |\vec a|,\:a_1\ge\ell+2,\:a_2\ge 1,\:\ell\ge 0\quad
&\text{if $(i,j)=(1,2)$},\\
n-1\ge k\ge |\vec a|,\:a_1\ge 1,\:a_2\ge\ell+2,\: \ell\ge 0\quad
&\text{if $(i,j)=(1,3)$},\\
n\ge |\vec a|,\:a_1\ge k+2,\:a_2\ge\ell+2,\:k,\ell\ge 0\quad
&\text{if $(i,j)=(2,3)$}.
\end{align*}
Let
\begin{align*}
n_1=n-k-1,\:n_2=k-|\vec a|,\:n_3=a_1-\ell-2,\:n_4=a_2-1\quad
&\text{if $(i,j)=(1,2)$},\\
n_1=n-k-1,\:n_2=k-|\vec a|,\:n_3=a_1-1,\:n_4=a_2-\ell-2\quad
&\text{if $(i,j)=(1,3)$},\\
n_1=n-|\vec a|,\:n_2=a_1-k-2,\:n_3=a_1-\ell-2,\:n_4=k\quad
&\text{if $(i,j)=(2,3)$}.
\end{align*}
Then we have the equation
$$n-4=n_1+n_2+n_3+n_4+\ell,\quad n_i,\ell\ge 0.$$
And $\overline N_{A,i,j}$ is the number of nonnegative integer
solutions for the quintuple $(n_1,n_2,\allowbreak n_3,n_4,\ell)$ to
the above equation. So $\overline N_{A,i,j}=\binom{5+(n-4)-1}{n-4
}=n(n-1)(n-2)(n-3)/24$ and so $\overline N_{A,i,j}=N_{A,i,j}$.

Let $\overline N_{1,2,3}$ be the number of triples $\{d_1(k),
d_2(\ell),d_3(m)\}$ such that $[d_1(k),d_2(\ell)]$,
$[d_2(\ell),d_3(m)]$, and $[d_1(k),d_3(m)]$ appear in
$\im(\Phi_{B_nT_3})$. The following inequalities must holds:
$$n-1\ge k\ge \ell+1,m+1,\quad n-2\ge \ell+m,\quad \ell,m\ge 0.$$
Then $\overline N_{1,2,3}$ is equal to the number of nonnegative
integer solutions for the triples $(k,\ell,m)$ to the inequalities.
Recall $N_{1,2,3}=n(n-1)(n-2)/6$. We regard $N_{1,2,3}$ as the
number of subsets $\{a,b,c\}\subset\{1,\cdots,n\}$. We can assume
that $a>b>c$. Define a function $f:\{\{a,b,c\}\}\to \{(k,\ell,m)\}$
defined by as follows:
$$ f(\{a,b,c\})=\left\{\begin{array}{lr}
(a-1,b-1,c-1)&{\rm if} \quad b+c\le n \\
(n-c,n-a,n-b)&{\rm otherwise}
\end{array} \right.$$
Then it is easy to see that $f$ is well defined and injective. And
the image of $f$ does not contain $(1,0,0)$. So $\overline N_{1,2,3}
> N_{1,2,3}$. This is a contradiction because $\rk(H_3(B_nT_3))$
computed from the presentation under the assumption that  $B_nT_3$
is a right-angled Artin group is larger than that obtained from the
Morse complex of the configuration space.
\end{proof}

\section{Graphs containing $S_0$}\label{s:four}
Recall the homomorphism $\Phi_G: R/[F,R]\to [F,F]/[F,[F,F]]$ induced
by the inclusion $R\hookrightarrow [F,F]$ for a commutator-related
group $G=F/R$. If $G$ is a right-angled Artin group, then $\Phi_G$
is injective since $\Phi_G$ is regarded as the dual of the cup
product $H^1(G)\wedge H^1(G) \to H^2(G)$ which is surjective for a
right-angled Artin group $G$. We state this as a proposition
together with a simple proof.

\begin{pro}\label{pro:injrightangled}
If $G$ is a right-angled Artin group then $\Phi_G$ is injective.
\end{pro}

\begin{proof}
Suppose that $G=F/R$. Since $G$ is a right-angled Artin group, $R$
is normally generated by a collection $C$ of commutators of
generators. Modulo $[F,R]$, every conjugation acts trivially on $C$
and $R/[F,R]$ is identified to the free abelian group generated by
$C$. Similarly, $[F,F]/[F,[F,F]]$ is identified to the free abelian
group generated by basic commutators by the Hall basis theorem.
Since elements of $C$ are basic commutators, $\Phi_G$ is merely an
inclusion under these identifications,
\end{proof}

The converse of the proposition is not true. The braid groups of the
tree $T_0$, and all the graphs in Figure~\ref{fig8} are not a
right-angled Artin group but the homomorphism $\Phi_G$ is injective
for these groups. For example, it is not hard to see that
$\Phi_{B_nT_3}$ is injective from the presentation of $B_nT_3$ given
in the proof of Theorem~\ref{thm:T_3}.  In fact, it is sufficient to
check that $\rk(H_2(B_nT_3))=\rk(\im \Phi_{B_nT_3})$ since they are
free abelian group. The former rank can be obtained from the Morse
complex and the latter rank can be seen from the presentation.

In contrast, we will prove in \S\ref{ss41:s0theta} that the braid
groups of $S_0$ and the $\theta$-shaped graph $\Theta$ are not a
right-angled Artin group by showing that $\Phi_G$ is not injective.

We now consider a homomorphism between two commutator-related
groups. Let $G=F/R$ and $G'=F'/R'$ be commutator-related groups so
that $R\subset [F,F]$ and $R'\subset [F',F']$, and let $h:G'\to G$
be a homomorphism. Then it is clear that the diagram

\begin{displaymath}
\begin{CD}
R'/[F',R'] @>{h_*}>>R/[F,R]\\
@V{\Phi_{G'}}VV @VV{\Phi_G}V\\
[F',F']/[F',[F',F']] @>{\bar h}>>[F,F]/[F,[F,F]]
\end{CD}
\end{displaymath}
commutes where $h_*$ and $\bar h$ are homomorphisms derived from
$h$.

The assumption that $\Phi_{G'}$ is not injective and $h_*$ is
injective implies $\Phi_G$ is not injective and so $G$ is not a
right-angled Artin group  by Proposition~\ref{pro:injrightangled}.
In order to show that $B_n\Gamma$ is not a right-angled Artin group for a given graph $\Gamma$, we introduce a graph $\Gamma'$ contained in $\Gamma$ such that $\Phi_{B_n\Gamma'}$ is defined and is not injective. Then we choose an embedding $i:\Gamma'\to\Gamma$. We may assume that $B_n\Gamma$ is commutator-related. Otherwise, we are done. If we show that $i$ induces an injection on the second homologies of braid groups, then $\Phi_{B_n\Gamma}$ is not injective and we conclude that  $B_n\Gamma$ is not a right-angled Artin group.

We try to apply this strategy to all graphs containing $S_0$.
Unfortunately, there are graphs $\Gamma$ such that no embedding
$i:S_0\to\Gamma$ induces an injection $i_*:H_2(B_nS_0)\allowbreak\to
H_2(B_n\Gamma)$. In fact, $\Theta$ is one of such examples as shown
in the following lemma.

\begin{lem}
For $n\ge 5$, there is no embedding $i:S_0\to\Theta$ that induces an
injection $i_*:H_2(B_nS_0)\to H_2(B_n\Theta)$.
\end{lem}

\begin{proof}
It is sufficient to show that
$\rk(H_2(B_nS_0))>\rk(H_2(B_n\Theta))$. Indeed,
$\rk(H_2(B_nS_0))=(n-3)(n-2)(n-1)/6$ as we will see in the proof of
Lemma~\ref{lem:S0}, and $\rk(H_2(B_n\Theta)) = (n-3)(n-2)/2$ as we
will see in the remark following Lemma~\ref{lem:Theta}.
\end{proof}

If $\Gamma$ contains $S_0$ but does not contain $\Theta$ then there
is an embedding $i:S_0\to\Gamma$ that induces an injection
$i_*:H_2(B_nS_0)\to H_2(B_n\Gamma)$ as shown in
Lemma~\ref{lem:S0toGamma}. On the other hand, if $\Gamma$ contains $\Theta$, then
it seems difficult to directly show that there is an embedding
$i:\Theta\to\Gamma$ that induces an injection $i_*:H_2(B_n\Theta)\to
H_2(B_n\Gamma)$. Fortunately, $\Phi_{B_n\Theta}$ is the trivial map
as shown in Lemma~\ref{lem:Theta} and therefore in order to show
that $\Phi_{B_n\Gamma}$ is not injective, it suffices to construct an embedding
$i:\Theta\to\Gamma$ that induces
a nontrivial homomorphism $i_*:H_2(B_n\Theta)\to H_2(B_n\Gamma)$, which is shown in Lemma~\ref{lem:ThetatoGamma} under the
extra assumption that $\Gamma$ is planar.

For non-planar graphs, we use yet another method to prove that non-planar
graph braid groups are not a right-angled Artin group. In
\S\ref{ss44:nonplanar}, it will be shown that for a non-planar graph
$\Gamma$ and $n\ge 2$, $H_1(B_n\Gamma)$ always has a torsion and so
$B_n\Gamma$ cannot be a right-angled Artin group.

Consequently, the following theorem is proved by combining results
in this section.

\begin{thm}\label{thm:graphwithS_0}
If a graph $\Gamma$ contains $S_0$, then $B_n\Gamma$ is not a
right-angled Artin group for $n\ge 5$.
\end{thm}

To obtain lemmas in \S\ref{ss43:h2inj} and \S\ref{ss44:nonplanar},
it is necessary to extract certain information on boundary
homomorphisms in Morse complexes. Since these computations are
lengthy and  rather technical, they are separately presented in
\S\ref{ss42:boundary homomorphisms}.

\subsection{Braid groups of $S_0$ and $\Theta$}\label{ss41:s0theta}
In Example~\ref{ex:S_0}, $B_4S_0$ was shown to be a right-angled
Artin group. One can show that $B_3\Theta$ is also a right-angled
Artin group. In this section, we show that $\Phi_{B_nS_0}$
($\Phi_{B_n\Theta}$) is not injective for $n\ge 5$ ($n\ge 4$,
respectively). By Proposition~\ref{pro:injrightangled}, this implies
the braid groups of $S_0$ and $\Theta$ are not a right-angled Artin
group for the higher braid indices.

\begin{lem}\label{lem:S0}
For $n\ge 5$, $\Phi_{B_nS_0}$ is not injective and so $B_nS_0$ is
not a right-angled Artin group.
\end{lem}

\begin{proof}
\begin{figure}[ht]
\psfrag{*}{\small0}
\psfrag{A}{\small$A$}
\psfrag{e}{\small$d$}
\centering
\includegraphics[height=2.3cm]{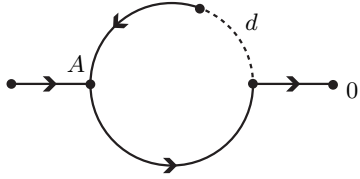}
\caption{Maximal tree and order of $S_0$}\label{fig13}
\end{figure}

Choose a maximal tree of $S_0$ and give an order as in
Figure~\ref{fig13}. Let $d(k)$ denote the set consisting of the
deleted edge $d$ together with $k$ vertices blocked by $\tau(d)$. The configuration space $UD_nS_0$ has
two kind critical 1-cells: $d(k)$, $A_2(a,b)$ for $0\le k\le n-1$,
$1\le a,b$, and $a+b\le n$, and critical 2-cells $A_2(a,b) \cup
d(\ell)$ for $1\le a,b$, $0\le\ell$, and $\ell+a+b\le n-1$, and no
critical $i$-cells for $i\ge 3$. We obtain relations
by rewriting the boundaries of critical 2-cells. That is,
$$\tilde r(\partial(A_2(a,b)\cup d(\ell))) =
(d(\ell+a+b))(A_2(a,b))(d(\ell+a+b))^{-1}(A_2(a,b+1))^{-1}.$$

When $\ell+a+b=n-1$, we have the relations $A_2(a,b+1)=d(n-1)
A_2(a,b)d(n-1)^{-1}$. Applying Tietze transformations on
$A_2(a,b+1)=d(n-1)^b A_2(a,1)d(n-1)^{-b}$, we eliminate the
generators $A_2(a,b+1)$ for $b\ge 1$. For each $\ell+a+b\le n-2$, the boundary word is rewritten as $[d(n-1)^{-b} d(\ell+a+b) d(n-1)^{b-1},A_2(a,1)]$ . Let $m=\ell+a+b$. We introduce
new generators $\bar d(m)$ by adding the relations $\bar
d(m)=(d(n-1))^{-1}d(m)$, and then eliminate the generators $d(m)$.
We obtain a presentation of $B_n S_0$ whose generators are $\bar
d(0),\ldots,\bar d(n-2)$, $d(n-1)$, $A_2(1,1),\ldots, A_2(n-1,1)$
and whose relators are $[d(n-1)^{1-b}\bar d(m)d(n-1)^{b-1},
A_2(a,1)]$ for $1\le a,b$ and $2\le a+b\le m\le n-2$. Thus $B_nS_0$
is a commutator-related group and so the homomorphism
$\Phi_{B_nS_0}$ is defined and
$$\Phi_{B_nS_0}([d(n-1)^{1-b}\bar d(m) d(n-1)^{b-1},
A_2(a,1)])= [\bar d(m),A_2(a,1)] \pmod{[F,[F,F]]}$$ where $F$ is the
free group on the set of generators. Since $b$ becomes irrelevant, we set $b=1$ to enumerate the number of commutators $[\bar d(m),A_2(a,1)]$. Then we have $2\le
a+1\le m\le n-2$. The number of pairs
$(a,m)$ satisfying the inequalities is $(n-2)(n-3)/2$ which is equal
to $\rk(\im \Phi_{B_nS_0})$.

On the other hand, the Morse chain complex of $UD_n S_0$ looks like
\begin{displaymath}
\begin{CD}
0@>>>M_2(UD_nS_0)@>\widetilde\partial >>M_1(UD_nS_0)@>{0}>>\mathbb
Z.
\end{CD}
\end{displaymath}
Then $\rk(H_2( B_n S_0))=\rk(M_2(UD_n S_0))- \rk(\im
\widetilde\partial)$. Since $\widetilde\partial(A_2(a,b)\cup
d(\ell))=A_2(a,b+1)-A_2(a,b)$, $\ell$ is irrelevant.
We set $\ell=0$. Then $\im\widetilde\partial$ is generated by $A_2(a,b+1)-A_2(a,b)$ for $1\le a,b$ and $a+b\le n-1$. The number of pairs $(a,b)$ satisfying the
inequalities is $(n-1)(n-2)/2$ which is $\rk(\im
\widetilde\partial)$. And $\rk(M_2(UD_nS_0))$ is equal to the number
of triples $(a,b,\ell)$ satisfying $a,b\ge 1$, $\ell\ge 0$, and
$a+b+\ell\le n-1$. So $\rk(M_2(UD_n S_0))=n(n-1)(n-2)/6$. Thus
$\rk(H_2( B_n S_0))=(n-1)(n-2)(n-3)/6$.

Consequently, since $(n-2)(n-3)/2<(n-1)(n-2)(n-3)/6$ for $n\ge 5$,
$\Phi_{B_nS_0}$ cannot be injective.
\end{proof}

\begin{lem}\label{lem:Theta}
For $n\ge 4$, $\Phi_{B_n\Theta}$ is the trivial map. However,
$H_2(B_n\Theta)$ is not trivial and so $B_n\Theta$ is not a
right-angled Artin group.
\end{lem}

\begin{proof}
\begin{figure}[ht]
\psfrag{*}{\small0}
\psfrag{A}{\small$Y$}
\psfrag{e1}{\small$d_1$}
\psfrag{e2}{\small$d_2$}
\centering
\includegraphics[height=2.2cm]{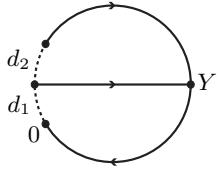}
\caption{Maximal tree and order of $\Theta$}\label{fig14}
\end{figure}

Choose a maximal tree of $\Theta$ and give an order as in
Figure~\ref{fig14}.  Then $UD_n \Theta$ has two kind critical
1-cells: $d_\ell$, $Y_2(a,b)$ for $1\le a,b$, $a+b\le n$,
$\ell=1,2$, and critical 2-cells $Y_2(a,b)\cup d_\ell$ for $1\le
a,b$, $a+b\le n-1$, $\ell=1,2$,  and no critical $i$-cells for $i
\ge 3$.

There are two type relations obtained by rewriting boundary words of
critical 2-cells as follows:
\begin{align*}
Y_2(a+1,b)&=\Pi^1_b d_1 Y_2(a,b) d_1^{-1}(\Pi^1_{b-1})^{-1},\\
Y_2(a,b+1)&=d_2(\Pi^1_b)^{-1} Y_2(a+1,b)\Pi^1_{b-1}d_2^{-1}
=d_2 d_1 Y_2(a,b)(d_2 d_1)^{-1}.
\end{align*}
where $\Pi^1_k=Y_2(1,k)Y_2(1,k-1)\cdots Y_2(1,1)$. Using the
relations
\begin{align*}
Y_2(a+1,1)&=Y_2(1,1)d_1 Y_2(a,1) d_1^{-1}=(Y_2(1,1)d_1)^{a} Y_2(1,1) d_1^{-a},\\
Y_2(a,b+1)&=(d_2 d_1)^{b} Y_2(a,1)(d_2 d_1)^{-b}\\
&=(d_2 d_1)^{b}(Y_2(1,1)d_1)^{a-1} Y_2(1,1) d_1^{1-a}(d_2 d_1)^{-b},
\end{align*}
we eliminate the generators $Y_2(a,b)$ except $Y_2(1,1)$. Then there
are three generators $d_1$, $d_2$, and $Y_2(1,1)$, and one type
relators as follows: For $1\le a$, $2\le b$, and $a+b\le n-1$,
\begin{align*}
&[d_1^{-1}\{(d_2d_1)^{-1}Y_2(1,1)\}^{b-1}d_2^{-1}(d_2d_1)^b,\ (Y_2(1,1)d_1)^{a}d_1^{-a} ].
\end{align*}

Thus $B_n \Theta$ is a commutator-related group and we can consider
$\Phi_{B_n \Theta}$. It is easy to see that $\Phi_{B_n \Theta}$
annihilates the relations above. So $\rk(\im \Phi_{B_n \Theta})=0$.
But $H_2(B_n\Theta)$ is not trivial because there are
commutator relators for $n\ge 4$. Thus $\rk(H_2(B_n\Theta)) \neq 0$.
Thus $\Phi_{B_n \Theta}$ cannot be injective.
\end{proof}

\paragraph{\bf Remark.} With not much extra work, we can now compute
$\rk(H_2(B_n\Theta)))$.
Since $B_n\Theta$ is a commutator-related group over three generators
as in the proof of Lemma~\ref{lem:Theta},
$\rk(H_1(B_n\Theta))=3$. Since $\rk(M_1(UD_n\Theta))=n(n-1)/2+2$,
$$\rk(\im(\widetilde\partial_2))=\rk(M_1(UD_n\Theta))
-\rk(H_1(B_n\Theta))=(n-2)(n+1)/2.$$
Since $\rk(M_2(UD_n\Theta))=(n-1)(n-2)$,
$$\rk(\ker(\widetilde\partial_2))=\rk(M_2(UD_n\Theta))
-\rk(\im(\widetilde\partial_2))=(n-2)(n-3)/2.$$ Since
$M_3(UD_n\Theta)=0$,
$$\rk(H_2(B_n\Theta)))=\rk(\ker(\widetilde\partial_2))=(n-2)(n-3)/2.$$

\subsection{Boundary homomorphisms in Morse complexes}\label{ss42:boundary homomorphisms}
In this section we will discuss the behavior of boundary
homomorphisms in Morse chain complexes. In general, it is not easy
to fully compute boundary homomorphisms for Morse complexes of even
small graphs such as $K_5$. One of the reasons is that there are too many
kinds of critical cells. Another reason is that Morse complexes
heavily depend on geometric information such as choices of maximal
trees and orderings. To show that the first homology of a non-planar
braid group has a torsion, we need to understand the second boundary
homomorphism of its Morse complex under a particular choice of a
maximal tree. To show that braid groups of planar graphs that contain
$S_0$ is not right-angled Artin groups, we need to look into
certain images under the third boundary homomorphism.

Recall that $\widetilde\partial=\widetilde R\partial$ and the image
of an $i$-cell $c$ in $UD_n\Gamma$ under $\partial$ is the sum of
$i$ pairs of faces of $c$ in opposite sign. So if the images of a
pair of faces under $\widetilde R$ are identical, then this pair has
no contribution  in $\widetilde\partial(c)$. This is why
$\widetilde\partial$ for any graph not containing $S_0$ is trivial
as shown in \S\ref{ss31:homology}. We introduce useful properties
for $\widetilde R$ before getting into the discussing images under
$\widetilde\partial$.

Let $c$ be a redundant $i$-cell in $UD_n\Gamma$, $v$ be an unblocked
vertex in $c$ and $e$ be the edge in $T$ starting from $v$.  The
vertex $v$ is said to be {\em simply unblocked} if there is no
vertex $w$ that is either in $c$ or an end vertex of an edge in $c$
and satisfies $\tau(e)< w<\iota(e)$. Let $W_{e}(c)$ ($V_{e}(c)$, respectively)
denote the $(i+1)$-cell (the $i$-cell) obtained from $c$ by
replacing $v$ by $e$ (by $\tau(e)$).

\begin{lem}\label{lem:naturalmove}
Suppose a redundant cell $c$ has a simply unblocked vertex $v$ and
$e$ is the edge starting from $v$. Then $\widetilde R(c)=\widetilde
RV_e(c)$.
\end{lem}
\begin{proof}
Every cell $c_j$ except $c$ and $V_e(c)$ that is a summand in
$\partial W_e(c)$, has an order-respecting edge $e$ starting from
$v$ that vacuously satisfies the hypothesis of
Lemma~\ref{lem:redundant1}. Thus $\widetilde R(c_j)=0$ as well as
$\widetilde R(W_e(c))=0$. Then
$$0=\widetilde\partial\widetilde R(W_e(c))=\widetilde
R\partial(W_e(c))=\widetilde R(\pm(c-V_e(c))+\sum_jc_j).$$
Thus $\widetilde R(c)=\widetilde RV_e(c)$
\end{proof}

Since $\widetilde\partial=\widetilde R\partial$, the computation of
$\widetilde\partial(c)$ for a cell $c$ in $UD_n\Gamma$ involves
numerous rewriting $\widetilde R$ and so it is important to have an
efficient way to compute $\widetilde R$.  Our strategy is basically to
repeat the following two steps:
\begin{enumerate}
\item[(I)] Via repeated applications of this lemma to the smallest
simply unblocked vertex, we may replace any redundant cell $c$
by another redundant cell $c'$ such that $\widetilde
R(c)=\widetilde R(c')$ and $c'$ has no simply unblocked
vertices.
\item[(II)] Compute $R(c')$, that is, replace $c'$ by the alternating sum
of faces of $W(c')$ other than $c'$. Go back to Step (I) for
each redundant cell in the sum.
\end{enumerate}

We now consider the relationship between the choice of maximal tree
and the dimension of the corresponding Morse complex. An edge in a
critical cell can be either a deleted edge or a non-order-respecting
edge together with a blocked vertex near a vertex of valency $\ge
3$. Thus the dimension of the Morse complex over a graph $\Gamma$
under the choice of maximal tree $T$ is bounded above by $$\min\left\{n,
N_d+\lfloor (n-N_d)/2\rfloor, N_d+N_v\right\}$$ where $n$ is the
braid index, $N_d$ is the number of deleted edges in $\Gamma$, and
$N_v$ is the number of vertices of valency $\ge 3$ in $T$.

When the braid index is relatively large, we had better obtain a
maximal tree by deleting edges incident to a vertex of valency $\ge
3$ in order to achieve the minimal dimension of a Morse complex.
This is how we chose the maximal trees in Lemma~\ref{lem:S0} and
Lemma~\ref{lem:Theta}.

However, this choice is not so wise when we consider the second
boundary homomorphism in Morse complexes of non-planar graphs
because deleted edges may block other vertices and this makes the
rewriting process extremely complicated. Thus for a non-planar graph
$\Gamma$, we will choose a  maximal tree $T$  by deleting edges
whose ends are of valency two in $\Gamma$. By choosing a planar
embedding of $T$, vertices of $\Gamma$ are ordered and $UD_n \Gamma$
collapses to a Morse complex as explained in \S\ref{ss22:morse}.
Then there are no vertices blocked by deleted edges due to the
construction of $T$ and there are three kinds of critical 2-cells
$A_k(\vec a) \cup B_\ell(\vec b)$, $A_k(\vec a)\cup d$, and $d\cup
d'$, where $A$ and $B$ are vertices with valency $\ge 3$, and $d$
and $d'$ are deleted edges. Remember that we are omitting the $0_s$
part, that is, a sequence of $s$ vertices blocked by the base vertex.
We note that a critical 2-cells $A_k(\vec a) \cup B_\ell(\vec b)$ is contained in
the tree $T$ and vanishes under $\widetilde\partial$ as shown in
\cite{Far} or in Theorem~\ref{thm:homology}. In \S\ref{ss44:nonplanar}, we need to know the image of critical 2-cells $d\cup d'$ under $\widetilde\partial : M_2(UD_n \Gamma)\to M_1(UD_n \Gamma)$. To help understanding, we first look at an
example.

\begin{figure}[ht]
\psfrag{0}{\small0}
\psfrag{A}{\small$A$}
\psfrag{B}{\small$B$}
\psfrag{d}{\small$d$}
\psfrag{d1}{\small$d'$}
\centering
\includegraphics[height=4cm]{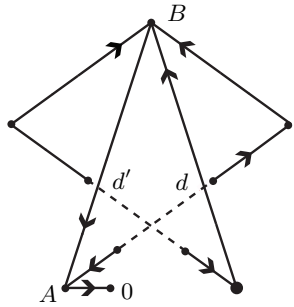}
\caption{Graph containing $c=d\cup d'$}
\label{fig15}
\end{figure}

\begin{exa}\label{ex:caseb}
Let $\Gamma$ be a graph in Figure~\ref{fig15} and a maximal tree
and an order be given as the figure. As always, we assume $\Gamma$
is subdivided sufficiently according to the braid index. We want to
compute $\widetilde\partial(c)$ for the 2-cell $c=d \cup
d'$ in $M_2(UD_n\Gamma)$.
\end{exa}
Since $\iota(d')<\iota(d)$, using
Lemma~\ref{lem:naturalmove} we have
\begin{align*}
\widetilde\partial(c)=&\:\widetilde R(d'\cup \iota(d))-\widetilde R(d'\cup \tau(d))-\widetilde R(d\cup\iota(d'))+\widetilde R(d\cup\tau(d'))\\
=&\:\widetilde R(d'\cup A(0,1))-\widetilde R(d'\cup B(0,1,0))-\widetilde R(d\cup B(0,0,1))+d
\end{align*}
Since $d'\cup A_2(0,1)$ is collapsible, using
Lemma~\ref{lem:naturalmove}, we have
\begin{align*}
0=&\:\widetilde\partial\widetilde R(d'\cup A_2(0,1))=\widetilde R\partial(d'\cup A_2(0,1))\\
=&\:\widetilde R(d'\cup A(0,1)-d'\cup \{A\}-A_2(0,1)\cup\iota(d')+A_2(0,1)\cup\tau(d'))\\
=&\:\widetilde R(d'\cup A(0,1))-d'-A_2(1,1)+A_2(1,1)
\end{align*}
and so $\widetilde R(d'\cup A(0,1))=d'$. Similarly,
\begin{align*}
0=&\:\widetilde\partial\widetilde R(d'\cup B_2(0,1,0))=\widetilde R(\partial(d'\cup B_2(0,1,0))\\
=&\:\widetilde R(B_2(0,1,0)\cup\iota(d')-B_2(0,1,0)\cup\tau(d')-d'\cup B(0,1,0)+d'\cup\{B\})\\
=&\:\widetilde R(B_2(0,1,0)\cup\iota(d'))-B_2(1,1,0)-\widetilde R(d'\cup B(0,1,0))+d'.
\end{align*}
So $\widetilde R(d'\cup B(0,1,0))=d'-B_2(1,1,0)$ since $B_2(0,1,0)\cup\iota(d')$ is collapsible.
\begin{align*}
0=&\:\widetilde\partial\widetilde R(d\cup B_3(0,0,1))=\widetilde R(\partial(d\cup B_3(0,0,1))\\
=&\:\widetilde R(B_3(0,0,1)\cup\iota(d)-B_3(0,0,1)\cup\tau(d)-d\cup B(0,0,1)+d\cup\{B\})\\
=&\:\widetilde R(B_3(0,0,1)\cup\iota(d))-B_3(0,1,1)-\widetilde R(d\cup B(0,0,1))+d.
\end{align*}
So $\widetilde R(d\cup B(0,0,1))=d-B_3(0,1,1)$ since $B_3(0,0,1)\cup\iota(d)$ is collapsible.
Consequently, we have
$$\widetilde\partial(c)=B_2(1,1,0)+B_3(0,1,1).\qed$$

In order to give a general formula for $\widetilde\partial(d\cup d')$, we classify them into several cases
depending on the relative locations of end vertices of
deleted edges. The following notations are handy for this purpose and will be used to compute the boundary of critical 3-cells later.
For each vertex $v$ in
$\Gamma$, there is a unique edge path $\gamma_v$ from $v$ to the
base vertex 0 in the tree $T$. For vertices $v, w$ in $\Gamma$,
$v\wedge w$ denotes the vertex that is the first intersection
between $\gamma_v$ and $\gamma_w$. Obviously, $v\wedge w\le v$ and
$v\wedge w\le w$. For a vertex $v$ of valency $\mu$, we assign
$0,1,\ldots,\mu-1$ to the branches incident to $v$ clockwise
starting from the branch joined to the base vertex. Then $g(v,w)$
denotes the number assigned to the branch of $v$ containing $w$. In
the case of $v$ with valency two, $g(v,w)=1$ if $v=w\wedge v$, and
$g(v,w)=0$ if $v>w\wedge v$.

Although the following lemma will be used later in \S\ref{ss44:nonplanar}, we present it here in order to keep mechanical notations and ideas in one place and its proof is somewhat helpful to understand the rest of this technical section.

\begin{lem}\label{lem:casec}
Suppose that a Morse complex is obtained by choosing a maximal tree $T$ of a (non-planar) graph $\Gamma$ so that both ends of every deleted edge have valency two in $\Gamma$. For a critical 2-cell $d\cup d'$in
$M_2(UD_n\Gamma)$ such that $d$ and $d'$ are deleted edges with
$\tau(d)<\tau(d')$, let $A=\tau(d')\wedge\tau(d)$,
$B=\tau(d')\wedge\iota(d)$, $C=\iota(d') \wedge\tau(d)$ and
$D=\iota(d')\wedge\iota(d)$. Then the
images under the boundary map are as follows:\\
{\rm(1)} If $\tau(d)<\iota(d)<\tau(d')<\iota(d')$, then
       \begin{align*}
            \widetilde\partial(d\cup d') =
             &\:A_{g(A,\tau(d'))}(\vec\delta_{g(A,\tau(d'))}+ \vec\delta_{g(A,\tau(d))})\\
             &- B_{g(B,\tau(d'))}(\vec\delta_{g(B,\tau(d'))} + \vec\delta_{g(B,\iota(d))})\\
             &- C_{g(C,\iota(d'))}(\vec\delta_{g(C,\iota(d'))} + \vec\delta_{g(C,\tau(d))})\\
             &+ D_{g(D,\iota(d'))}(\vec\delta_{g(D,\iota(d'))} + \vec\delta_{g(D,\iota(d))})
        \end{align*}
{\rm(2)} If $\tau(d)<\tau(d')<\iota(d')< \iota(d)$, then
        \begin{align*}
            \widetilde\partial (d \cup d') =&- A_{g(A,\tau(d'))}(\vec\delta_{g(A,\tau(d'))}+\vec\delta_{g(A,\tau(d))})\\
            & - B_{g(B,\iota(d))}(\vec\delta_{g(B,\tau(d'))} + \vec\delta_{g(B,\iota(d))})\\
            & + C_{g(C,\iota(d'))}(\vec\delta_{g(C,\iota(d'))} + \vec\delta_{g(C,\tau(d))})\\
            & + D_{g(D,\iota(d))}(\vec\delta_{g(D,\iota(d'))} + \vec\delta_{g(D,\iota(d))}),
        \end{align*}
{\rm(3)} If $\tau(d)<\tau(d')<\iota(d)<\iota(d')$, then
        \begin{align*}
            \widetilde\partial(d\cup d') =
            &A_{g(A,\tau(d'))}(\vec\delta_{g(A,\tau(d'))} + \vec\delta_{g(A,\tau(d))})\\
            &+ B_{g(B,\iota(d))}(\vec\delta_{g(B,\tau(d'))} + \vec\delta_{g(B,\iota(d))})\\
            &- C_{g(C,\iota(d'))}(\vec\delta_{g(C,\iota(d'))} + \vec\delta_{g(C,\tau(d))})\\
            &+ D_{g(D,\iota(d'))}(\vec\delta_{g(D,\iota(d'))} + \vec\delta_{g(D,\iota(d))})
        \end{align*}
where there may be 1-cells that are not critical and should be
ignored.
\end{lem}
\begin{proof}
The boundary of a critical 2-cell made of only deleted edges
consists of four 1-cells of the form $d\cup \{v\}$ where $d$ is a
deleted edge and $v$ is a vertex. Since both ends of every deleted
edge are of valency two in $\Gamma$, $v$ is a vertex of valency one
in the tree $T$. Let $E=\iota(d) \wedge v$ and $F=\tau(d) \wedge v$.
Then $$\widetilde R(d\cup \{v\})
=d+E_{g(E,v)}(\vec\delta_{g(E,v)}+\vec\delta_{g(E,\iota(d))})-F_{g(F,v)}(\vec\delta_{g(F,v)}+\vec\delta_{g(F,\tau(d))}).$$
If $v<\iota(d)$, that is, $g(E,v)<g(E,\iota(d))$, then the term
$E_{g(E,v)}(\vec\delta_{g(E,v)}+\vec\delta_{g(E,\iota(d))})$ represents a
collapsible 1-cell and so disappears in the formula.

Given two deleted edges $d$ and $d'$, there are three possible
relative positions of them as classified in the lemma. In the case
(1), since $\iota(d)<\iota(d')$,
$$\widetilde\partial(d\cup d') = \widetilde R(d\cup \iota(d'))-
\widetilde R(d\cup \tau(d')) -\widetilde R(d'\cup \iota(d)) +
\widetilde R(d'\cup \tau(d)).$$
Since $\tau(d)<\tau(d')<\iota(d')$, $\widetilde R(d'\cup
\tau(d))=d'$. And since $\iota(d)<\tau(d')<\iota(d')$, $\widetilde
R(d'\cup \iota(d))=d'$. So we obtain the desired formula. Similarly,
we can obtain the other two formulae.
\end{proof}

Next we consider the third boundary homomorphisms in Morse complexes
that is needed to prove that braid groups of planar graphs
containing $S_0$ are not right-angled Artin groups. Let $\Gamma'$ be
either $S_0$ or $\Theta$. Choose a maximal tree $T'$ of $\Gamma'$
and its planar embedding as in Figure~\ref{fig16} so that there are
no $i$-cells in $UD_n\Gamma'$ for $i\ge3$.  As explained in
\S\ref{ss22:morse}, a planar embedding of $T'$ and a choice of a
vertex as a base vertex determine an order on vertices of $\Gamma'$.
Now suppose there is an embedding $i:\Gamma' \to \Gamma$ into
another planar graph $\Gamma$. We may choose a maximal tree $T$ of $\Gamma$
so that $i|_{T'}:T'\to T$ is an embedding. Then the image of the
base vertex $T'$ becomes naturally the base vertex of $T$ and we may
extend the planar embedding of $T'$ over $T$. The extended planar
embedding of $T$ together with the choice of base vertex determine an
order on vertices of $\Gamma$. Then it is easy to see that the
embedding $i$ preserves the base vertex, the order, and deleted edges.
Thus the embedding $i$ naturally induces the chain map $\tilde i$
between two Morse chain complexes and so it induces homomorphisms
$i_*:H_*(B_n\Gamma')\to H_*(B_n\Gamma)$. Note that the requirement
of an embedding $i:\Gamma' \to \Gamma$ is simpler than that in
\S\ref{ss33:embedding} because the braid indices for $\Gamma'$ and
$\Gamma$ are the same in this section.

\begin{figure}[ht]
\psfrag{*}{\small0}
\psfrag{A}{\small$Y$}
\psfrag{e1}{\small$d_1$}
\psfrag{e2}{\small$d_2$}
\centering
\subfigure[$S_0$]
{\includegraphics[height=2cm]{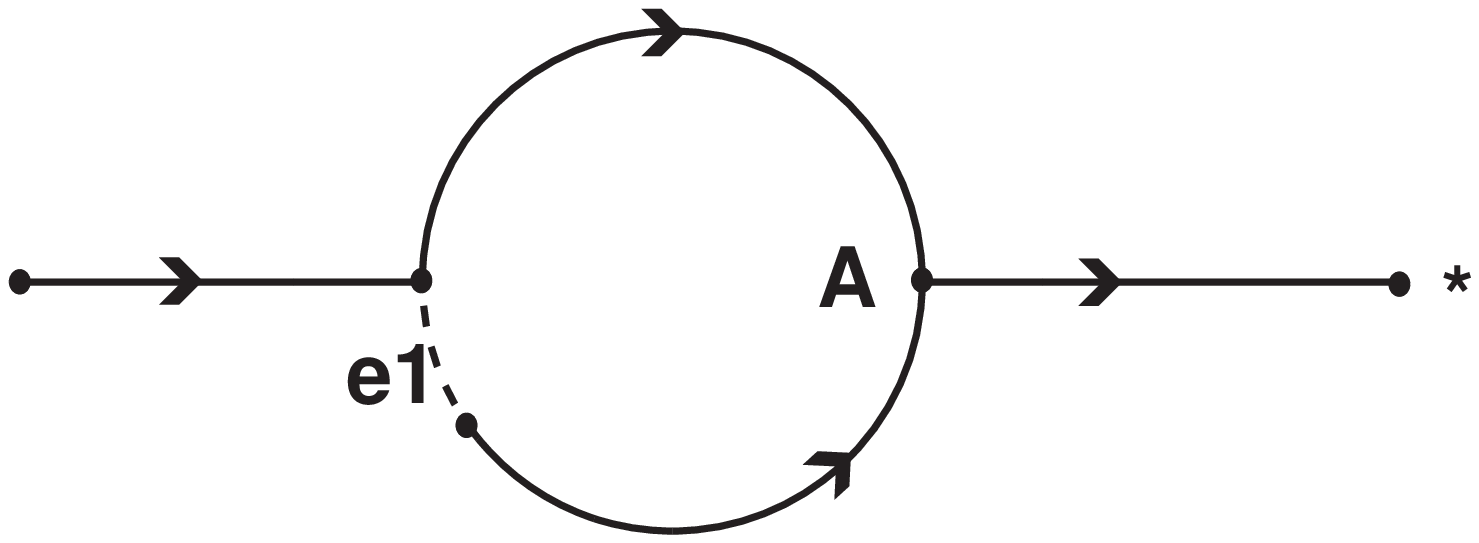}}\qquad
\subfigure[$\Theta$]
{\includegraphics[height=2cm]{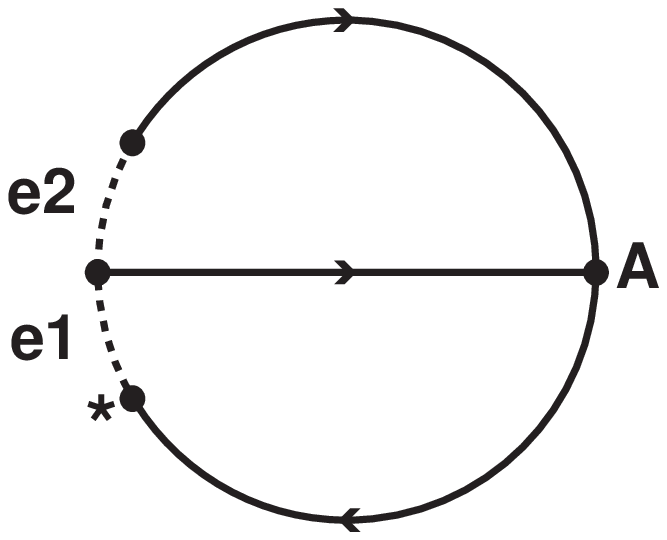}}
\caption{Maximal trees and orders of $S_0$ and $\Theta$}\label{fig16}
\end{figure}

It is difficult to give a general formula for the boundary
homomorphism $\widetilde\partial : M_3(UD_n\Gamma)\to
M_2(UD_n\Gamma)$. Fortunately, all we will need in \S\ref{ss43:h2inj} is
the relationship between $\widetilde\partial(M_3(UD_n\Gamma))$ and
$\widetilde i(M_2(UD_n\Gamma'))$. Using the notations in
Figure~\ref{fig16}, every critical 2-cell in $UD_nS_0$ must
contain $Y_2(a,b)\cup d_1$ as a subset and every critical 2-cell in
$UD_n\Theta$ is the form $Y_2(a,b)\cup d_j$ for $j=1,2$. It turns out that we need not consider critical 2-cells containing $d_2$ for $UD_n\Theta$. Let $c$ be a critical 2-cell containing $d_1$ in $UD_n\Gamma'$. Then
$\tilde i(c)$ must contain $X_s(a\vec\delta_r+b\vec\delta_s)\cup q$ as a subset for
$1\le r <s\le \mu$ where $i(Y)=X$, $i(d_1)=q$, and $\mu+1$ is the
valency of $X$ in the maximal tree $T$ of $\Gamma$. Consider the
subgroup $M_q$ of $M_2(UD_n\Gamma)$ generated by all critical
2-cells containing the deleted edge $q$ of $\Gamma$ that is the image of the deleted edge $d_1$ of $\Gamma'$
under the embedding $i$. In particular, $\widetilde
i(M_2(UD_n\Gamma'))\subset M_q$ if $\Gamma'=S_0$. For the projection
$\pi:M_2(UD_n\Gamma)\to M_q$, the composites
$\pi\circ\widetilde\partial$ and $\pi\circ\widetilde R$ will be
denoted by $\widetilde\partial_q$ and $\widetilde R_q$. From now on,
a vertex denoted by the letter $X$ and a deleted edge denoted by the letter $q$ will be
reserved as above when we consider a graph $\Gamma$ containing $S_0$
or $\Theta$.

Let $c$ be a critical 3-cell in $UD_n\Gamma$ containing $q$. When we
compute $\widetilde\partial_q(c)$, two faces of $c$ determined by
taking end vertices of $q$ do not contain $q$ and so we only need to
consider the remaining four faces of $\partial(c)$. Consequently,
any vertex blocked by $q$ plays no role in computing
$\widetilde\partial_q(c)=\widetilde R_q\partial(c)$. Thus we
immediately obtain the following modified version of
Lemma~\ref{lem:naturalmove}.

\begin{lem}\label{lem:naturalmove-q}
Suppose a redundant cell $c$ has an unblocked vertex $v$ and $e$ is
the edge starting from $v$. Assume also that if a vertex $w$ that is
either in $c$ or an end vertex of an edge other than $q$ in $c$
satisfies $\tau(e)< w<\iota(e)$, then $w$ is blocked by $q$. Then
$\widetilde R_q(c)=\widetilde R_qV_e(c)$.
\end{lem}

Let $v$ be a vertex in $\Gamma$ with the maximal tree $T$ and $e$ be
a edge in $\Gamma$.  The edge $e$ is said to be {\em separated by}
$v$ if $\iota(e)$ and $\tau(e)$ lie in two distinct components of
$T-\{v\}$. It is clear that only a deleted edge can be separated by
a vertex. If a deleted edges $d$ is not separated by $v$, then
$\iota(d)$, $\tau(d)$, and $\iota(d)\wedge\tau(d)$ are all in the
same component of $T-\{v\}$. We only need to compute images of
redundant 2-cells under $\widetilde R_q$ and we can strengthen
Lemma~\ref{lem:naturalmove-q} for redundant 2-cells.

\begin{lem}\label{cor:naturalmove-q2}
Let $c$ be a redundant 2-cell containing the deleted edge $q$ and
another edge $p$. Suppose the redundant 2-cell $c$ has an unblocked vertex
$v$ and $e$ is the edge starting from $v$ satisfying the following conditions:
\begin{itemize}
\item[(a)] Every vertex $w$ in $c$ satisfying $\tau(e)< w<\iota(e)$ is blocked.
\item[(b)] If an end vertex $w$ of $p$ satisfies $\tau(e)< w<\iota(e)$
then $p$ is not separated by any of $\tau(e)$ and end vertices of $q$.
\end{itemize}
Then $\widetilde R_q(c)=\widetilde R_qV_e(c)$. Therefore if $p$ is not a deleted edge then $\widetilde R_q(c)=\widetilde R_q\widetilde V(c)$.
\end{lem}

\begin{proof}
If both end vertices of $p$ are not between $\tau(e)$ and $\iota(e)$ then $c$ satisfies the hypothesis of
Lemma~\ref{lem:naturalmove-q}. So we are done. Otherwise $p$ is not separated by any of $\tau(e)$ and two end vertices of $q$. So $\iota(p)$ and $\tau(p)$ are in the same component $T_p$ of $T-\{\tau(e),\iota(q),\tau(q)\}$. Let $w$ be any vertex in $T_p$. Let $c_w$ denote a cell obtained from $c$ replacing $v$ and $p$ by $e$ and $w$. Let $c'$ be a redundant 2-cell obtained from $c_{\iota(p)}$ replacing vertices in $T_p$ to other vertices that are all in $T_p$. Note that $\{c'\}$ is a finite set since the number of vertices in $T_p$ is finite. If $c'$ has no unblocked vertex in $T_p$ then $c'$ is unique because $\Gamma$ is sufficiently subdivided. This 2-cell is denoted by $c_p$. If $c'$ has an unblocked vertex in $T_p$ then $c'$ has the smallest unblocked vertex $u$ in $T_p$ that satisfies the hypothesis of Lemma~\ref{lem:naturalmove-q} since $c'$ contains exactly two edges $q$ and $e$ and every vertex in $T_p$ is between $\tau(e)$ and $\iota(e)$. So $\widetilde R_q(c')=\widetilde R_qV_{e'}(c')$ where $e'$ is the edge starting from $u$. By repeating this argument, we have $\widetilde R_q (c_{\iota(p)})=\widetilde R_q (c_p)=\widetilde R_q (c_{\tau(p)})$ because $V_{e'}(c')$ is also in $\{c'\}$ that is a finite set. Thus $$\widetilde R_q(c)=\widetilde R_qV_e(c)\pm \{\widetilde R_q(c_{\iota(p)})-\widetilde R_q(c_{\tau(p)})\}=\widetilde R_qV_e(c)$$ where the sign $\pm$ is determined by the order between the initial vertices of three edges $p$, $q$ and $e$. Therefore if $p$ is not a deleted edge then the smallest unblocked vertex in $c$ satisfies the hypothesis of this lemma. So $\widetilde R_q(c)=\widetilde R_qV(c)$. By repeating this argument, we have $\widetilde R_q(c)=\widetilde R_q\widetilde V(c)$.
\end{proof}

To demonstrate how to compute an image of $\widetilde\partial_q$, we
first give an example.

\begin{figure}[ht]
\psfrag{*}{\small0}
\psfrag{A}{\small$X$}
\psfrag{B}{\small$B$}
\psfrag{C}{\small$A$}
\psfrag{e0}{\small$q$}
\centering
\includegraphics[height=2.3cm]{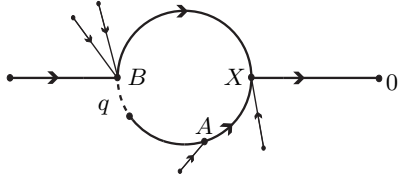}
\caption{Graph containing $S_0$ but not $\Theta$}
\label{fig17}
\end{figure}

\begin{exa}\label{ex:withS_0}
For a graph $\Gamma$ and its maximal tree and the order given in
Figure~\ref{fig17}, we compute $\widetilde\partial_q(c)$ for the
critical 3-cell $c=X_3(1,2,2) \cup q\cup A_2(1,1)\cup B(2,0,0)$ in
$M_3(UD_n\Gamma)$.
\end{exa}
Since $\iota(A_2(0,1))<\iota(X_3(0,0,1))<\iota(q)$, we have
\begin{align*}
\partial_q(c)=&+(X(1,2,2)\cup q\cup A_2(1,1)\cup B(2,0,0))\\
&-(\dot X(1,2,1)\cup q \cup A_2(1,1)\cup B(2,0,0))\\
&-(X_3(1,2,2)\cup q\cup A(1,1)\cup B(2,0,0))\\
&+(X_3(1,2,2)\cup q\cup \dot A(1,0)\cup B(2,0,0))
\end{align*}

Consider the pair of faces determined by the edge $A_2(0,1)$. The edge $X_3(0,0,1)$ is not a deleted edge and their images under $\widetilde V$ are equal. So their images $\widetilde R_q$ are also equal by Lemma~\ref{cor:naturalmove-q2}. Thus the
images cancel out in $\widetilde\partial_q(c)$ because they have the opposite signs. Similarly, the images of the pair of
faces determined by the edge $X_3(0,0,1)$ under $\widetilde R_q$ also cancel out. Thus $$\widetilde\partial_q(X_3(1,2,2) \cup q\cup A_2(1,1)\cup
B(2,0,0))=0.$$

\begin{lem}\label{lem:onedel}
Let $\Gamma$ be a finite planar graph containing $S_0$ or $\Theta$.
Suppose that $c$ is a critical 3-cell in $UD_n\Gamma$ as the form $A_k(\vec
a)\cup B_\ell(\vec b) \cup q \cup C(\vec c)$ where $C=\iota(q)$. Then
$\widetilde\partial_q(c)=0$.
\end{lem}
\begin{proof}
We consider two faces of $\partial(c)$ determined by $A_k(\vec\delta_k)$. They contain $B_\ell(\vec\delta_\ell)$ that is not a deleted edge and have the same image under $\widetilde V$. By Lemma~\ref{cor:naturalmove-q2}, their images under $\widetilde R_q$ are the same. Since they have the opposite signs, they cancel each other in $\widetilde\partial_q(c)$. By the same reason images of two faces determined by $B_\ell(\vec\delta_\ell)$ under $\widetilde R_q$ cancels each other in $\widetilde\partial_q(c)$. Consequently, $\widetilde\partial_q(c)=0$.
\end{proof}

If a critical 3-cell $c$ contains other deleted edges in addition to $q$, we have to deal with too many cases to compute $\widetilde\partial_q(c)$. In order to narrow down the possibilities, we need the following lemma.

\begin{lem}\label{lem:specialembedding}
Let $\Gamma'$ be either $S_0$ or $\Theta$ embedded in a plane as Figure~\ref{fig16}. Let $\Gamma$ be a finite planar graph containing $\Gamma'$. Assume that if $\Gamma'$ is $S_0$, then $\Gamma$ does not contain $\Theta$. Then we can choose an embedding $i:\Gamma'\to\Gamma$ and modify the planar embedding of $\Gamma$ so that all edges of $\Gamma$ that are incident to the vertex $X=i(Y)$ and are not in the image of $i$, lie between $i(e_0)$ and $i(e_1)$
where $e_0$, $e_1$ are edges incident to $Y$ in $\Gamma'$ such that $e_0$  heads for the base vertex $0$ and $e_1$ is next to $e_0$ clockwise.
\end{lem}

\begin{proof} Let $e_2$ denote the remaining edge incident to $Y$ in $\Gamma'$. Figure~\ref{fig18} or Figure~\ref{fig19} show only the part of $\Gamma$ consisting of $i(\Gamma')$ and edges incident to $X$.
Suppose that $\Gamma$ contains $S_0$ but not $\Theta$. Then we choose an embedding $i:S_0\to \Gamma$ such that $\Gamma$ has no edges incident to $X$ between $i(e_2)$ and $i(e_0)$. Since $\Gamma$ contains no $\Theta$, every edge incident to $X$ in $\Gamma$ between $i(e_1)$ and $i(e_2)$ should not be joined to $i(\Gamma')$ by an edge path (indicated by a dotted arc in Figure~\ref{fig18}) unless it is a loop at $X$. Thus the subgraph attached to these edges is free from $i(\Gamma')$ except at $X$ and it can be flipped and placed between $i(e_0)$ and $i(e_1)$ by modifying the planar embedding of $\Gamma$.
\begin{figure}[ht]
\psfrag{a}{\small{$i(e_0)$}}
\psfrag{b}{\small{$i(e_1)$}}
\psfrag{c}{\small{$i(e_2)$}}
\subfigure[Before modification]
{\includegraphics[height=1.5cm]{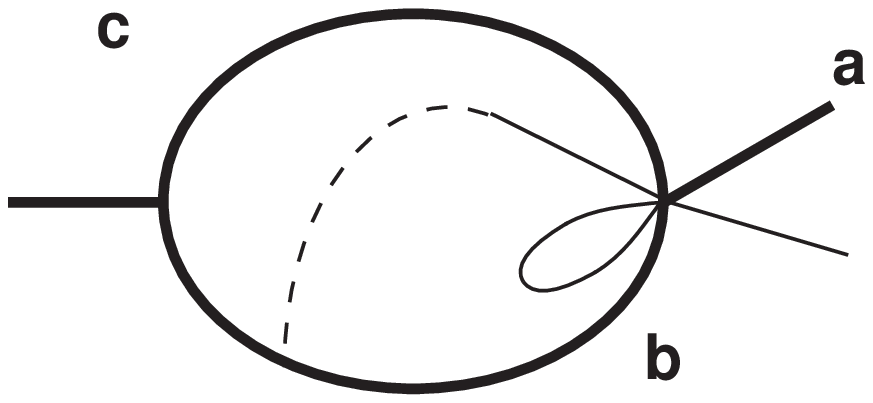}}\qquad\qquad
\subfigure[After modification]
{\includegraphics[height=1.5cm]{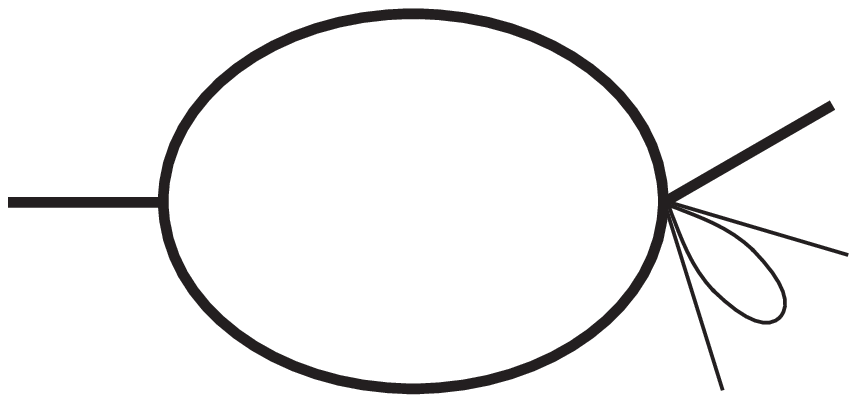}}
\caption{$i:S_0\to\Gamma$}
\label{fig18}
\end{figure}

Suppose that $\Gamma$ contains $\Theta$. Then we choose an embedding $i:\Theta\to \Gamma$ such that (i) $i(\Theta)$ is outmost in the sense that it is engulfed by no other larger embedding $i'$ of $\Theta$ such that $i'(Y)=X$, and (ii) no other edge path joining $X$ and another vertex on $i(e_1\cup e_2)$ lies inside the circle $i(e_1\cup e_2)$. By (i), every edge incident to $X$ in $\Gamma$ between $i(e_2)$ and $i(e_0)$ should not be joined to $i(\Gamma')$ by an edge path (indicated by a dotted arc in Figure~\ref{fig19}) unless it is a loop at $X$. By (ii), every edge incident to $X$ in $\Gamma$ between $i(e_1)$ and $i(e_2)$ should not be joined to $i(\Gamma')$ by an edge path unless it is a loop at $X$. Again the subgraphs attached to these edges are free from $i(\Gamma')$ except at $X$ and they can be placed between $i(e_0)$ and $i(e_1)$ by modifying the planar embedding of $\Gamma$.
\begin{figure}[ht]
\psfrag{a}{\small{$i(e_0)$}}
\psfrag{b}{\small{$i(e_1)$}}
\psfrag{c}{\small{$i(e_2)$}}
\subfigure[Before modification]
{\includegraphics[height=1.75cm, width=3.5cm]{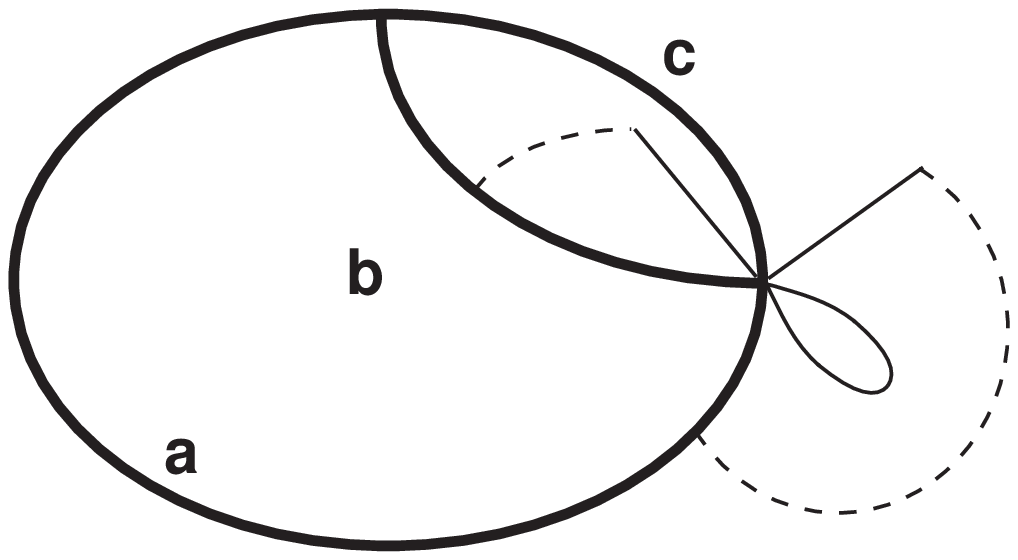}}\qquad\qquad
\subfigure[After modification]
{\includegraphics[height=1.75cm, width=3cm]{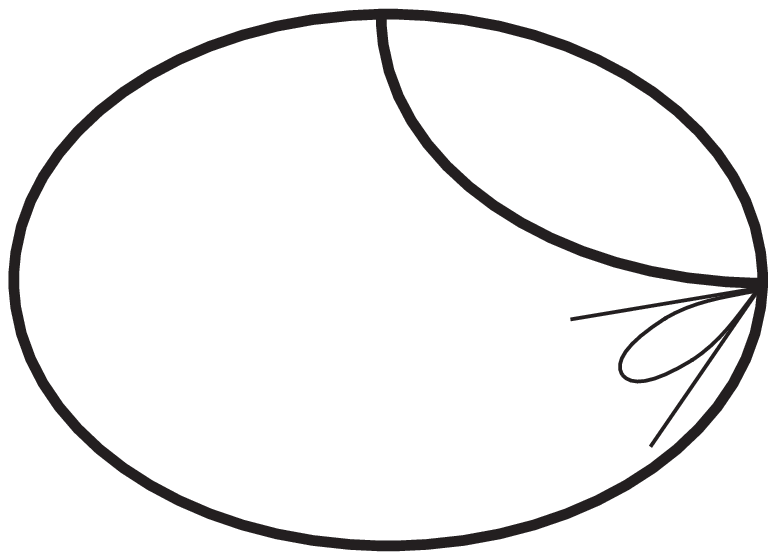}}
\caption{$i:\Theta\to\Gamma$}
\label{fig19}
\end{figure}

\end{proof}

After we choose an embedding $i:\Gamma'\to \Gamma$ and a planar embedding of $\Gamma$ given by Lemma~\ref{lem:specialembedding}, we
choose a maximal tree $T$ of $\Gamma$ so that it satisfies the following:
\begin{itemize}
\item[(1)] Delete the edge that is the image of $d_1$ or $d_2$ in $S_0$ or $\Theta$ under $i$;
\item[(2)] Delete an edge incident to $\iota(q)$ whenever we need to break a circuit containing $\iota(q)$;
\item[(3)] Delete an edge incident to $X$ whenever we need to break a circuit containing $X$ but not $\iota(q)$ and $\Gamma$ does not contain $\Theta$;
\item[(4)] All other deleted edges have valency 2 at both ends.
\end{itemize}
The conditions on a maximal tree $T$ of a planar graph $\Gamma$ together with an embedding $i:\Gamma'\to\Gamma$ given by Lemma~\ref{lem:specialembedding} will be referred as $\cond(\Gamma,i,T)$.
There are useful consequences of assuming $\cond(\Gamma,i,T)$.
The condition (1) guarantees that $i$ induces a chain map between Morse complexes if $i$ preserve the order of $\Gamma'$. The conditions (2) guarantee that there are no deleted edges separated by $\iota(q)$. The conditions (3) implies that if $X$ separates a deleted edge in a graph containing $S_0$ but not $\Theta$, then the deleted edge is $q$. There are no vertices blocked by deleted edges satisfying the condition (4).

After fixing the base vertex of $\Gamma$ by choosing a vertex of valency 1 in $T$ that can be joined to $i(0)$ via an edge path without $i(e_0)$, we give $\Gamma$ the order determined by the planar embedding of $T$. Then the embedding $i:\Gamma'\to\Gamma$ is order preserving, that is, $v_1<v_2$ implies $i(v_1)<i(v_2)$ for vertices $v_1,v_2$ of $\Gamma'$.
Since $\cond(\Gamma,i,T)$ assumes the conclusion of Lemma~\ref{lem:specialembedding}, it implies
$$\tilde i(Y_2(a,b)\cup d_1)=X_s(a\vec\delta_{\mu-1}+b\vec\delta_\mu)\cup q$$
where $\mu+1$ be the valency of $X$ in a maximal tree $T$ of $\Gamma$.
The amount of computation of $\partial_q$ will considerably be reduced by assuming $\cond(\Gamma,i,T)$ since it limits a great deal of possibilities.

In \S\ref{ss43:h2inj}, we will need the fact that $\tilde i(M_2(UD_n\Gamma'))\cap\tilde\partial(M_3(UD_n\Gamma))=\{0\}$.
Under the assumption $\cond(\Gamma,i,T)$,  $\tilde i(M_2(UD_n\Gamma'))$ is generated by critical 2-cells of the form $X_\mu(\vec x)\cup q$ where the vertex $X$ is of valency $\mu +1$ in $T$.
To demonstrate how to compute $\widetilde\partial_q(c)$, let us
see an example first.
\begin{exa}\label{ex:twodel1}
For a graph $\Gamma$ and its maximal tree and the order given in
Figure~\ref{fig20}, we compute $\widetilde\partial_q(c)$ for a
3-cell $X_2(2,1)\cup q\cup d$ in $M_3(UD_n\Gamma)$.
\end{exa}
\begin{figure}[ht]
\psfrag{*}{\small0}
\psfrag{A}{\small$X$}
\psfrag{e1}{\small$q$}
\psfrag{e}{\small$d$}
\centering
\includegraphics[height=2.3cm]{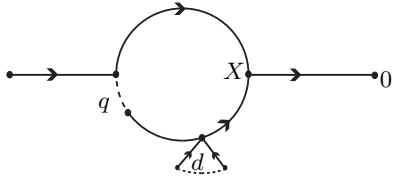}
\caption{Graph containing $X_2(2,1)\cup q \cup d$}
\label{fig20}
\end{figure}
Since $\iota(d)<\iota(X_2(0,1))<\iota(q)$, we have
\begin{align*}
\widetilde\partial_q(c)=&\widetilde R_q(q\cup d\cup X(2,1))-\widetilde R_q(q \cup d\cup \dot X(2,0))\\
&-\widetilde R_q(X_2(2,1)\cup q\cup \iota(d))+\widetilde R_q(X_2(2,1)\cup q\cup \tau(d)).
\end{align*}
After the above four faces are rewritten to eliminate all unblocked vertices satisfying the hypothesis of Lemma~\ref{cor:naturalmove-q2}, we have
\begin{align*}
&\widetilde R_q(q\cup d\cup X(2,1))=\widetilde R_q(q \cup d\cup \dot X(2,0))=q\cup d\\
&\widetilde R_q(X_2(2,1)\cup q\cup \iota(d))=\widetilde R_q(X_2(2,1)\cup q\cup \tau(d))=X_2(3,1)\cup q.
\end{align*}
Consequently, we have $\widetilde\partial_q(c)=0$.

\begin{lem}\label{lem:twodel1}
Let $\Gamma$ be a planar graph satisfing $\cond(\Gamma,i,T)$. Suppose that $c$ is a critical 3-cell in $UD_n\Gamma$ of the form $A_k(\vec a) \cup q \cup d$.
If $c$ satisfies one of the following:
\begin{itemize}
\item[(i)] $A\ne X$;
\item[(ii)] $A=X$ and $k < \mu$;
\item[(iii)] $A=X$, $k=\mu$ and $d$ is not separated by $X$.
\end{itemize}
Then $\widetilde\partial(c)$ contains no critical 2-cells of the form $X_\mu(\vec x)\cup q$ as a summand.
\end{lem}
\begin{proof}
First consider two faces of $\partial_q(c)$ determined by $d$, that is, $A_k(\vec a)\cup q \cup \{v\}$ where $v$ is either $\iota(d)$ or $\tau(d)$. Since $A_k(\vec\delta_k)$ is not a deleted edge, $\widetilde R_q(A_k(\vec a)\cup q \cup \{v\})=\widetilde R_q\widetilde V(A_k(\vec a)\cup q \cup \{v\})$ by Lemma~\ref{cor:naturalmove-q2}. If $A\ne X$ or $k\ne \mu$, then $\widetilde R_q\widetilde V(A_k(\vec a)\cup q \cup \{v\})$ cannot be a critical 2-cell of the form $X_\mu(\vec x)\cup q$. If $A=X$ and $k=\mu$ then their images under $\widetilde R_q\widetilde V$ are equal since $d$ is not separated by $X$. So the two faces cancel each other in $\widetilde\partial_q(c)$.

Now consider two faces determined by $A_k(\vec\delta_k)$ under $\widetilde R_q$. %This computation is similar to the proof of Lemma~\ref{lem:caseb}. So we discuss the reason why there are no critical 2-cells of the form $X_\mu(\vec x)\cup q$ in $\widetilde\partial_q(c)$.
If $A\ne X$ then $A$ and $A(\vec\delta_k)$ are in the same branch of $X$. So if there is a critical 2-cell of the form $X_\mu(\vec x)\cup q$ in the image of a face under $\widetilde R_q$ then the critical 2-cell is also in the image of the other face under $\widetilde R_q$. Thus all critical 2-cells of the form $X_\mu(\vec x)\cup q$ cancel out in $\widetilde\partial_q(c)$ since each pair of faces has the opposite sign in $\partial_q(c)$. If $A=X$ and $k < \mu$ then all critical 2-cells of the form $X_\mu(\vec x)\cup q$ also cancel out in $\widetilde\partial_q(c)$ because the numbers of vertices in each pair of faces on $\mu$-th branch of $X$ are equal. %as in the proof of Lemma~\ref{lem:caseb}.
Finally if $A=X$ and $k=\mu$ then there are no critical 2-cells of the form $X_\mu(\vec x)\cup q$ in the image of faces under $\widetilde R_q$ since $d$ is not separated by $X$.
\end{proof}

If $\Gamma$ does not contain $\Theta$, then there are blocked vertices by a deleted edge $d_X$ incident to $X$. So we must consider a critical 3-cell in $UD_n\Gamma$ of the form $A_k(\vec a) \cup q \cup d_X \cup X(\vec x)$. Any vertex blocked by $X$ plays no role in the proof of the above lemma. Thus we have the following corollary.

\begin{cor}\label{cor:twodel1-1}
Let $\Gamma$ be a graph that does not contain $\Theta$ and satisfies $\cond(\Gamma,i,T)$. Suppose that $c$ is a critical 3-cell in $UD_n\Gamma$ of the form $A_k(\vec a) \cup q \cup d_X\cup X(\vec x)$. Then $\widetilde\partial(c)$ contains no critical 2-cells of the form $X_\mu(\vec x)\cup q$ as a summand.
\end{cor}

Now consider a critical 3-cell containing $X_\mu(\vec x)$ and a deleted edge $d$ such that $X$ separates $d$.
\begin{exa}\label{ex:twodel2}
Let $\Gamma$ be a graph as the Figure~\ref{fig21}. We
compute $\widetilde\partial_q(c)$ for the 3-cell $c=X_2(1,2)\cup q
\cup d$ in $M_3(UD_n\Gamma)$. Note that $A=X\wedge\tau(d)$.
\end{exa}
\begin{figure}[ht]
\psfrag{*}{\small0}
\psfrag{A}{\small$X$}
\psfrag{B}{\small$A$}
\psfrag{e}{\small$d$}
\psfrag{e1}{\small$q$}
\centering
\includegraphics[height=2.3cm]{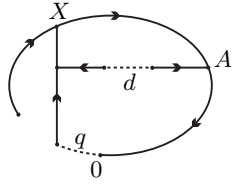}
\caption{Graph containing $X_2(1,2)\cup q \cup d$}
\label{fig21}
\end{figure}
Since $\iota(d)<\iota(q)<\iota(X_2(0,1))$, using Lemma~\ref{lem:naturalmove-q} we have
\begin{align*}
\widetilde\partial_q(c)=&-\widetilde R_q(q\cup d\cup X(1,2))+\widetilde R_q(q \cup d\cup \dot X(1,1))\\
&-\widetilde R_q(X_2(1,2)\cup q\cup \iota(d))+\widetilde R_q(X_2(1,2)\cup q\cup \tau(d))\\
=&-\widetilde R_q(q\cup d\cup X(0,2)\cup A(0,1))+\widetilde R_q(q \cup d\cup X(0,1)\cup A(0,2))\\
&-(X_2(2,2)\cup q)+(X_2(1,2)\cup q).
\end{align*}
Since $q\cup d\cup X(0,2) \cup A(0,1)$ has the negative sign in the boundary of the collapsible 3-cell $A_2(0,1)\cup q\cup d\cup X(0,2)$,
\begin{align*}
\widetilde R_q(&q\cup d\cup X(0,2) \cup A(0,1))\\
=&\widetilde R_q(\partial (A_2(0,1)\cup q\cup d\cup X(0,2))+ q\cup d\cup X(0,2) \cup A(0,1))\\
=&\widetilde R_q(q\cup d\cup \{A\} \cup X(0,2))+\widetilde R_q(A_2(0,1) \cup q\cup \iota(d) \cup X(0,2))\\
 &-\widetilde R_q(A_2(0,1) \cup q\cup \tau(d) \cup X(0,2))\\
=&\widetilde R_q(q\cup d\cup X(0,2))-(A_2(1,3) \cup q).
\end{align*}
By repeating a similar argument,
\begin{align*}
\widetilde R_q(q\cup d\cup X(0,2))=&\widetilde R_q(\partial (X_2(0,2)\cup q \cup d)+ q\cup d\cup X(0,2))\\
=&\widetilde R_q(q\cup d \cup \dot X(0,1)) -\widetilde R_q(X_2(0,2)\cup q \cup \iota(d)) + \widetilde R_q(X_2(0,2)\cup q \cup \tau(d))\\
=&\widetilde R_q(q\cup d \cup X(0,1))-(A_2(1,2)\cup q)-(X_2(1,2)\cup q).
\end{align*}
and $\widetilde R_q(q\cup d\cup \dot X(1,1))=\widetilde R_q(q\cup d\cup X(0,1))-(A_2(1,3) \cup q)-(A_2(1,2) \cup q).$
Consequently, we have $\widetilde\partial_q(c)=2(X_2(1,2)\cup q)-(X_2(2,2)\cup q)$.

A graph considered in the following lemma necessarily contains $\Theta$ since it satisfies $\cond(\Gamma,i,T)$ and has a deleted edge other than $q$ separated by $X$.
\begin{lem}\label{lem:twodel2}
Let $\Gamma$ be a planar graph satisfying $\cond(\Gamma,i,T)$. Suppose that $c$ is a critical 3-cell in $UD_n\Gamma$ of the form $X_\mu(\vec x) \cup q \cup d$ and  $d$ is separated by the vertex $X$ but not by $X(\vec\delta_\mu)$. Then
\begin{align*}
\widetilde\partial_q(c)=&(X_\mu(\vec x+\vec\delta_{g(X,\tau(d))})\cup q)-(X_\mu(\vec x+\vec\delta_{g(X,\iota(d))}) \cup q)\\
&-(X_\mu(\vec\delta_{g(X,\tau(d))}+x_\mu\vec\delta_\mu) \cup q)+(X_\mu(\vec\delta_{g(X,\iota(d))}+x_\mu\vec\delta_\mu) \cup q)+ \mbox{Irrelevant}
\end{align*}
where the part ``Irrelevant" contains no critical 2-cells of the form  $X_\mu(\vec y)\cup q$, and $x_\mu$ is the $\mu$-th component of $\vec x$, and the expression may contain collapsible 2-cells that should be regarded as trivial
\end{lem}

\begin{proof}
Since $d$ is separated by $X$ but not by $\iota(q)$ due to $\cond(\Gamma,i,T)$, $X$ and both end of $d$ are in the same component of $T-\{\iota(q)\}$. Since $d$ is not separated by $X(\vec\delta_\mu)$, $\iota(d)<X(\vec\delta_\mu)$. And $\iota(d)<\iota(q)$ by $\cond(\Gamma,i,T)$. Thus we only need to consider the two possibilities
$\iota(d)<X(\vec\delta_\mu)<\iota(q)$ and $\iota(d)<\iota(q)<X(\vec\delta_\mu)$.
Assume $\iota(d)<X(\vec\delta_\mu)<\iota(q)$. Then we have
\begin{align*}
\widetilde\partial_q(c)=&\widetilde R_q(X_\mu(\vec x)\cup q \cup \tau(d)) -\widetilde R_q(X_\mu(\vec x)\cup q \cup \iota(d))\\
&-\widetilde R_q(q\cup d\cup \dot X(\vec x-\vec\delta_\mu)) +\widetilde R_q(q\cup d\cup X(\vec x)).
\end{align*}
For the pair of faces of $c$ determined by $d$, we can apply Lemma~\ref{cor:naturalmove-q2} since $X_\mu(\vec\delta_\mu)$ is not a deleted edge. Then we have
\begin{align*}
\widetilde R_q(X_\mu(\vec x)\cup q \cup \tau(d))& -\widetilde R_q(X_\mu(\vec x)\cup q \cup \iota(d))\\
=&X_\mu(\vec x+\vec\delta_{g(X,\tau(d))})\cup q - X_\mu(\vec x+\vec\delta_{g(X,\iota(d))})\cup q
\end{align*}

To consider the pair of faces determined by $X_\mu(\vec\delta_\mu)$, let $A=X\wedge\tau(d)$. Since $d$ is separated by $X$, $X\wedge\iota(d)=X$. Since the smallest unblocked vertex of $q\cup d\cup \dot X(\vec x-\vec\delta_\mu)$ is $X$, it does not give a critical 2-cell of the form $X_\mu(\vec{x'})\cup q$ in its image under $\widetilde R_q$ as follows:
$$\widetilde R_q(q\cup d\cup \dot X(\vec x-\vec\delta_\mu))=\widetilde R_q(q\cup d\cup X(\vec x-\vec\delta_\mu))-A_{g(A,X)}(|\vec x|\vec\delta_{g(A,X)}+\vec\delta_{g(A,\tau(d))})\cup q$$
where the second term disappears if $A=X$. On the other hand, if the smallest unblocked vertex $v$ of $q\cup d\cup X(\vec x)$ satisfies $g(X,v)\ne \mu$, it does not produce a critical 2-cell of the form $X_\mu(\vec{x'})\cup q$ since $v<X(\vec\delta_\mu)$, that is, we have
$$\widetilde R_q(q\cup d\cup X(\vec x))=\widetilde R_q(q\cup d\cup X(\vec x-1))+\mbox{Irrelevant terms}.$$
By iterating this step, we have
$$\widetilde R_q(q\cup d\cup X(\vec x))=\widetilde R_q(q\cup d\cup X(x_\mu\vec\delta_\mu))+\mbox{Irrelevant terms}.$$
From $\iota(d)<X(\vec\delta_\mu)<\iota(q)$, we have
\begin{align*}
\widetilde R_q(q\cup d\cup X(x_\mu\vec\delta_\mu))=&\widetilde R_q(-\partial(q\cup d\cup X_\mu(x_\mu\vec\delta_\mu))+q\cup d\cup X(x_\mu\vec\delta_\mu))\\
=&\widetilde R_q(q\cup d\cup \dot X((x_\mu-1)\vec\delta_\mu))+\widetilde R_q(X_\mu(x_\mu\vec\delta_\mu)\cup q\cup \iota(d))\\
&-\widetilde R_q(X_\mu(x_\mu\vec\delta_\mu)\cup q \cup \tau(d))\\
=&\widetilde R_q(q\cup d\cup X((x_\mu-1)\vec\delta_\mu)-A_{g(A,X)}(x_\mu\vec\delta_{g(A,X)}+\vec\delta_{g(A,\tau(d))})\cup q\\
&+X_\mu(x_\mu\vec\delta_\mu+\vec\delta_{g(X,\iota(d))})\cup q-X_\mu(x_\mu\vec\delta_\mu+\vec\delta_{g(X,\tau(d))})\cup q.
\end{align*}
So we obtain the desired formula for the case $\iota(d)<X(\vec\delta_\mu)<\iota(q)$. Similarly, we can proceed in the case $\iota(d)<\iota(q)<X(\vec\delta_\mu)$.
\end{proof}

\begin{exa}\label{ex:threedel}
Let $\Gamma$ be a graph as the Figure~\ref{fig22}, and we choose a
maximal tree and give an order as the Figure~\ref{fig22}. We
compute $\widetilde\partial_q(c)$ for the 3-cell $q \cup d \cup
d'$ in $M_3(UD_n\Gamma)$.
\end{exa}
\begin{figure}[ht]
\psfrag{*}{\small0}
\psfrag{A}{\small$X$}
\psfrag{B}{\small$d'$}
\psfrag{e}{\small$d$}
\psfrag{e1}{\small$q$}
\centering
\includegraphics[height=2.3cm]{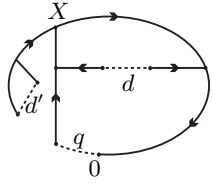}
\caption{Graph containing $q \cup d \cup d'$}
\label{fig22}
\end{figure}
Since $\iota(d)<\iota(q)<\iota(d')$, using
Lemma~\ref{cor:naturalmove-q2} we have
\begin{align*}
\widetilde\partial_q(c)=&-\widetilde R_q(q\cup d'\cup \iota(d))+\widetilde R_q(q \cup d'\cup \tau(d))\\
&-\widetilde R_q(q\cup d \cup\iota(d'))+\widetilde R_q(q\cup d\cup\tau(d'))\\
=&-(q\cup d')+(q\cup d')-\widetilde R_q(q\cup d \cup\{X(\vec\delta_\mu)\})+\widetilde R_q(q\cup d \cup\{X(\vec\delta_\mu)\})=0
\end{align*}

\begin{lem}\label{lem:threedel}
Let $\Gamma$ be a planar graph satisfying $\cond(\Gamma,i,T)$. Suppose that $c$ is a critical 3-cell in $UD_n\Gamma$ of the form $q \cup d\cup d'$. If both $d$ and $d'$ are not separated by $X(\vec\delta_\mu)$, then $\widetilde\partial(c)$ contains no critical 2-cells of the form $X_\mu(\vec x)\cup q$ as a summand.
\end{lem}
\begin{proof}
Let $T'_c$ be the smallest subtree of $T$ that contains the base vertex and four ends of $d$ and $d'$. Note that $\widetilde\partial_q(c)$ is determined by $T'_c$. If both $d$ and $d'$ have no end vertex $v$ such that $g(X,v)=\mu$, then $T'_c$ does not contain the vertex $X(\vec\delta_\mu)$ and so the edge $X_\mu(\vec\delta_\mu)$ does not appear in the computation of $\widetilde\partial_q(c)$ since end vertices of $d$ and $d'$ are smaller than $X(\vec\delta_\mu)$. Thus we are done.

If both $d$ and $d'$  have end vertices $v$ of such
that $g(X,v)=\mu$, then the valency of $X$ in $T'_c$ is two since $d$ and $d'$ are not separated by $X(\vec\delta_\mu)$. So $\widetilde\partial_q(c)$ does not contain $X_\mu(\vec x)\cup q$ as summands. Otherwise we may assume that two end of $d$ are on the last branch of $X$. Then $\iota(d)>\tau(d)>\iota(d')>\tau(d')$. So $\widetilde R(q\cup d\cup \iota(d'))=\widetilde R(q\cup d \cup \tau(d'))$. Since $\iota(d)\wedge \iota(d')=\tau(d)\wedge \iota(d')$ and $\iota(d)\wedge \tau(d')=\tau(d)\wedge \tau(d')$, $\widetilde R(q\cup d'\cup \iota(d))=\widetilde R(q\cup d' \cup \tau(d))$. These imply $\widetilde\partial_q(c)=0$ by Lemma~\ref{cor:naturalmove-q2}.
\end{proof}

If there are deleted edges $d_X$ incident to $X$ in a graph $\Gamma$ containing $S_0$ but not $\Theta$, we need to consider critical 3-cells in $UD_n\Gamma$ of the form $q \cup d_X \cup d \cup X(\vec x)$. From (3) of $\cond(\Gamma,i,T)$, there are no deleted edges other than $q$ separated by $X$. Thus we have the following corollary of Lemma~\ref{cor:naturalmove-q2}.

\begin{cor}\label{cor:threedel2}
Let $\Gamma$ be a graph containing no $\Theta$ and satisfying $\cond(\Gamma,i,T)$. Suppose that $c$ is a critical 3-cell of the form $q \cup d_X \cup d \cup X(\vec x)$ in $UD_n\Gamma$. Then $\widetilde\partial(c)$ contains no critical 2-cells of the form $X_\mu(\vec x)\cup q$ as a summand.
\end{cor}

Note that only limited kinds of critical 3-cells were considered in this section but they are sufficient in achieving our goal in the following sections where graphs should be assumed to satisfy $\cond(\Gamma,i,T)$.

\subsection{Homomorphisms between the second homologies}\label{ss43:h2inj}
Let $\Gamma'$ be either $S_0$ or $\Theta$ and $\Gamma$ be a planar
graph containing $\Gamma'$. As explained in \S\ref{ss42:boundary
homomorphisms}, we choose an embedding $i:\Gamma' \to \Gamma$ and their maximal trees and orders satisfying $\cond(\Gamma,i,T)$ so that $i$ induces a
chain map $\tilde i$ between Morse complexes as follows:

\begin{displaymath}
\xymatrix{0 \ar[r] &M_2(UD_n\Gamma') \ar[r]^-{\widetilde\partial_2}
\ar@{^{(}->}[d]^{\tilde i} & M_1(UD_n\Gamma') \ar[r]^-{0}
\ar@{^{(}->}[d]^{\tilde i} &  \mathbb Z
\ar[r]^-{0} \ar[d]^{\simeq}& 0\\
M_3(UD_n\Gamma) \ar[r]^-{\widetilde\partial_3} &
M_2(UD_n\Gamma)\ar[r]^-{\widetilde\partial_2} & M_1(UD_n\Gamma)
\ar[r]^-{0} & \mathbb Z \ar[r]^-{0} & 0\\}
\end{displaymath}

Recall the subgroup $M_q$ of $M_2(UD_n\Gamma)$ and the homomorphism
$\widetilde\partial_q : M_3(UD_n\Gamma)\to M_q$ defined in the
previous section. When $\Gamma'=S_0$, we want to prove that $\tilde
i$ induces an injection $i_*:H_2(UD_n\Gamma')\to H_2(UD_n\Gamma)$
by showing a sufficient condition that
$\tilde i(M_2(UD_n\Gamma')) \cap \widetilde\partial(M_3(UD_n\Gamma))=\{0\}$.

When $\Gamma'=\Theta$, we want to prove that
$i_*:H_2(UD_n\Gamma')\to H_2(UD_n\Gamma)$ is a non-trivial
homomorphism by showing that there is a cycle $z$ in
$M_2(UD_n\Gamma')$ and a homomorphism $\rho: M_2(UD_n\Gamma)\to
\mathbb Z$ such that $\rho
\circ\widetilde\partial(M_3(UD_n\Gamma))=0$ but $\rho\circ\tilde
i(z)\ne 0$.

\begin{lem}\label{lem:S0toGamma}
Let $\Gamma$ be a graph that contains $S_0$ but does not contain
$\Theta$ and $n\ge 5$. Then there is an embedding $i:S_0 \to \Gamma$
that induces an injection $i_*:H_2(B_n S_0) \to H_2(B_n \Gamma)$.
\end{lem}

\begin{proof}
\begin{figure}[ht]
\psfrag{*}{\small0}
\psfrag{A}{\small$X$}
\psfrag{e0}{\small$q$}
\psfrag{e1}{\small$d_X$}
\centering
\includegraphics[height=2.8cm]{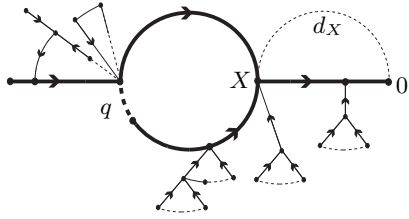}
\caption{Maximal tree and order of graph containing $S_0$ but not $\Theta$}
\label{fig23}
\end{figure}

Consider an embedding $i:S_0\to \Gamma$ and a maximal tree $T$ of $\Gamma$ satisfying $\cond(\Gamma,i,T)$ as the Figure~\ref{fig23}. We need to show that for each critical 3-cell $c$, $\widetilde\partial(c)\not\in\im(\tilde i)-\{0\}$. We know that $\im(\tilde i)$ is generated by critical 2-cells of the form $X_\mu(\vec x)\cup q$. A critical 3-cell $c$ must contain two deleted edges including $q$. Otherwise, $\tilde\partial_q(c)=0$ by Lemma~\ref{lem:onedel}, Lemma~\ref{lem:threedel} and Corollary~\ref{cor:threedel2}. So $c$ can be either $A_k(\vec a)\cup q \cup d$ or $A_k(\vec a)\cup q \cup d_X \cup X(\vec x)$ for some vertex $A$ of valency $\ge3$. Since $\cond(\Gamma,i,T)$ implies that $d$ is not separated by $X$, $\widetilde\partial(c)\not\in\im(\tilde i)-\{0\}$ by Lemma~\ref{lem:twodel1} and Corollary~\ref{cor:twodel1-1}.
\end{proof}

\begin{lem}\label{lem:ThetatoGamma}
Let $\Gamma$ be a planar graph that contains $\Theta$ and $n\ge 4$.
There exists an embedding $i:\Theta\to \Gamma$ such that the induced
map $i_*:H_2(B_n \Theta)\to H_2(B_n \Gamma)$ is non-trivial.
\end{lem}

\begin{proof}
\begin{figure}[ht]
\psfrag{*}{\small0}
\psfrag{A}{\small$X=i(Y)$}
\psfrag{e1}{\small$q=i(d_1)$}
\psfrag{e2}{\small$q_2=i(d_2)$}
\centering
\includegraphics[height=4cm]{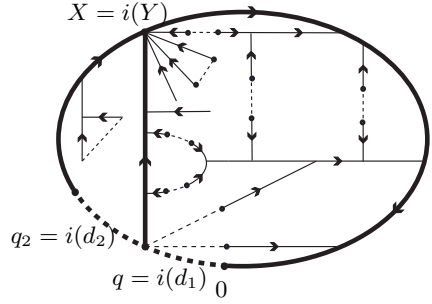}
\caption{Planar graph containing $\Theta$} \label{fig24}
\end{figure}

Choose an embedding $i:\Theta\to \Gamma$ and a maximal tree $T$ of
$\Gamma$ such that they satisfies $\cond(\Gamma,i,T)$. Then $g(X,\iota(q_2))=\mu$ and $g(X,\iota(q))=\mu-1$ where $q_2=i(d_2)$ as in Figure~\ref{fig24} and so $X(\vec\delta_\mu)$ separates no deleted edges. Thus the hypothesis of Lemma~\ref{lem:twodel2} and Lemma~\ref{lem:threedel} are automatically satisfied.

To show that $\tilde i_*$ is nontrivial, let
$\alpha=(Y_2(1,2)\cup d_1)-(Y_2(2,1)\cup d_1)+(Y_2(1,1)\cup d_1)
-2(Y_2(1,1)\cup d_2)+(Y_2(2,1)\cup d_2)$
where $Y$ is the vertex and $d_1$ and $d_2$ are deleted edges in the graph $\Theta$ with the planar embedding given by Figure~\ref{fig16}. Then
$\tilde\partial(\alpha)=(2Y_2(1,2)-Y_2(2,2))-(Y_2(2,1)+Y_2(1,1)-Y_2(3,1))+(2Y_2(1,1)-Y_2(2,1))
-2(Y_2(1,2)-Y_2(2,1)+Y_2(1,1))+(Y_2(2,2)-Y_2(3,1)+Y_2(1,1))=0$.
Thus $\alpha$ is a 2-cycle and so $\tilde i_*(\alpha)=(X_\mu(\vec\delta_{\mu-1}+2\vec\delta_\mu)\cup q)-(X_\mu(2\vec\delta_{\mu-1}+\vec\delta_\mu)\cup q)+(X_\mu(\vec\delta_{\mu-1}+\vec\delta_\mu)\cup q)
-2(X_\mu(\vec\delta_{\mu-1}+\vec\delta_\mu)\cup q_2)+(X_\mu(2\vec\delta_{\mu-1}+\vec\delta_\mu)\cup q_2)$ represents a homology class of $H_2(B_n \Gamma)$.

Define a homomorphism $\rho:
M_2(UD_n\Gamma) \to \mathbb Z$ by $\rho(X_\mu(\vec x) \cup
q)= x_\mu(x_1+\cdots +x_{\mu-1})$, and $\rho(c)=0$ for
any other critical 2-cell $c$. To show $\rho$ induces a homomorphism on homologies, we need the fact that $\rho(\im \widetilde\partial)$ is trivial.
By the definition of $\rho$. we need to show $\rho(c)=0$ for each critical 3-cell $c$ in $M_3(UD_n \Gamma)$ such that $\widetilde\partial(c)$ contains a 2-cell of the form $X_\mu(\vec x)\cup q$ as a summand. By Lemma~\ref{lem:onedel}, Lemma~\ref{lem:twodel1}, and Lemma~\ref{lem:threedel}, it suffices to consider $c=X_\mu(\vec x)\cup q \cup d$ for some $\vec x$ and a deleted edge $d$. Then by Lemma~\ref{lem:twodel2},
\begin{align*}
\widetilde\partial(c)=&(X_\mu(\vec x+\vec\delta_{g(X,\tau(d))})
\cup q)-(X_\mu(\vec x+\vec\delta_{g(X,\iota(d))}) \cup q)\\
&+(X_\mu(\vec\delta_{g(X,\iota(d))}+x_\mu\vec\delta_\mu) \cup q)
-(X_\mu(\vec\delta_{g(X,\tau(d))}+x_\mu\vec\delta_\mu) \cup q)+\mbox{Irrelevant}.
\end{align*}
Then $\rho(\widetilde\partial(c))=0$ and so $\rho$ induces a
homomorphism $\rho_*:H_2(M_*(UD_n\Gamma)) \to \mathbb Z$. Since
$\rho (\tilde i_*(\alpha))=1$, $\rho_*(\tilde
i_*(\alpha))$ is a nontrivial class in
$H_2(M_*(UD_n\Gamma))$. So the homomorphism $i_*:H_2(B_n \Theta)\to
H_2(B_n \Gamma)$ is non-trivial.
\end{proof}

\subsection{Non-planar graphs}\label{ss44:nonplanar}
If we choose a maximal tree of $K_{3,3}$ and give an order as the
Figure~\ref{fig25}. Then we obtain a presentation of $B_2K_{3,3}$
with five generators $d_1$, $d_2$, $d_3$, $d_4$, $B_2(1,1)$ and one
relator $B_2(1,1)d_3^{-1}B_2(1,1)d_2d_3d_1^{-1}d_2^{-1}
d_4d_2d_1d_2^{-1}d_4^{-1}$. The abelianizatoin of the relator is
$(B_2(1,1))^2$. Thus $H_1(B_2K_{3,3})$ has a 2-torsion even though
every graph braid group is torsion free.

However, there are two difficulties in a straightforward extension
of this argument to show that the first homology group of any
non-planar graph braid group has a torsion. One is that it is
quickly getting out of hand to obtain a presentation of the
$n$-braid group of $K_5$ or $K_{3,3}$ as the braid index $n$
increases. And the other is that both $K_5$ or $K_{3,3}$ have no
vertices of valency 1 and so there is no room to accommodate extra
punctures when we construct an embedding to obtain corresponding
results for larger braid indices. Instead of computing the whole $H_1(B_n\Gamma)$, we will find a critical 1-cell that produce a nontrivial torsion homology class. Lemma~\ref{lem:casec} is useful to show it is in fact a 2-torsion. And the homomorphism $\pi_*:H_1(B_n\Gamma)\to\mathbb Z_2(=H_1(S_n))$ discussed after Example~\ref{ex:S_0} is used to check that it is nontrivial.

Let a graph $\Gamma$ be ordered by a planar embedding of its maximal tree $T$  satisfying that both ends of every deleted edge have valency two in
$\Gamma$. Suppose the vertex $i$ in $0_n$ is regarded as a puncture labeled by $n-i$ for $0\le i\le n-1$.
If neither ends of a deleted edge $d$ are the base vertex 0, $\pi(d)$ is the identity and so $\pi_*([d])=0$. If one of ends of $d$ is the base vertex 0, $\pi(d)=(1,n,n-1,\cdots,2)$ and so $\pi_*([d])\equiv n-1\pmod 2$. For a critical 1-cell $A_k(\vec a)$, $\pi(A_k(\vec a))=(|\vec a|,|\vec a|-1,\cdots,|\vec a|_k)$ and so $$\pi_*([A_k(\vec a)]) \equiv |\vec a|-|\vec a|_k \pmod 2$$
where $|\vec a|_k=a_k+a_{k+1}+\cdots+a_m$ for $\vec a=(a_1,\cdots,a_m)$.

\begin{lem}\label{lem:K_3,3}
For the complete bipartite graph $K_{3,3}$ of 6 vertices and for
$n\ge 2$, $H_1(B_n K_{3,3})$ has a torsion.
\end{lem}

\begin{proof}
\begin{figure}[ht]
\psfrag{*}{\small0}
\psfrag{A}{\small$A$}
\psfrag{B}{\small$B$}
\psfrag{C}{\small$C$}
\psfrag{D}{\small$D$}
\psfrag{E}{\small$E$}
\psfrag{F}{\small$F$}
\psfrag{e1}{\small$d_1$}
\psfrag{e2}{\small$d_2$}
\psfrag{e3}{\small$d_3$}
\psfrag{e4}{\small$d_4$}
\centering
\includegraphics[height=3.5cm]{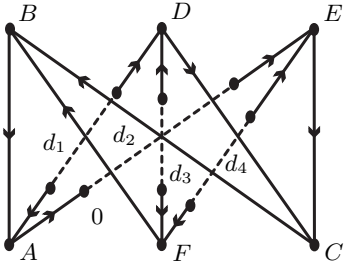}
\caption{Maximal tree and order of $K_{3,3}$} \label{fig25}
\end{figure}

If we choose a maximal tree of $K_{3,3}$ and its planar embedding as
in Figure~\ref{fig25}, then the complex $UD_n K_{3,3}$ has critical
1-cells $B_2(1,1)$, $C_2(1,1)$, and critical 2-cells $d_1 \cup d_4$,
$d_2 \cup d_3$. Using Lemma~\ref{lem:casec} and the maximal tree, the boundaries of critical 2-cells are given by:
$$\widetilde\partial(d_1\cup d_4)=B_2(1,1)-C_2(1,1),
\quad \widetilde\partial(d_2\cup d_3)=B_2(1,1)+C_2(1,1).$$ Thus
$2B_2(1,1)$ is a boundary and so it represents 0 in
$H_1(M_*(UD_nK_{3,3}))$. But $\pi_*([B_2(1,1)])=1$. So $B_2(1,1)$
represents a nontrivial torsion class in $H_1(B_nK_{3,3})$.
\end{proof}

\begin{lem}\label{lem:K_5}
For the complete graph $K_5$ of 5 vertices and for $n\ge 2$,
$H_1(B_nK_5)$ has a torsion.
\end{lem}

\begin{proof}
\begin{figure}[ht]
\psfrag{*}{\small0}
\psfrag{A}{\small$A$}
\psfrag{B}{\small$B$}
\psfrag{C}{\small$C$}
\psfrag{D}{\small$D$}
\psfrag{E}{\small$E$}
\psfrag{e1}{\small$d_1$}
\psfrag{e2}{\small$d_2$}
\psfrag{e3}{\small$d_3$}
\psfrag{e4}{\small$d_4$}
\psfrag{e5}{\small$d_5$}
\psfrag{e6}{\small$d_6$}
\centering
\includegraphics[height=4cm]{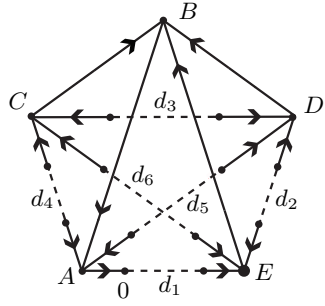}
\caption{Maximal tree and order of $K_{5}$} \label{fig26}
\end{figure}

If we choose a maximal tree of $K_5$ and a planar embedding as in
Figure~\ref{fig26}, then $UD_nK_5$ has critical 1-cells
$B_2(1,1,0)$, $B_3(1,0,1)$, $B_3(0,1,1)$, and critical 2-cells $d_1
\cup d_3$, $d_2\cup d_4$, $d_5\cup d_6$. Using Lemma~\ref{lem:casec} and the maximal tree, the boundaries of critical
2-cells are given by $\widetilde\partial_2(d_1 \cup
d_3)=B_3(1,0,1)-B_3(0,1,1)$, $\widetilde\partial_2(d_5 \cup
d_6)=B_2(1,1,0)+B_3(0,1,1)$, $\widetilde\partial_2(d_2 \cup
d_4)=B_3(1,0,1)-B_2(1,1,0)$. Thus $2B_2(1,1,0)$ is a boundary and
represents the trivial class in $H_1(M_*(UD_nK_5))$. But
$\pi_*([B_2(1,1,0)])=1$ and so it represents a torsion class. So
$H_1(B_n K_{3,3})$ has a torsion.
\end{proof}

\begin{lem}
For any non-planar graph $\Gamma$ and for $n\ge 2$, $H_1(B_n\Gamma)$
has a torsion and so $B_n\Gamma$ is not a right-angled Artin group.
\end{lem}

\begin{proof}
Since every non-planar graph contains $K_5$ or $K_{3,3}$, there is
an embedding $i:K\to\Gamma$ where $K$ is $K_5$ or $K_{3,3}$. We
choose a maximal tree $T_K$ and its planar embedding as in
Figure~\ref{fig25} or~\ref{fig26}, and choose a maximal tree $T$ of
$\Gamma$ satisfying the following conditions:
\begin{itemize}
    \item[(1)] Both ends of each deleted edge have valency two in
    $\Gamma$;
    \item[(2)] The base vertex of $K$ maps to the base vertex on
    $\Gamma$ under $i$;
    \item[(3)] For an edge of $K$, $i(e)$ is a deleted edge of
    $\Gamma$ if and only if $e$ is a deleted edge of $K$;
    \item[(4)] $i$ is order preserving.
\end{itemize}

We note that the conditions (3) and (4) automatically hold if $T$
contains $T_K$ as a subgraph in the plane. Under these conditions,
it is not hard to see that the embedding $i$ preserves critical
cells and boundaries of critical cells, that is, $i$ induces a chain
map on Morse chain complexes. Thus it induces a homomorphism
$i_*:H_*(M_*(UD_n K))\to H_*(M_*(UD_n \Gamma))$.

Let $[B_2(\vec a)]$ be a torsion class in  $H_1(M_*(UD_n K))$ as
given in Lemma~\ref{lem:K_3,3} or Lemma~\ref{lem:K_5}. Then
$\pi_*(i_* ([B_2(\vec a)]))=1$ and so $i_*([B_2(\vec a)])$ is a
torsion class in $H_1(M_*(UD_n\Gamma))$.
\end{proof}

\section{Conjectures}
We finish the article with few observations and conjectures.

An $n$-\emph{nucleus} is a minimal graph whose $n$-braid group is
not a right-angled Artin group, that is, the $n$-braid group of an
$n$-nucleus is not a right-angled Artin group and the $n$-braid
group of any graph contained in an $n$-nucleus is a right-angled
Artin group. Our main result says that for $n\ge5$, a graph $\Gamma$
contains no $n$-nuclei if and only if $B_n\Gamma$ is a right-angled
Artin group, and that there are two $n$-nuclei $T_0$ and $S_0$. It
is natural to expect that the corresponding statement works for all
braid indices and so we first propose

\begin{con}
For any braid index $n\ge2$, a graph $\Gamma$ contains no $n$-nuclei
if and only if $B_n\Gamma$ is a right-angled Artin group.
\end{con}

Then we propose a complete list of $4$-nuclei and 3-nuclei as
follows:

\begin{con} There are four
4-nuclei as given in Figure~\ref{fig27}, and ten 3-nuclei as given
in Figure~\ref{fig28}.
\end{con}

\begin{figure}[ht]
\subfigure[]
{\includegraphics[height=1.7cm]{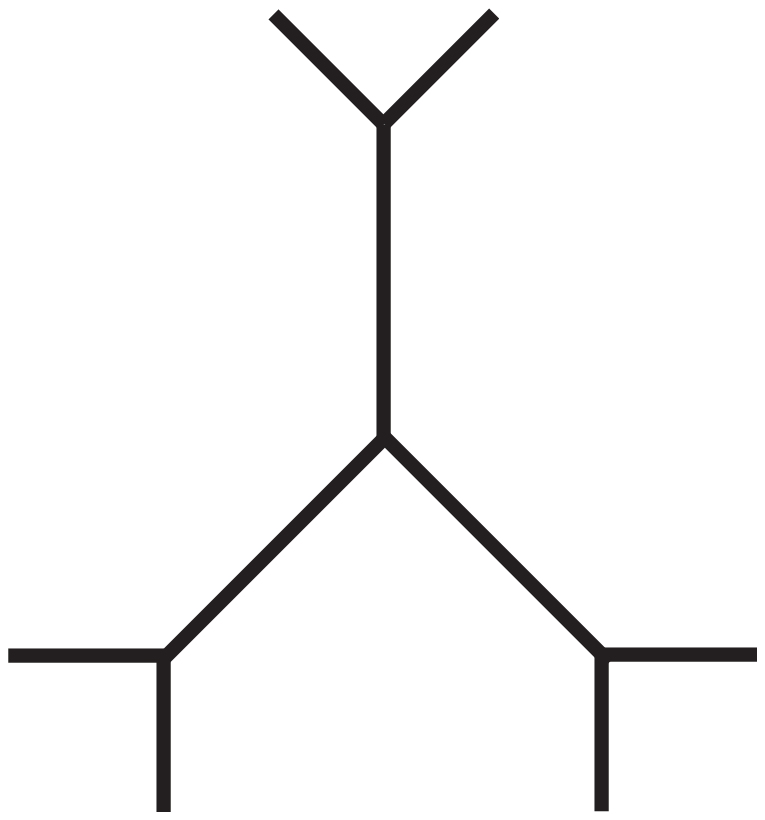}}\qquad
\subfigure[]
{\includegraphics[height=1.7cm]{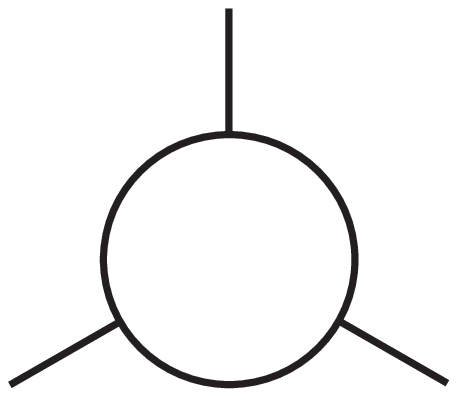}}\qquad
\subfigure[]
{\includegraphics[height=1.2cm]{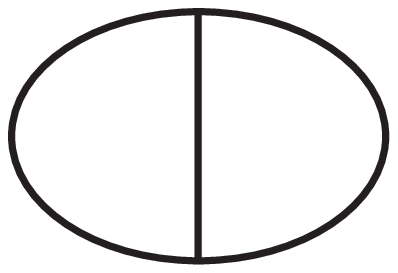}}\qquad
\subfigure[]
{\includegraphics[height=1cm]{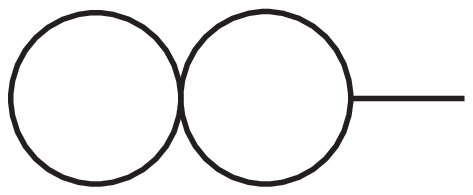}}
\caption{4-Nuclei}
\label{fig27}
\end{figure}

\begin{figure}[ht]
\subfigure[]
{\includegraphics[height=1.2cm]{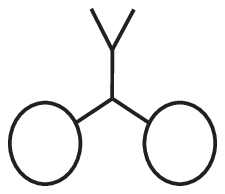}}\qquad
\subfigure[]
{\includegraphics[height=1cm]{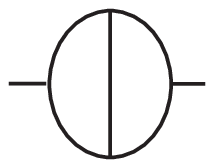}}\qquad
\subfigure[]
{\includegraphics[height=1cm]{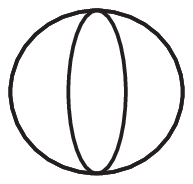}}\qquad
\subfigure[]
{\includegraphics[height=0.8cm]{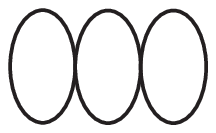}}\qquad
\subfigure[]
{\includegraphics[height=1cm]{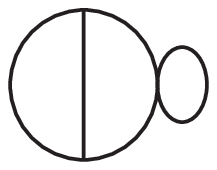}}\\
\subfigure[]
{\includegraphics[height=1cm]{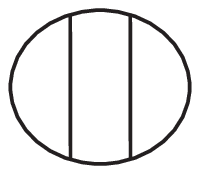}}\qquad
\subfigure[]
{\includegraphics[height=1cm]{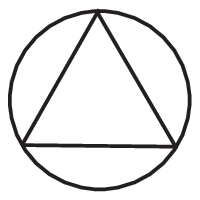}}\qquad
\subfigure[]
{\includegraphics[height=1.2cm]{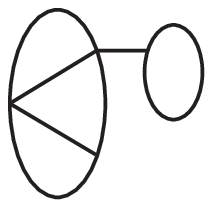}}\qquad
\subfigure[]
{\includegraphics[height=0.8cm]{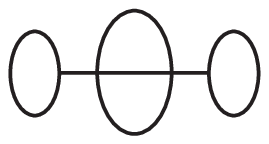}}\qquad
\subfigure[]
{\includegraphics[height=0.8cm]{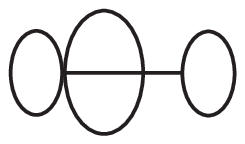}}
\caption{3-Nuclei}
\label{fig28}
\end{figure}

We found more than sixty 2-nuclei as given in Figure~\ref{fig29}.
Due to their plethora, only representatives from each type of
2-nuclei are listed and each type contains as many graphs as the
number in the parenthesis.

\begin{figure}[ht]
\subfigure[Non-planar (2)]
{\includegraphics[height=1.2cm]{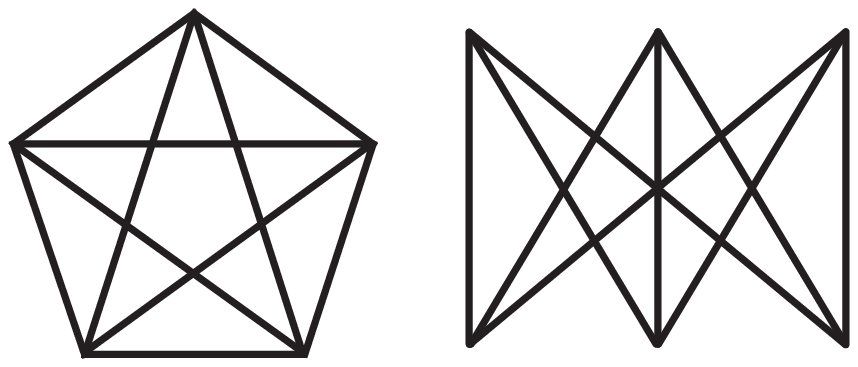}}\qquad
\subfigure[Type 1 (5)]
{\includegraphics[height=1.2cm]{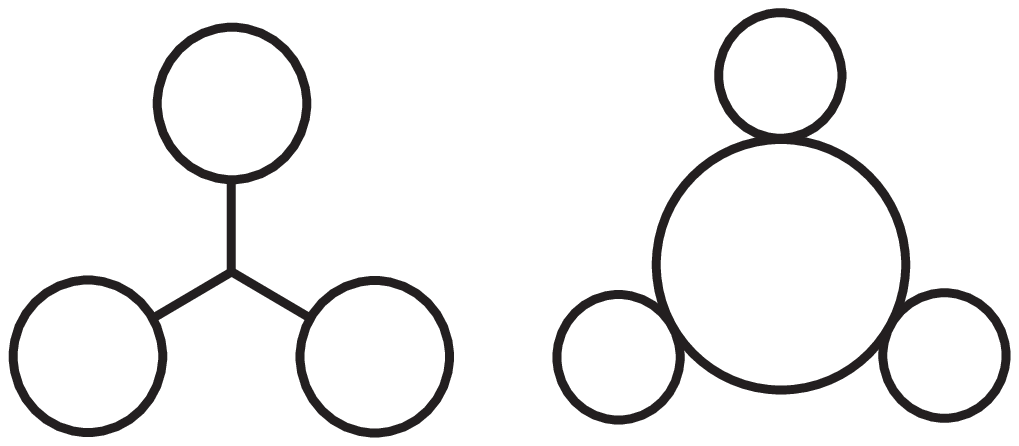}}\qquad
\subfigure[Type 2 (7)]
{\includegraphics[height=1.2cm]{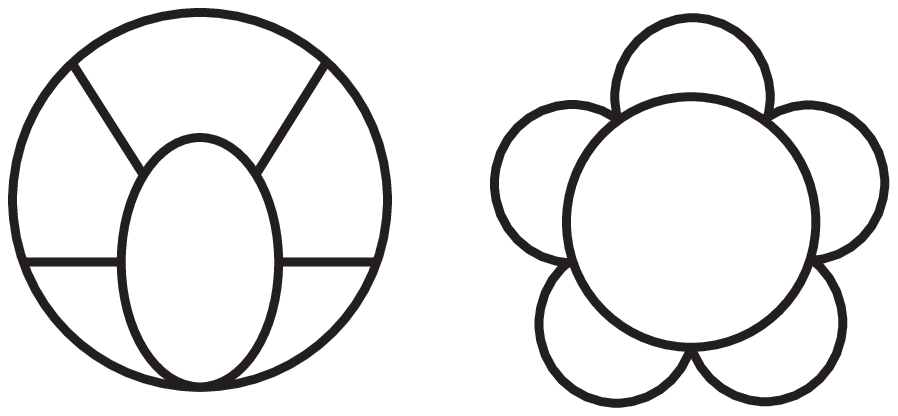}}\qquad
\subfigure[Type 3 (17)]
{\includegraphics[height=1.2cm]{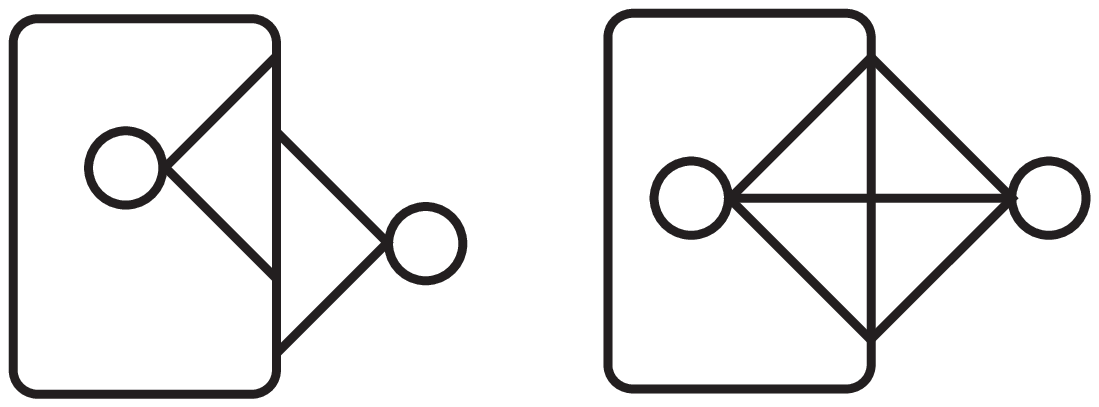}}\qquad
\subfigure[Type 4 ($\ge$30)]
{\includegraphics[height=1.2cm]{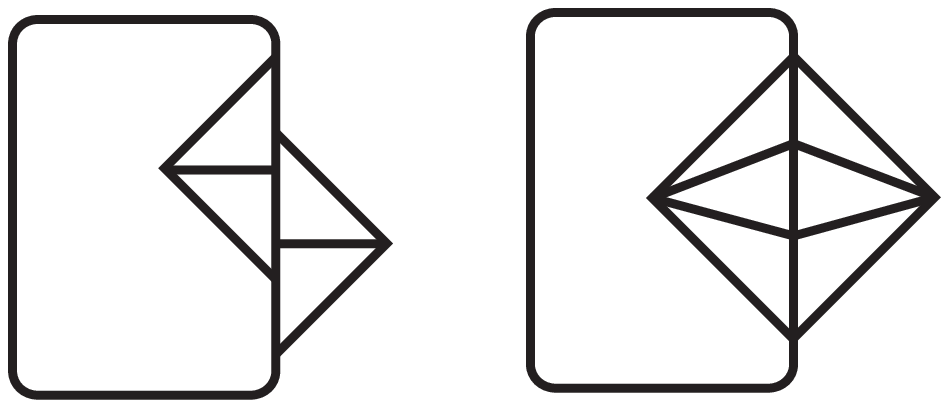}}
\caption{2-Nuclei by types.}
\label{fig29}
\end{figure}

Given a planar graph, the presentation of its $n$-braid group
obtained by choosing a maximal tree as explained in
\S\ref{ss22:morse} seems relatively nice in the sense that there is
a systematic application of Tietze transformations to make it a
commutator-ralated group. We state this as a conjecture:

\begin{con}
Let $\Gamma$ be a planar graph and $n\ge 2$. Then $B_n\Gamma$ is a
commutator-ralated group.
\end{con}

In \S\ref{ss44:nonplanar}, we show that if a given graph is
non-planar, the first homology of its braid group has a 2-torsion.
The converse is a corollary of the previous conjecture and so we
propose

\begin{con}
For $n\ge 2$, a graph $\Gamma$ is planar if and only if
$H_1(B_n\Gamma)$ is torsion-free.
\end{con}

We can prove the above conjecture for braid index $n=2$ by using
Lemma~\ref{lem:casec} as follows.

\begin{thm}
A graph $\Gamma$ is planar if and only if $H_1(B_2\Gamma)$ is
torsion-free.
\end{thm}
\begin{proof}
It only remains to prove that if a graph is planar then
$H_1(B_2\Gamma)$ is torsion-free. We first assume that each path between two vertices of valency $\neq 2$ or each simple loop in $\Gamma$ passes through exactly 3 edges. We choose the maximal tree $T$ of
$\Gamma$ and give an order on vertices so that for each deleted edge $d$, both $\iota(d)$ and $\tau(d)$ are of valency 2 and $\iota(d)\wedge\tau(d)$ is always the nearest vertex of valency $\ge3$ from $\tau(d)$ unless $\tau(d)=0$.
This can be achieved by the following steps:
\begin{itemize}
\item [(I)] Choose an planar embedding of $\Gamma$ and choose a vertex of valency 1 as a base vertex 0 if there is. Let $T=\Gamma$. If there are no vertices of valency 1, choose a vertex of valency 2 as a base vertex 0 and delete the edge one of whose end vertices is 0 and let $T$ be the rest of $\Gamma$. Go to (II).
\item [(II)] Take a regular neighborhood $R$ of $T$. As traveling the outmost component of $\partial R$ clockwise from the vertex numbered the last, number unnumbered vertices of $T$ until either coming back to 0 or numbering a vertex of valency 2 that belongs to a circuit in $T$. If the former is the case, we are done. If the latter is the case, delete
    the edge incident to the vertex in front and let $T$ be the rest. Repeat (II).
\end{itemize}

\begin{figure}[ht]
\psfrag{0}{\small0}
\psfrag{1}{\small1}
\psfrag{2}{\small2}
\psfrag{3}{\small3}
\psfrag{4}{\small4}
\psfrag{5}{\small5}
\psfrag{6}{\small6}
\psfrag{7}{\small7}
\psfrag{8}{\small8}
\psfrag{9}{\small9}
\psfrag{10}{\small10}
\psfrag{11}{\small11}
\psfrag{12}{\small12}
\psfrag{13}{\small13}
\psfrag{14}{\small14}
\psfrag{15}{\small15}
\centering
\includegraphics[scale=.3]{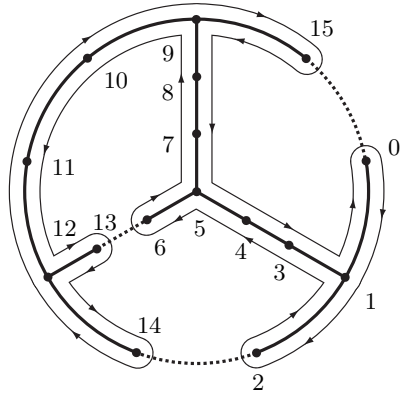}
\caption{Numbering a maximal tree $T$ of $K_4$}
\end{figure}

Then it is clear that we have the
property that
\begin{itemize}
\item[$(\ast)$] for each deleted edge $d$, there are no vertices of valency $\ge3$ between $\tau(d)$ and $\iota(d)\wedge\tau(d)$ in $T$.
\end{itemize}

A 2-cell cannot contain any non-order-respecting edge in $UD_2\Gamma$ and so there
is only one kind of critical 2-cells of the form $d\cup d'$ for deleted edges $d,d'$.

To classify these critical 2-cells $c$ and to compute their images under
$\widetilde\partial$, we consider the smallest subtree $T_c$ of the maximal tree $T$ that contains
four ends of $d$ and $d'$ together with the base vertex.
By the condition $(\ast)$, there are eight possibilities for
$T_c$ given by Figure~\ref{fig31}(a)-(h) if none of four ends is the base vertex. And there are three possibilities for $T_c$ given by Figure~\ref{fig31}(i)-(k) if one of four ends is the base vertex. We may consider each possible
$T_c$ together with the formulae for $\widetilde\partial_2(d\cup
d')$ in Lemma~\ref{lem:casec}.
\begin{figure}[ht]
\psfrag{0}{\small0}
\subfigure[]
{\includegraphics[height=.8cm]{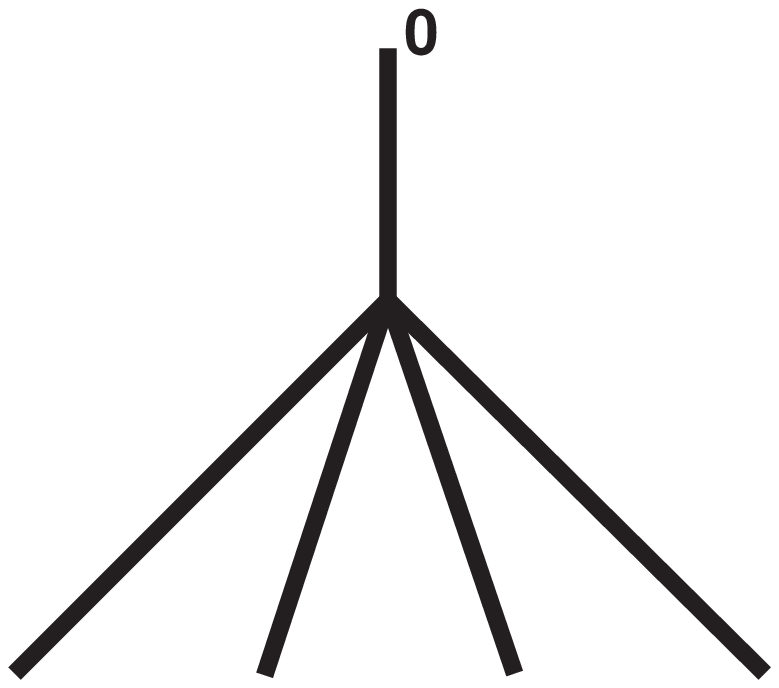}}\quad
\subfigure[]
{\includegraphics[height=.8cm]{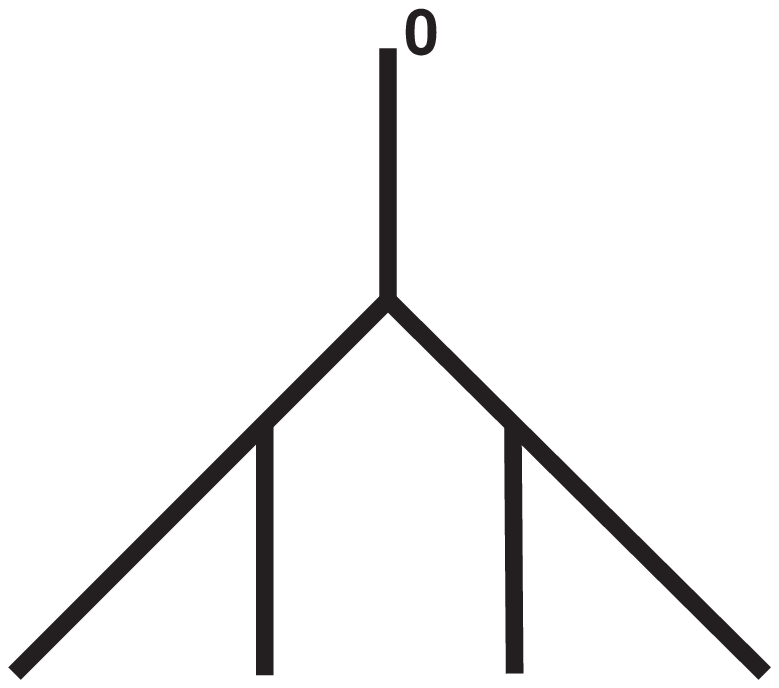}}\quad
\subfigure[]
{\includegraphics[height=.8cm]{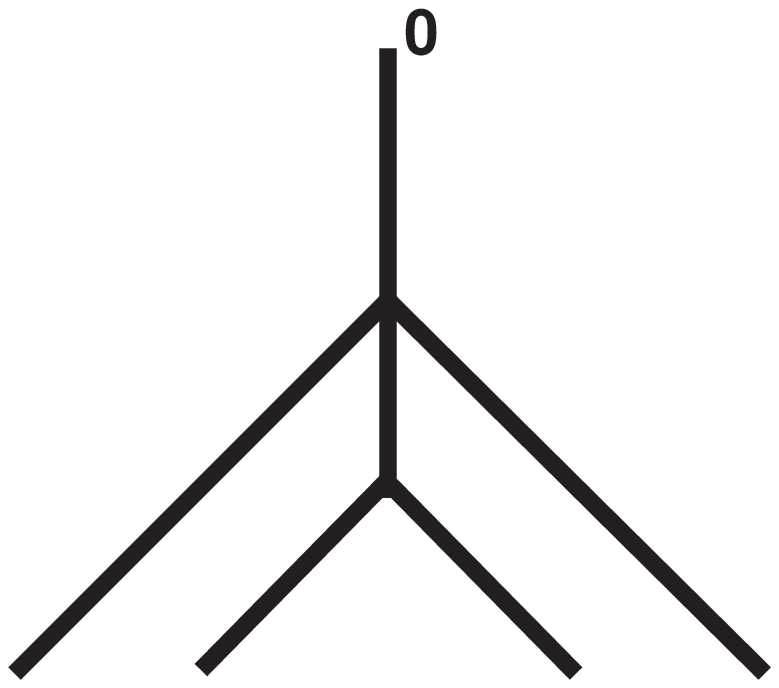}}\quad
\subfigure[]
{\includegraphics[height=.8cm]{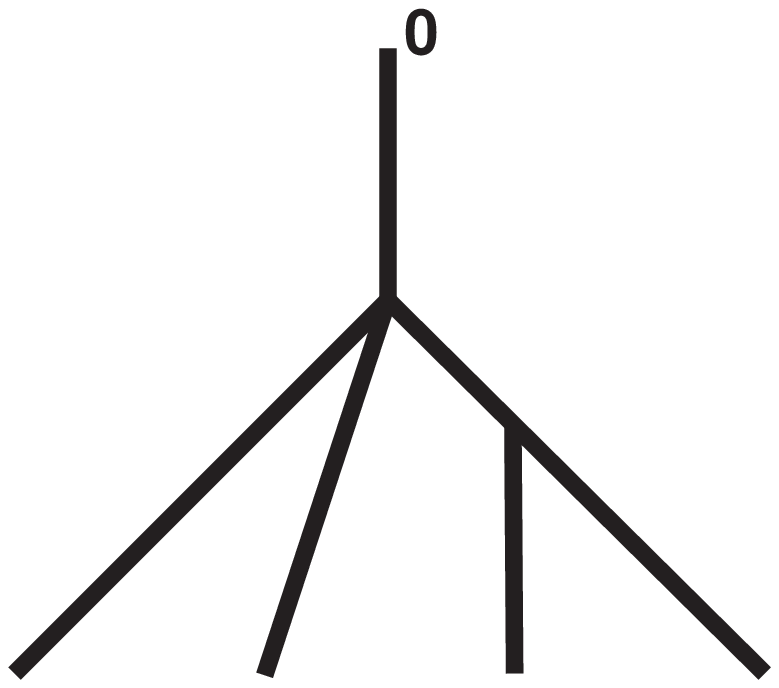}}\quad
\subfigure[]
{\includegraphics[height=.8cm]{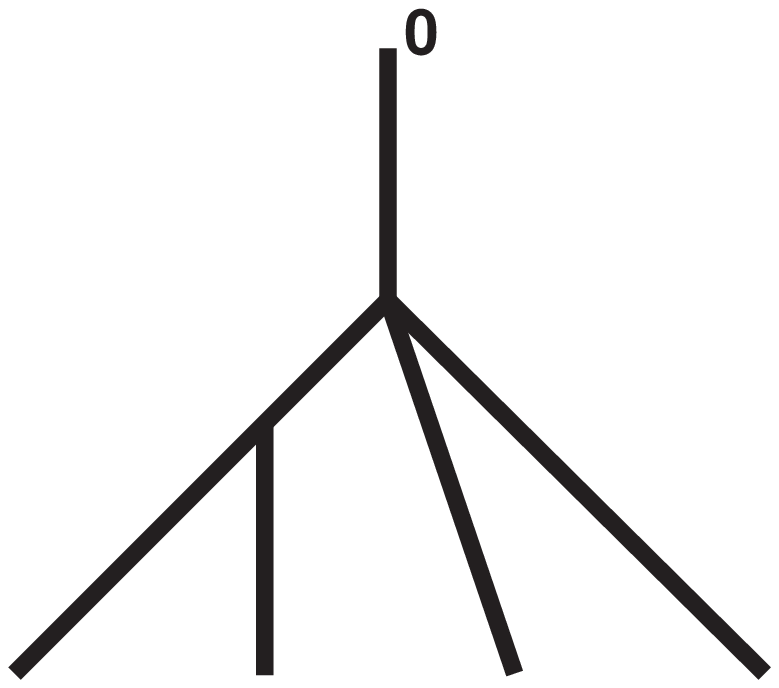}}\quad
\subfigure[]
{\includegraphics[height=.8cm]{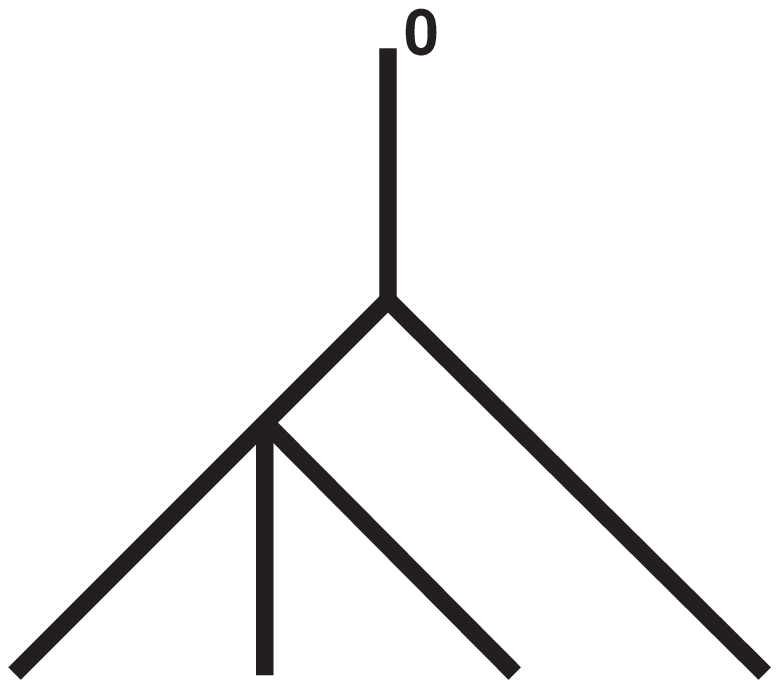}}\\
\subfigure[]
{\includegraphics[height=.8cm]{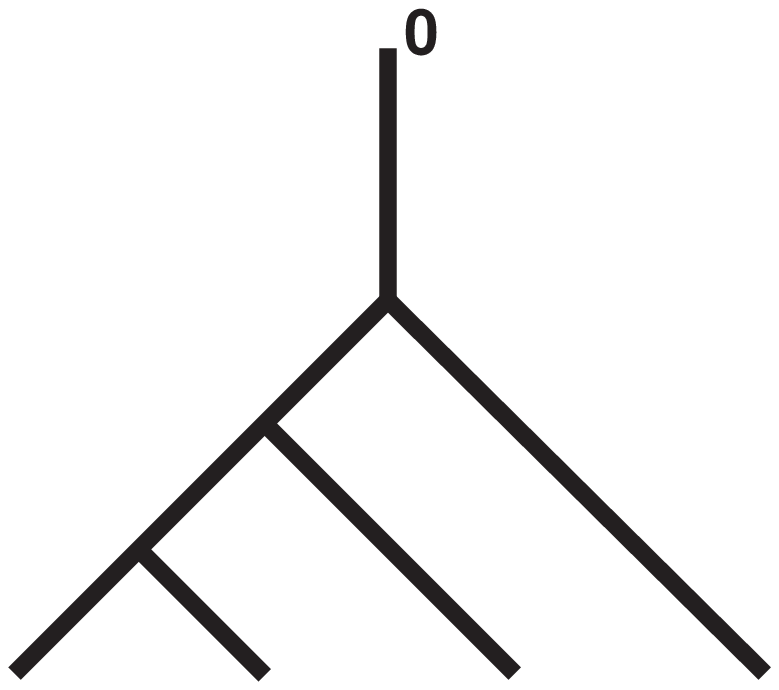}}\quad
\subfigure[]
{\includegraphics[height=.8cm]{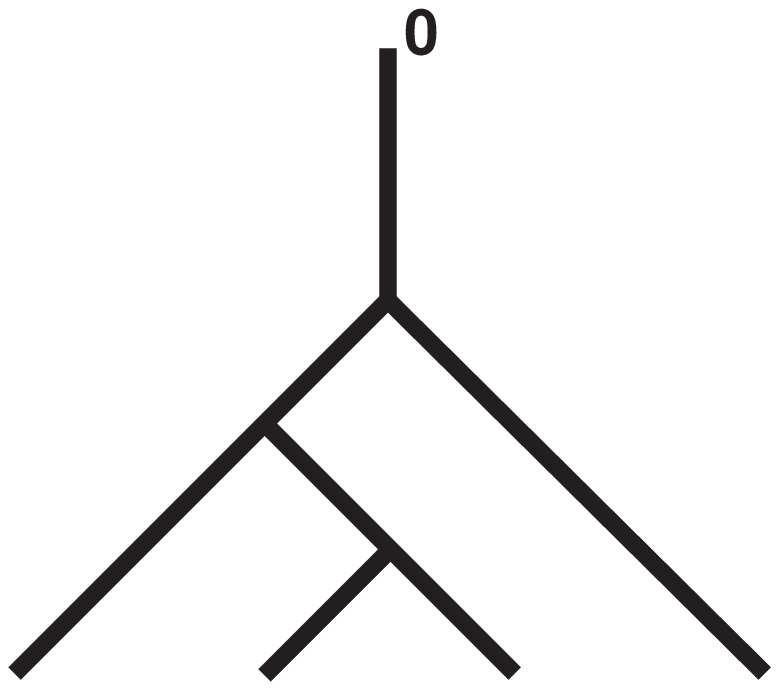}}\quad
\subfigure[]
{\includegraphics[height=.8cm]{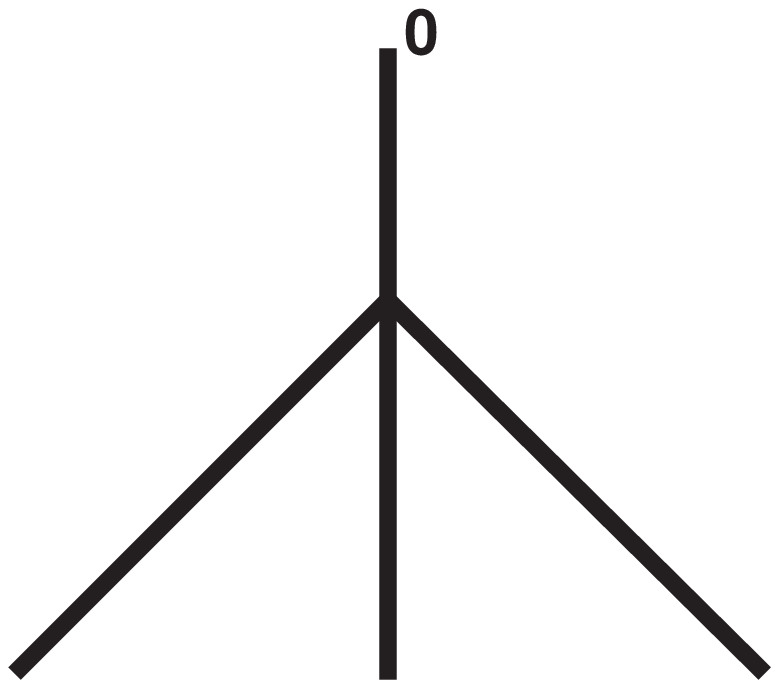}}\quad
\subfigure[]
{\includegraphics[height=.8cm]{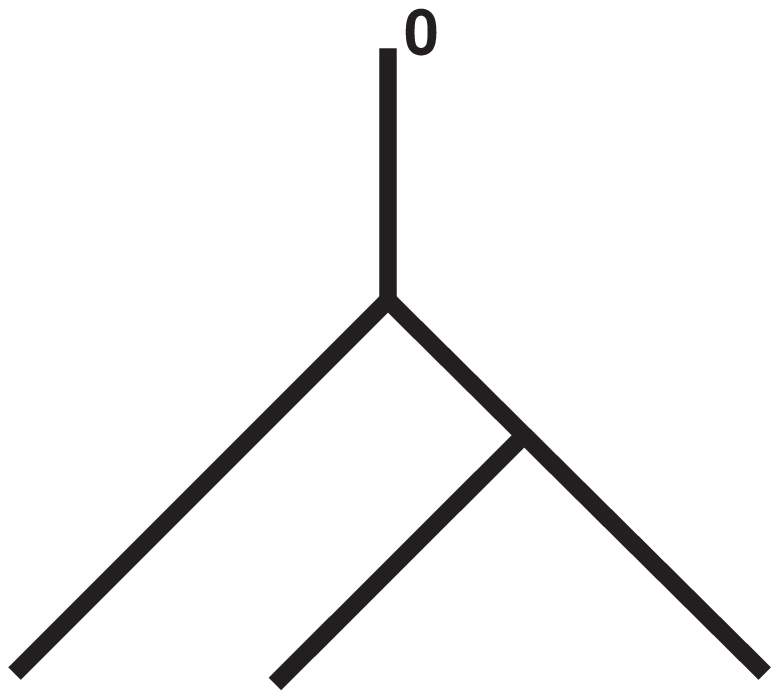}}\quad
\subfigure[]
{\includegraphics[height=.8cm]{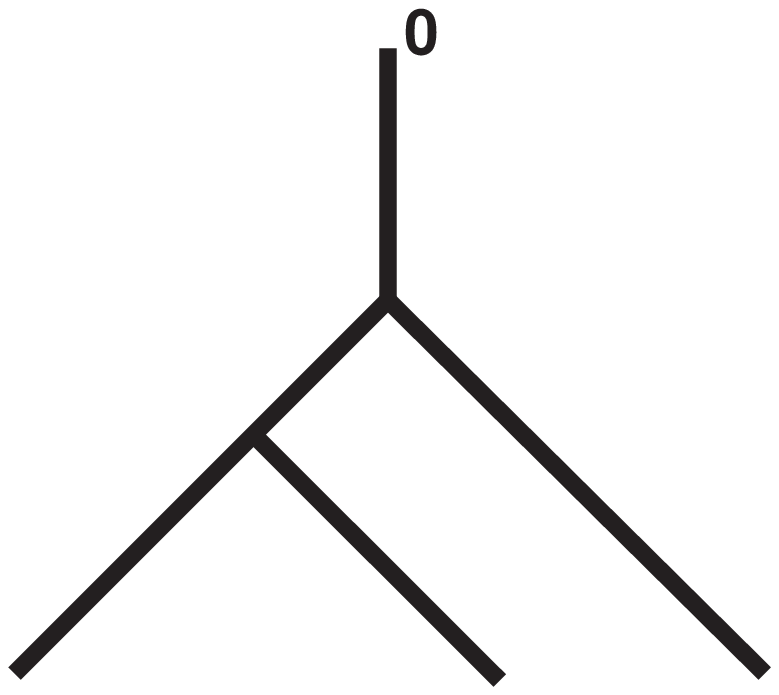}}
\caption{Possible $T_c$}
\label{fig31}
\end{figure}
Note that the order of $T_c$ is given as in Step I of
\S\ref{ss22:morse}.

Since $\Gamma$ is planar, we do not have the case (3) of Lemma~\ref{lem:casec}. The critical 2-cell satisfying
Lemma~\ref{lem:casec}$(\ell)$ for $\ell=1,2$ and whose four end
vertices form the subtree of $T$ given in
Figure~\ref{fig31}$(\alpha)$ for $\alpha=a,\cdots,n$ is denoted by
$c_{\ell\alpha}$. For example,
\begin{align*}
            \widetilde\partial(c_{1a}) =
             &\:A_{g(A,\tau(d'))}(\vec\delta_{g(A,\tau(d'))}+ \vec\delta_{g(A,\tau(d))})
             - A_{g(A,\tau(d'))}(\vec\delta_{g(A,\tau(d'))} + \vec\delta_{g(A,\iota(d))})\\
             &- A_{g(A,\iota(d'))}(\vec\delta_{g(A,\iota(d'))} + \vec\delta_{g(A,\tau(d))})
             + A_{g(A,\iota(d'))}(\vec\delta_{g(A,\iota(d'))} + \vec\delta_{g(A,\iota(d))})
        \end{align*}
where $A=\iota(d)\wedge\tau(d)=\iota(d')\wedge\tau(d')$. It is
easy to see that every critical 2-cell except the
five kinds critical 2-cells $c_{1a}$, $c_{1c}$, $c_{2a}$, $c_{2d}$,
and $c_{2e}$ gives either the trivial relation or a relation which is the difference between two
critical 1-cells. By the condition $(\ast)$, $c_{1c}$ and $c_{2d}$ cannot occur. For each of the other three cases, its boundary contains
$B_{g(B,\tau(d'))}(\vec\delta_{g(B,\tau(d'))}+ \vec\delta_{g(B,\tau(d))})$
as a summand which does not appear in the boundary of any other
critical 2-cells where $B=\tau(d)\wedge\tau(d')$.

Consequently, we have a presentation matrix of $H_1(B_2\Gamma)$ whose rows satisfy one of
the following:
\begin{enumerate}
\item[(i)] consisting of all zeros;
\item[(ii)] consisting of zeros except a pair of entries 1 and -1;
\item[(iii)] containing 1 that is the only nonzero entry in the column it belongs to.
\end{enumerate}
Then it is easy to see via elementary row and column operations that this implies that $H_1(B_2\Gamma)$ is torsion free.
\end{proof}

\end{document}